\PassOptionsToPackage{table}{xcolor}
\documentclass[12pt]{amsart}

\usepackage{mystyle}

\pgfplotsset{compat=1.18}
\begin{document}

\title{A dynamic model of congestion}

\author[H. A. Chang-Lara]{H\'ector A. Chang-Lara}
\address{Department of Mathematics, CIMAT, Guanajuato, Mexico}
\email{hector.chang@cimat.mx}

\author[S. D. Zapeta-Tzul]{Sergio D. Zapeta-Tzul}
\address{Department of Mathematics, University of Minnesota, Twin Cities, USA}
\email{zapet001@umn.edu}

\begin{abstract}
We revisit the classic problem of determining optimal routes in a graph for transporting two given distributions defined on its nodes, originally studied by Wardrop and Beckmann in the 1950s. The global congestion profile at any given time defines a dynamic metric on the graph, for which the routes must be geodesics. Our first contribution is the introduction of a dynamic version of the Beckmann problem, for which we derive the corresponding discrete partial differential equations governing the evolution of the system. These equations enable us to estimate the size of the support of the edge flow. Finally, we present some numerical simulations to illustrate the behavior of efficient equilibria in a dynamic setting with non-local interactions.
\end{abstract}
\subjclass{90B20, 49K99}
\keywords{Wardrop equilibria, traffic congestion, Beckmann problem}

\maketitle


\section{Introduction}

We address the problem of transporting two given distributions, defined over the nodes of a given graph, within a finite and discrete time horizon. The primary objective is to model the behavior of rational agents who aim to minimize their individual travel costs. The congestion profile at any given moment determines the cost of traversing each edge of the graph. This notion, originally formalized by Wardrop in the 1950s and published in \cite{wardrop1952road}, is known as \textit{Wardrop equilibria}. Additionally, the model aims to satisfy an efficiency assumption by minimizing the expected distance for the given \textit{transport plan}, as measured by a \textit{Kantorovich functional}. All these concepts will be carefully revisited in the preliminary section of this article, immediately after this introduction.

Once we establish the precise definitions of the problem outlined in the previous paragraph, two main challenges become apparent about how to compute the equilibria in practice. First, this is a problem rooted in game theory, with a much weaker existence theory as compared with classical optimization. Each agent seeks to minimize their own travel cost, but their collective choices influence the metric used to measure those costs. Second, the natural variable of the model—probability distributions over the set of paths on the graph—resides in a space of large dimension, usually out of reach for numerical implementations.

Following Wardrop's work, Beckmann, McGuire, and Winsten presented in \cite{beckmann1955studies} an equivalent characterization of the \textit{long-term congestion problem} in terms of an optimization problem over edge flows—functions defined on the edges of the graph—finding a way around the two main challenges in the previous paragraph. This problem, now known as the \textit{Beckmann problem}, is recognized as an important tool in the theory of optimal transport. For a detailed discussion on this connection, we recommend Chapter 4 in the book by Santambrogio \cite{MR3409718}. For the reader's convenience, most of our notation and concepts align with those established in this reference.

Section 4.4.1 of Santambrogio's book discusses the \textit{static} version of the problem relevant to this article, also referred to as the long-term problem. In the aforementioned treatment, agents constantly move with a given flow on the graph. In our \textit{dynamic} perspective, as time progresses, the agents adjust their routes to complete the transportation task under equilibrium and efficiency assumptions. Moreover, we also allow for non-local interactions, meaning that the cost in a given edge is not only determined by the edge flow in such edge, but also by the global edge flow. These non-local assumption becomes relevant in practice if we want to model the influence of flows in opposite directions over the same edge; or the intersection of two roads, as the one illustrated in Figure \ref{fig:intersection}.

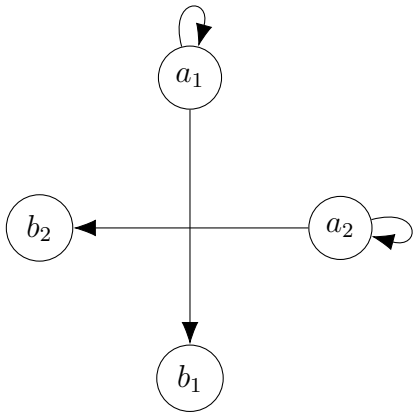
\begin{figure}
    \centering
\begin{tikzpicture}[every loop/.style={}]
  \node[circle, draw] (1) at (0,2) {$a_1$};
  \node[circle, draw] (2) at (2,0) {$a_2$};
  \node[circle, draw] (3) at (0,-2) {$b_1$};
  \node[circle, draw] (4) at (-2,0) {$b_2$};

  \draw[-{Latex[length=3mm]}] (1) -- (3);
  \draw[-{Latex[length=3mm]}] (2) -- (4);

  \path[-{Latex[length=3mm]}] (1) edge [loop above] node {} ();
  \path[-{Latex[length=3mm]}] (2) edge [loop right] node {} ();
\end{tikzpicture}
    \caption{A graph modelling the intersection of two roads. In order to describe the congestion effects of this intersection, one should consider that the cost of crossing the edge $(a_1,b_1)$ should depend on the flows over the edges $(a_1,b_1)$ and $(a_2,b_2)$.}
    \label{fig:intersection}
\end{figure}

One of the main contribution of this work is to introduce a dynamic version of the Beckmann problem and derive the corresponding (discrete) partial differential equations (PDEs) that govern the equilibria. These equations consist of a conservation law or divergence constrains, and what we identify as a \textit{constitutive relation} that results from the critical equation of the Beckmann functional with a corresponding Lagrange multiplier. Details are given in Section \ref{sec:cons_rel}, while a connection with very degenerate and non-local elliptic equation is discussed in Section \ref{sec:ellip_eqns}.

To reach this PDE formulation of the problem we provide a detailed discussion of the construction in a very general setup in Section \ref{sec:prelim} and \ref{sec:BP}. While much of this may be already available in the literature, with perhaps only minor modifications from our part, we include it in this work for the sake of completeness and pedagogical clarity. A comprehensive review of the literature can be found in \cite{correa2011wardrop}, see also \cite{MR2870230} for the continuous counterpart.

The main idea behind the dynamic formulation is to extend the model into a phase space that includes the time dimension. This approach is discussed in Section \ref{sec:dm}. The challenge of the dynamic problem arises from the fact that extending the long-term problem does not preserve the structure of a long-term problem in the extended graph. We present two ways to formulate this model as a Beckmann-type optimization problem and derive the corresponding set of equations. The first approach involves generalizing the notion of the Beckmann problem, for which we devote the Section \ref{sec:BP}. The second strategy requires introducing a further extension of the graph, along with a corresponding long-term problem that turns out to be equivalent to the dynamic model, as discussed in Section \ref{sec:aux_cons}.

As an application of the general theory developed in Section \ref{sec:BP}, we are able to estimate the size of the support of an edge flow minimizing the Beckmann functional, see Theorem \ref{thm:fin_prop} and Corollary \ref{cor:fin_prop} for the static case, and Theorem \ref{thm:fin_prop2} for the dynamic one. From the PDE perspective, this result guarantees the existence of \textit{free boundaries} in the solutions of our equations. Moreover, we provide in Theorem \ref{thm:dyn_ext} and Corollary \ref{cor:5}, a criterion that allows to check whenever a solution of the dynamic Beckmann problem is allowed to be extended in time by zero or not.

To the best of our knowledge, this article may be the first one to establish bounds on the support of the edge flow. The method of proof is inspired from techniques in optimal transport that we detail in the Section \ref{sec:beckeq_to_wareq}, \ref{sec:bound_spt_EF}, and \ref{sec:sym}. In the Section \ref{sec:beckeq_to_wareq}, we show that paths supported in the positivity set of any solution of the constitutive relation must be geodesics. We use this result in Section \ref{sec:bound_spt_EF} to give an estimate on the size of the support of the flow in terms of the support of its divergence. In Section \ref{sec:sym}, we find necessary and sufficient geometric conditions for the edge flow to satisfy the constitutive relation.

To illustrate the theory, we analyze several examples all along the paper. In particular, Section \ref{sec:dyn_ex2} examines one of the simplest scenarios: the intersection of two roads, as shown in Figure \ref{fig:intersection}. The task is to transport an initial amount of masses $m_1$ and $m_2$ from $a_1$ and $a_2$ respectively to $b_1$ and $b_2$ within $T$ time steps. A surprising finding from our analysis is that the equations governing the edge flow are equivalent to a system of \textit{discrete obstacle problems}. We provide as well the results of some numerical experiments that illustrate the behavior of the solutions (see Figure \ref{fig:2roads}).

\subsection{Related work}

The algorithmic and computational treatment of discrete congestion problems is a matured discipline, with numerous contributions from game theory, operation research, and optimization. Some books on the subject include \cite{yang2005mathematical,lawphongpanich2006mathematical,sandholm2010population,ni2015traffic}. In contrast, the continuous counterpart of the theory had to wait for the development of an appropriate analytical framework, which was provided by the theory of optimal transport in the 1990s (see \cite{MR1100809,MR1738163}) and mean field games in the 2000s (see \cite{1272542,MR2346927,4303232} and \cite{MR2269875,MR2271747,MR2295621}).

Some of the first works treating the continuous and stationary congestion models are those by Carlier, Jimenez and Santambrogio \cite{MR2407018}, followed by \cite{MR2525222} for the short-term problem, and \cite{MR2651987} for the long-term problem. A discrete-to-continuous limit was established by Baillon and Carlier in \cite{MR2928377}, and the augmented Lagrangian numerical implementation by Benamou and Carlier is presented in \cite{MR3395203}. In the dynamic setting we find some more recent developments such as the work by Gangbo, Li, and Mou \cite{MR4039140}, Graber, Mészáros, Silva, and Tonon \cite{MR3960798}, Mazanti and Santambrogio \cite{MR3986796}, and Cabrera \cite{cabrera2022optimaltransportationprincipleinteracting}. Interesting connections with free boundary problems, such as the Hele-Shaw flow, are discussed in the survey \cite{MR3883985}. 

Similar to the present work, all of the articles mentioned in the previous paragraph share the common feature of formulating a PDE that governs equilibrium configurations. The regularity theory for non-local and very degenerate equations, in both stationary and dynamic contexts, is a growing field of research (see for instance \cite{MR3133426,MR3296496,MR3705371,MR4566688}). We hope this contribution helps to further disseminate the connections between these non-local and very degenerate equations and encourages additional analysis within the community.

A distinguishing feature of our model, compared to the dynamic models recently studied in \cite{MR4039140,MR3960798,cabrera2022optimaltransportationprincipleinteracting}, is that in our case, the corresponding continuity equation does not transport $\mu$ to $\nu$ exactly at the end of the given time interval, say $[0,T]$. In the continuous setting, we find that in \cite{MR3960798,cabrera2022optimaltransportationprincipleinteracting} the continuity equation for the density $\r\colon \R^n\times[0,T]\to [0,\8)$, being transported by the vector field $\mathbf v\colon \R^n\times[0,T]\to\R^n$ is given by
\[
\p_t\rho + \operatorname{div}(\rho\mathbf v) = 0, \qquad \rho|_{t=0}=\mu, \qquad \rho|_{t=T}=\nu.
\]
This is also the continuity equation that one considers for the Benamou-Brenier approach to optimal transport \cite{MR1738163,MR3395203}, or \cite{MR4039140} in the discrete setting.

In our case, the analogue situation would be instead
\[
\p_t\rho + \operatorname{div}(\rho \mathbf v) \leq 0, \qquad \rho|_{t=0}=\mu, \qquad \rho|_{t=T}=0, \qquad \int_0^T \operatorname{div}(\rho \mathbf v)dt = \mu-\nu.
\]
(Actually, $\rho|_{t=T}=0$ can be deduced from the other conditions and $\r\geq0$ at every time). In other words, the vector field $\mathbf v$ takes the initial distribution $\mu$ to $\nu$ but is allowed to progressively finish this job at any intermediate time. This is rigorously explained for the discrete setting in Section \ref{sec:dm} through the construction of the set of transport plans $\Pi(\mu,\nu)^T$ over the extended graph.

In both cases, the pair $(\r,\mathbf v)$ (or $(\rho,\rho \mathbf v)$) serves as the variable of an optimization problem. In our model, by progressively removing the density that has reached its target, the cost distinguishes two densities: $\r$, representing the portion still in transit, and
\[
\eta(t) := \m - \rho(t) - \int_0^t \operatorname{div}(\rho\mathbf v)ds,
\]
representing the portion that has completed the transportation task and complements the conservation law of $\r$
\[
\p_t \eta = - \p_t\r - \operatorname{div}(\rho \mathbf v) \geq 0.
\]

The work by Gangbo, Li, and Mou \cite{MR4039140} is closely related to ours, as it is also formulated on a graph (over continuous time). A key distinction is that their model considers not only the flows along the edges at each moment in time but also the mass at the nodes. In \cite{MR4039140}, the cost associated with any given edge is determined by the mass at the endpoints of that edge. An intriguing question arises regarding how these models could be related at the discrete level, in a manner similar to the connections between the Benamou-Brenier problem, congestion, and optimal transport, which are known in the continuous models.

The work by Mazanti and Santambrogio \cite{MR3986796} can be seen as the continuous analogue of a particular case in our article. Their goal is to model a population attempting to escape a given domain through its boundary. The resulting equations can be seen as a non-local analogue of a dynamic eikonal equation. In our approach, we avoid the delicate aspects of analysis in infinite-dimensional spaces, allowing us to present more general models and features that, in the discrete setup, include those in \cite{MR3986796}. Our first reference to this connection appears in Remark \ref{rmk:MR3986796}.

The work by Brasco and Petrache \cite{Brasco2014} focuses on extending classical equivalences between total variation minimization, Kantorovich duality, and path-based transports to broader functional spaces (namely, the dual of Sobolev spaces). Our work introduces a dynamic Beckmann problem on graphs, where the agents' routing choices under time-dependent congestion are governed by Wardrop equilibria. Both contributions emphasize the deep connections between optimization, duality, and PDE-like structures; Brasco and Petrache by generalizing continuous models and our work by proving similar dualities for a dynamic discrete model.

\textbf{Acknowledgement:} We would like to thank Ryan Hynd, Edgard Pimentel, and Daniel Hernández for their helpful feedback on S. Zapeta-Tzul's thesis. S. Zapeta-Tzul was supported by the CONAHCyT-MEXICO scholarship 797093. H. Chang-Lara was supported by the CONAHCyT-MEXICO grant A1-S-48577.

\section{Preliminaries}\label{sec:prelim}

\subsection{Stochastic geodesic problem}

The standard (deterministic) geodesic problem consists of computing the distance and paths of minimum length between two points on a given graph. Its stochastic counterpart posses a similar question between probability distributions on the nodes. An efficient solution of the geodesic problem aims to find probability distributions on the set of geodesics that also minimizes the Kantorovich functional. In this section we give precise notions of these problems.

A directed graph is a pair $G=(N,E)$, where $N$ is the set of nodes and $E\ss N\times N$ is the set of edges that connect these nodes. For an edge $e=(x,y) \in N\times N$, we denote $e^- := x$, $e^+ := y$, and $-e = (y,x)$. We say that $G$ is undirected or symmetric if and only if $E = -E := \{-e \in N\times N\ | \ e\in E\}$.

\textbf{We will always assume in this work that $N$ is a finite set.}

A path on $G$ is a finite ordered list of consecutive nodes $\w = (x_0,\ldots,x_\ell)$, such that $e_i = (x_{i-1},x_i) \in E$ for $i\in\{1,\ldots,\ell\}$. We denote by $\w^- := x_0$ and $\w^+ := x_\ell$ the starting and final nodes of the path. A path with with $(\ell+1)$ nodes, allowing repetitions, is said to have length $\ell$. This definition allows the trivial paths of the form $\w=(x)$ with zero length. For a path of length $\ell\geq 1$, we can also refer it by listing its consecutive edges $\w = (e_1,\ldots,e_\ell)$.

We denote by $\operatorname{Path}(G)$ the set of all paths in $G$, and for any $x,y\in N$
\[
\operatorname{Path}_{xy}(G) := \{\w\in\operatorname{Path}(G)\ |\ \w^-=x, \w^+=y\}.
\]

A loop is a path of length $\ell\geq 1$ such that $\w_+=\w_-$. A simple path is one that does not repeat any node, in other words it does not contain a loop. We denote by $\operatorname{SPath}(G)$ the set of simple paths, and $\operatorname{SPath}_{xy}(G) = \operatorname{Path}_{xy}(G)\cap \operatorname{SPath}(G)$. Notice that set of paths is in general infinitely countable meanwhile the set of simple paths is finite.

We say that a node $x \in \w$ if $x$ appears in the list of nodes of $\w$. Similarly, we say that the edge $e \in \w$, if $e$ appears in the list of edges of $\w$. Given $A\ss E$, we say that $\w\ss A$ if for every $e\in \w$, one has that $e\in A$.

Given $\xi\colon E\to[0,\8)$, we define the length $L_\xi:\operatorname{Path}(G) \to [0,\8)$ by
\[
L_\xi(\w) := \sum_{e\in \w} \xi(e).
\]
If $e\in E$ appears multiple times in $\w$, then $\xi(e)$ appears with the same multiplicity in the sum above. In the case that $\w = (x)$ is a trivial path, $L_\xi(\w)=0$ by default. We may also refer to $\xi$ as a metric on the graph $G$.

For a pair of nodes $x,y\in N$, the (directed) distance $d_\xi(x,y)$ between them is defined as the minimum of the lengths of all paths from $x$ to $y$:
\[
d_\xi(x,y) := \min\{ L_\xi(\w) \ | \ \w \in \operatorname{Path}_{xy}(G)\}.
\]
A path $\w$ that realizes the distance between its end points is called a geodesic. We denote the sets of geodesics as $\operatorname{Geod}(G,\xi)$. The set of geodesics between two given nodes $x,y\in N$ is denoted by $\operatorname{Geod}_{xy}(G,\xi) := \operatorname{Geod}(G,\xi)\cap \operatorname{Path}_{xy}(G)$.

If $\operatorname{Path}_{xy}(G)=\emptyset$ we just say that $d_\xi(x,y)=+\8$. Given that the graphs in consideration are finite, we have that $\operatorname{Geod}_{xy}(G,\xi)=\emptyset$ is equivalent to $\operatorname{Path}_{xy}(G)=\emptyset$.

A well studied problem consists on computing $\operatorname{Geod}_{mn}(G,\xi)$ between two given nodes $m,n\in N$. A standard approach consists on solving the dynamic programming equation (usually implemented by iterative methods)
\[
\begin{cases}
\displaystyle u(x) = \min_{e^-=x} \{\xi(e)+u(e^+)\} \text{ if } x\neq n,\\
u(n) = 0.
\end{cases}
\]
As a result we obtain the distance function to the target node $n$ as $u(x) = d_{\xi}(x,n)$. The geodesics can then be computed from $u$ using the Pontryagin principle: $\w \in \operatorname{Path}_{mn}(G)$ is a geodesic if and only if it satisfies that:
\begin{itemize}
    \item $\w^- = m$, $\w^+=n$,
    \item For all $e\in \w$, it holds that $u(e^-) = \xi(e)+u(e^+)$.
\end{itemize}


Given two probability distributions\footnote{Given a countable set $X$ we let
\[
\mathcal P(X) := \{p\colon X\to[0,1] \ | \ \textstyle \sum_{x\in X} p(x) = 1\}. 
\]
These define all the probabilities over the sigma-algebra of all the subsets of $X$.} over the set of nodes, $\mu,\nu\in \mathcal P(N)$, a natural question would be to find a \textit{path profile} $q\in \mathcal P(\operatorname{Path}(G))$ that \textit{connects $\mu$ to $\nu$}, and is \textit{supported on geodesics}. This is a generalization of the previous case which arises when\footnote{Given $E\ss X$, we denote the indicator function of $E$ by $\mathbbm 1_E\colon X\to \R$ such that
\[
\mathbbm 1_E(x) = \begin{cases}
    1 \text{ if } x\in E,\\
    0 \text{ otherwise}.
\end{cases}
\]
If $E=\{x_0\}$ we just let $\mathbbm 1_{x_0} := \mathbbm 1_{\{x_0\}}$, and if $E=X$ we let $\mathbbm 1 := \mathbbm 1_X$.} $\m = \mathbbm 1_m$ and $\nu = \mathbbm 1_n$


Given $\mu,\nu\in \mathcal P(N)$, we say that a path profile $q\in \mathcal P(\operatorname{Path}(G))$ connects $\mu$ to $\nu$ if and only if for every $x,y\in N$
\[
\mathbb P_q(\w^-=x) = \sum_{\w^-=x} q(\w) = \mu(x), \qquad \mathbb P_q(\w^+=y) = \sum_{\w^+=y} q(\w) = \nu(y).
\]

In a similar way that a geodesic satisfies $L_\xi(\w) = d_\xi(\w^-,\w^+)$, there is a simple criterion that determines if a path profile $q\in \mathcal P(\operatorname{Path}(G))$ is supported on geodesics. For it we need to introduce the \textit{transport plan}.

\begin{definition}[Transport plan]
    The transport plan $\gamma[q] \in \mathcal P(N\times N)$ of a path profile $q\in \mathcal P(\operatorname{Path}(G))$ is given by
    \[
    \gamma[q](x,y) := \mathbb P_q(\w^-=x,\w^+=y) = \sum_{\w \in \operatorname{Path}_{xy}(G)}q(\w).
    \]
\end{definition}

\begin{lemma}[Characterization of path profiles supported on geodesics]\label{lem:char}
    Given $\xi \colon E\to [0,\8)$, a path profile $q\in \mathcal P(\operatorname{Path}(G))$ satisfies $\{q>0\}\ss \operatorname{Geod}(G,\xi)$ if and only if $\mathbb E_q(L_\xi)=\E_{\gamma[q]}(d_\xi)$, or equivalently
    \begin{align}\label{eq:1}
    \sum_{\w\in\operatorname{Path}(G)}L_\xi(\w)q(\w) = \sum_{x,y\in N} d_\xi(x,y)\gamma[q](x,y).
    \end{align}
\end{lemma}

\begin{proof}
    Notice that the left-hand side of \eqref{eq:1} is always bigger or equal than the right-hand side. Indeed
    \[
    \sum_{\w\in \operatorname{Path}(G)} L_{\xi}(\w)q(\w) \geq \sum_{\w\in \operatorname{Path}(G)} d_{\xi}(\w^-,\w^+)q(\w).
    \]
    By partitioning the paths according to their initial and final points
    \[
    \sum_{\w\in \operatorname{Path}(G)} d_{\xi}(\w^-,\w^+)q(\w) = \sum_{x,y\in N} \sum_{\w \in \operatorname{Path}_{xy}(G)} d_{\xi}(x,y)q(\w) = \sum_{x,y\in N} d_{\xi}(x,y)\gamma[q](x,y).
    \]
    The rest of the proof focuses on showing the equivalence with the reverse inequality.
    
    On one hand, if $q$ is supported on geodesics, then
    \[
    \sum_{\w\in \operatorname{Path}(G)} L_{\xi}(\w)q(\w) = \sum_{\w\in \operatorname{Path}(G)} d_{\xi}(\w^-,\w^+)q(\w).
    \]
    We already noticed that the right-hand side above is just the right-hand side in \eqref{eq:1}.
    
    Now we prove the other implication by contradiction. Assume that $q$ satisfies \eqref{eq:1}, but it is not supported on geodesics. That means that there must exists some $\w\in \operatorname{Path}(G)$ such that $q(\w)> 0$ and $L_{\xi}(\w) > d_{\xi}(\w^-,\w^+)$. Given that for every $\w\in \operatorname{Path}(G)$ it always holds that $L_{\xi}(\w) \geq d_{\xi}(\w^-,\w^+)$ we must have that
    \[
    \sum_{\w\in \operatorname{Path}(G)} L_{\xi}(\w)q(\w) > \sum_{\w\in \operatorname{Path}(G)} d_{\xi}(\w^-,\w^+)q(\w) = \sum_{x,y\in N} d_{\xi}(x,y)\gamma[q](x,y),
    \]
    which contradicts \eqref{eq:1}.
\end{proof}

Given $\Gamma\ss \mathcal P(N\times N)$, a more general geodesic problem can be posed by looking for path profiles $q\in \mathcal P(\operatorname{Path}(G))$, supported on geodesics, that also satisfy the assumption $\gamma[q] \in \Gamma$. We will denote the feasible set as
\[
\mathcal Q(\Gamma) := \{q\in \mathcal P(\operatorname{Path}(G)) \ | \ \gamma[q]\in \Gamma\}.
\]

For instance, given $\m,\nu \in \mathcal P(N)$, the geodesic problem from $\mu$ to $\nu$, which is referred to as the long-term problem, arises when
\[
\Gamma = \Pi(\m,\nu) := \{\gamma\in \mathcal P(N\times N) \ | \ \pi^-[\gamma] = \m, \pi^+[\gamma]=\nu\},
\]
where $\pi^\pm: \mathcal P(N\times N)\to \mathcal P(N)$ denote the marginals
\begin{align}
    \label{eq:marg}
    \pi^-[\gamma](x) := \sum_{y\in N} \gamma(x,y), \qquad \pi^+[\gamma](y) := \sum_{x\in N} \gamma(x,y).
\end{align}
Whenever $\Gamma = \{\gamma\}$, then it is known as the short-term problem. These names have been borrowed from \cite[Chapter 4]{MR3409718}.

\begin{remark}\label{rmk:MR3986796}
    Another example can be given whenever the goal is to transport $\m \in \mathcal P(N)$ inside a region $S \ss N$, then one should consider
    \[
    \Gamma = \{\gamma \in \mathcal P(N\times N) \ | \ \pi^-[\gamma] = \mu, \{\pi^+[\gamma]>0\}\ss S\} = \bigcup_{\{\nu>0\}\ss S}\Pi(\mu,\nu).
    \]
    This related with the model presented in \cite{MR3986796}, in a discrete and stationary setting.
\end{remark}

\begin{remark}\label{rmk:capacity}
    One may also want to connect two distributions that are constrained only in their support. Given $S^\pm \ss N$, we have in mind
    \[
    \Gamma = \{\gamma \in \mathcal P(N\times N) \ | \ \{\pi^-[\gamma] > 0\} \ss S^-, \{\pi^+[\gamma]>0\}\ss S^+\} = \bigcup_{\substack{\{\mu>0\}\ss S^-\\\{\nu>0\}\ss S^+}}\Pi(\mu,\nu).
    \]
\end{remark}

For $x,y\in N$, we have that $\operatorname{Path}_{xy}(G) \neq \emptyset$ represents a connectivity condition from $x$ to $y$. A natural way to extend this notion for $\gamma \in \mathcal P(N\times N)$ would be to say that there exists a path profile $q \in\mathcal P(\operatorname{Path}(G))$ such that $\gamma[q]=\gamma$.

The equivalence between $(1)$ and $(4)$ in the following lemma is the generalization of the equivalence between $\operatorname{Path}_{xy}(G)\neq \emptyset$ and $d_\xi(x,y)<\8$ already noticed for the deterministic case.

\begin{lemma}\label{lem:connect}
    The following are equivalent for $\gamma \in \mathcal P(N\times N)$ and $\xi\colon E\to[0,\8)$:
    \begin{enumerate}
        \item $\gamma \in \gamma[\mathcal P(\operatorname{Path}(G))]$,
        \item $\operatorname{Path}_{xy}(G)\neq \emptyset$ for every $(x,y) \in\{\gamma>0\}$,
        \item $d_{\xi}(x,y) < \8$ for every $(x,y) \in\{\gamma>0\}$,
        \item $\sum_{x,y\in N} d_\xi(x,y)\gamma(x,y)<\8$.
    \end{enumerate}
\end{lemma}

\begin{proof}
    The implications $(1)\Rightarrow(2)\Leftrightarrow(3)\Leftrightarrow(4)$ are straightforward, so we focus this proof on $(2)\Rightarrow(1)$. The argument follows by induction on the size of $\{\gamma>0\}$.

    If $\gamma=\mathbbm 1_{x_0y_0}$, then by hypothesis we must have a path $\w_0 \in \operatorname{Path}_{x_0y_0}(G)\neq \emptyset$. In this case we can take $q = \mathbbm 1_{\w_0} \in \mathcal P(\operatorname{Path}(G))$ such that $\gamma[\mathbbm 1_{\w_0}] = \mathbbm 1_{x_0y_0}$.

    Consider now the case where $k:= \#\{\gamma>0\} >1$, and assume that $(1)$ holds for any transport plan supported in a set of size $(k-1)$, and under the hypothesis $(2)$.
    
    Let $x_0,y_0 \in N$ such that $\gamma_0:= \gamma(x_0,y_0) \in (0,1)$, and $\w_0 \in \operatorname{Path}_{x_0y_0}(G)$. Let then $\gamma_1 \in \mathcal P(N\times N)$ such that
    \[
    \gamma_1 := \frac{1}{1-\gamma_0}\1\gamma - \gamma_0\mathbbm 1_{x_0y_0}\2.
    \]
    Then $\{\gamma_1>0\} = \{\gamma>0\}\sm \{(x_0,y_0)\}$ such that $\#\{\gamma_1>0\}=k-1$ and $\gamma_1$ satisfies $(2)$. By the inductive hypothesis there exists $q_1 \in \mathcal P(\operatorname{Path}(G))$ such that $\gamma_1 = \gamma[q_1]$. Letting
    \[
    q := (1-\gamma_0)q_1 + \gamma_0\mathbbm 1_{\w_0} \in \mathcal P(\operatorname{Path}(G)),
    \]
    we get that $\gamma[q] = \gamma$.
\end{proof}

In a similar way as $\operatorname{Path}_{xy}(G) \neq \emptyset$ is equivalent to $\operatorname{Geod}_{xy}(G,\xi)\neq \emptyset$, we present in the next lemma a stochastic extension of this fact. The proof follows the same inductive argument as in the previous lemma with slight modifications.

\begin{lemma}[Existence of geodesic profiles]\label{lem:exis_geo}
    Given $\xi \colon E\to [0,\8)$ and $\gamma \in \gamma[\mathcal P(\operatorname{Path}(G))]$, there exists $q \in\mathcal Q(\{\gamma\})$ such that $\{q>0\}\ss \operatorname{Geod}(G,\xi)$.
\end{lemma}

\begin{proof}
    We proceed by induction on the size of the set $\{\gamma>0\}$. If $\gamma(x_0,y_0)=\mathbbm 1_{x_0y_0}$, we must have that $\operatorname{Path}_{x_0y_0}(G)\neq \emptyset$ and there must exists some $\w_0 \in \operatorname{Geod}_{x_0y_0}(G,\xi)$. In this case we can take $q=\mathbbm 1_{\w_0}$.
    
    Consider now the general case where $k:=\#\{\gamma>0\}>1$, and assume that the result holds for any transport plan in $\gamma[\mathcal P(\operatorname{Path}(G))]$, but supported in a set of size $(k-1)$.
    
    By Lemma \ref{lem:connect}, we can find $q_0 \in \mathcal Q(\{\gamma\})$. Let $x_0,y_0 \in N$ such that $\gamma_0 := \gamma(x_0,y_0) \in (0,1)$, and let $q_1 \in \mathcal P(\operatorname{Path}(G))$ such that
    \[
    q_1 := \frac{1}{1-\gamma_0} \1q_0 - \sum_{\w\in \operatorname{Path}_{x_0y_0}(G)}q_0(\w)\mathbbm 1_\w\2.
    \]
    Then
    \[
    \gamma[q_1] = \frac{1}{1-\gamma_0}\1\gamma - \gamma_0\mathbbm 1_{x_0y_0}\2,
    \]
    so that $\{\gamma[q_1]>0\} = \{\gamma>0\}\sm \{(x_0,y_0)\}$. 
    
    By the inductive hypothesis, there exists $q_2 \in  \mathcal Q(\{\gamma[q_1]\})$ supported on geodesics. Letting $\w_0 \in \operatorname{Geod}_{x_0y_0}(G,\xi)$ we finally construct
    \[
    q := (1-\gamma_0)q_2+ \gamma_0\mathbbm 1_{\w_0} \in \mathcal P(\operatorname{Path}(G)).
    \]
    It verifies that $\gamma[q] = \gamma$, and is supported on geodesics.
\end{proof}

The following observation follows by applying the previous lemma to $\xi=\mathbbm 1$. Recall that $\operatorname{SPath}(G)$ denotes the set of simple paths in $G$, i.e., paths without loops.

\begin{corollary}\label{cor:3}
    Given $\gamma \in \gamma[\mathcal P(\operatorname{Path}(G))]$, there exists $q \in\mathcal Q(\{\gamma\})$ such that $\{q>0\}\ss \operatorname{SPath}(G)$.
\end{corollary}

The stochastic geodesic problem can also be coupled with a minimization on the transport plan known as the Kantorovich problem.

\begin{definition}[Efficient profile]
    Given $\xi\colon E\to [0,\8)$, we say that a path profile $q \in \mathcal Q(\Gamma)$ is efficient with respect to $\G \ss \mathcal P(N\times N)$ if and only if
    \[
    \gamma[q] \in \argmin_{\gamma\in \Gamma} \sum_{x,y\in N} d_\xi(x,y)\gamma(x,y).
    \]
\end{definition}

Due to the characterization provided in Lemma \ref{lem:char}, we get that a path profile $q \in \mathcal Q(\Gamma)$ is supported on geodesics and efficient with respect to $\G$, if and only if
\[
\sum_{\w\in\operatorname{Path}(G)}L_\xi(\w)q(\w) = \inf_{\gamma\in \Gamma} \sum_{x,y\in N} d_\xi(x,y)\gamma(x,y).
\]

Due to Lemma \ref{lem:connect} we are able to establish the following existence result.

\begin{corollary}[Existence of efficient geodesic profiles]
\label{cor:existence}
Let $\Gamma\ss \mathcal P(N\times N)$ be closed such that $\Gamma \cap \gamma(\operatorname{Path}(G)) \neq \emptyset$. Then there exists $q^* \in \mathcal Q(\Gamma)$ such that $\{q^*>0\}\ss\operatorname{Geod}(G,\xi)$ and
\[
\gamma[q^*] \in \argmin_{\gamma\in \Gamma} \sum_{x,y\in N} d_\xi(x,y)\gamma(x,y).
\]
\end{corollary}

\begin{proof}
By the equivalence of the first and fourth items in Lemma \ref{lem:connect}, we have that the minimum of $\gamma \in \Gamma \mapsto \sum_{x,y\in N} d_\xi(x,y)\gamma(x,y)$ must be restricted to the set $\Gamma\cap \gamma[\mathcal P(\operatorname{Path}(G))]$. Given that $\mathcal P(N\times N)$ is compact, we get that $\Gamma\cap \gamma[\mathcal P(\operatorname{Path}(G))]$ is also compact, and there exists a minimizer
\[
\gamma^* \in \argmin_{\gamma\in \Gamma} \sum_{x,y\in N} d_\xi(x,y)\gamma(x,y)\ss \gamma[\mathcal P(\operatorname{Path}(G))].
\]
Given that $\gamma^* \in \gamma[\mathcal P(\operatorname{Path}(G))]$, we use Lemma \ref{lem:exis_geo} to find a profile $q^* \in \mathcal Q(\{\gamma^*\})$ supported on geodesics. This profile $q^*$ is by construction an efficient geodesic profile.
\end{proof}

\subsection{Congestion}

The phenomenon of congestion appears when the metric $\xi$ depends on the path profile via the flow that it imposes on each one of its edges.

\begin{definition}[Edge flow]
    The edge flow $i[q]\colon E\to [0,\8)$ of a path profile $q\in \mathcal P(\operatorname{Path}(G))$ is given by
\[
i[q](e) := \mathbb P_q(\w \ni e) = \sum_{\w\ni e} q(\w).
\]
\end{definition}

Starting in this section, we will use $Y^X$ to denote the set of functions $f \colon X\to Y$. In this sense, $f(x)$ is the value of the $x$-coordinate of $f$ and may also be denoted by $f_x$. Usually we will have that $Y\ss\R$, then $Y^X$ inherits the corresponding structures from the Euclidean space. We also denote by $\R_{\geq0}$ the interval $[0,\8)$.

We are given for each edge $e\in E$ a cost function $g_e\colon \R_{\geq0}^E\to [0,\8)$. We may also denote the whole set of costs by $g = (g_e) \colon \R^E_{\geq 0}\to \R^E_{\geq 0}$. The idea is that the path profile $q \in \mathcal P(\operatorname{Path}(G))$ now determines the metric $\xi\colon E\to[0,\8)$ according to
\[
\xi(e) = g_e(i[q]) \text{ for each $e\in E$}.
\]

We say that the cost is \textit{local} at some given edge $e\in E$ if and only if we have that $g_e(i) = g_e(i_e)$, meaning that $g_e$ depends exclusively on the value $i_e$ and not the whole vector $i\in \R^E_{\geq 0}$. Otherwise and in general, we say that the cost is \textit{non-local}.

We use the following notations for the length and distance depending on a given path profile
\[
L_q := L_{g(i[q])}, \qquad d_q := d_{g(i[q])}.
\]
We hope that the context will make clear this abuse of notation.

\begin{definition}[Wardrop equilibrium]\label{def:1}
Given $g\colon \R^E_{\geq 0}\to \R^E_{\geq 0}$, a path profile $q\in \mathcal P(\operatorname{Path}(G))$ is a Wardrop equilibrium if and only if
\[
\{q>0\}\ss \operatorname{Geod}(G,g(i[q])).
\]

We say that the equilibrium is efficient in $\Gamma \ss \mathcal P(N\times N)$ if and only if
\[
\gamma[q] \in \argmin_{\gamma \in \Gamma} \sum_{x,y\in N} d_{q}(x,y)\gamma(x,y).
\]
We denote the set of efficient Wardrop equilibria in $\Gamma$ as $\operatorname{WE}(G,\Gamma,g)$.
\end{definition}

As a consequence of the characterization given by Lemma \ref{lem:char}, we may use the following identities to verify whether a path profile is a Wardrop equilibrium or not.

\begin{corollary}\label{cor:1}
The path profile $q\in \mathcal P(\operatorname{Path}(G))$ is a Wardrop equilibrium if and only if
\[
\sum_{\w\in \operatorname{Path}(G)} L_{q}(\w)q(\w) = \sum_{x,y\in N} d_{q}(x,y)\gamma[q](x,y).
\]
Given $\Gamma \ss \mathcal P(N\times N)$, a path profile $q\in \mathcal Q(\Gamma)$ is an efficient equilibrium in $\Gamma$ if and only if
\[
\sum_{\w\in \operatorname{Path}(G)} L_{q}(\w)q(\w) = \inf_{\gamma\in \Gamma}\sum_{x,y\in N} d_{q}(x,y)\gamma(x,y).
\]
\end{corollary}

The existence of a Wardrop equilibrium $q \in \mathcal Q(\Gamma)$ is non-trivial, in contrast with the proof of Lemma \ref{lem:exis_geo} or Corollary \ref{cor:existence} addressing the existence of efficient geodesic profiles. The challenge is that this is a problem in game theory, rather than a classical optimization. Each geodesic could be understood as the rational choices of the agents looking to minimize the costs of their respective paths, meanwhile their collective decision imposes the metric.

In precise terms one can approach the problem from the perspective of a fixed point statement. Given $\Gamma \ss \mathcal{P}(N\times N)$ and $\xi \colon E\to [0,\8)$, let $T[\xi]$ be the set of solutions to the efficient geodesic problem with the given data:
\[
T[\xi] := \{q\in \mathcal Q(\G) \ | \ \{q>0\}\ss\operatorname{Geod}(G,\xi)\}.
\]
By Lemma \ref{lem:exis_geo} we know that $T[\xi] \neq \emptyset$ if and only if $\mathcal Q(\Gamma) \neq \emptyset$. A Wardrop equilibrium is a fixed point for the equation
\[
q \in T[g(i[q])].
\]
It is not immediate that such equilibria should exist or that the iterations of $T\circ g\circ i$ converge in an appropriated sense.

We will include in the next section a hypothesis that will allow us to reformulate the problems in terms of an optimization one, for which we have available a much more robust theory for the existence of optimal configurations. 


\subsection{The potential}

We say that the congestion model is of potential type if there exists a continuously differentiable function $H \colon \R^E_{\geq0}\to \R$ such that $\nabla H = g$, which means that for every $e\in E$ we get that the partial derivative of $H$ in the $e$-coordinate, which will be denoted by $\p_e H$, equals $g_e$. Potential models are characterized by the symmetry condition provided by Poincaré's lemma.

\begin{lemma}[Poincaré]
Let $g\colon\R^E_{\geq0}\to\R^E$ be continuously differentiable. There exists $H\colon\R^E_{\geq0}\to \R$ twice continuously differentiable satisfying $\nabla H = g$ if and only if for any pair of edges $e_1,e_2\in E$ we have the symmetry condition $\p_{e_1}g_{e_2} = \p_{e_2}g_{e_1}$.
\end{lemma}

From the modelling point of view, we may interpret the equality $\p_{e_1}g_{e_2} = \p_{e_2}g_{e_1}$ saying that for any $i \in \R^E_{\geq0}$ fixed, an infinitesimal increment of $\d i$ on the edge $e_1$ has the same effect on the cost for the edge $e_2$, as the effect of an infinitesimal increment of $\d i$ on the edge $e_2$ on the cost for the edge $e_1$.

Under the assumptions of the previous lemma we obtain that the following integral provides a potential for $g$
\[
H(i) = \int_0^1 g(it)\cdot i dt = \int_0^1 \sum_{e\in E} g_e(it)i_e dt.
\]
In the local case ($g_e(i) = g_e(i_e)$) the symmetry condition is always satisfied. Under this assumption, we have that if $G_e$ is a primitive of $g_e$, then $H(i) = \sum_{e\in E} G_e(i_e)$ is a potential for $g$.

Here is the main result of this section. Namely the equivalence of efficient Wardrop equilibria for a potential type model, with an optimization problem in $\mathcal P(\operatorname{Path}(G))$. The following theorem can be compared to its continuous analogue, which is the main result in \cite{MR2407018}, or with the Section 4.4.1 the book \cite{MR3409718}. The particular case for the short term problem $\Gamma = \{\gamma\}$ and local monotone cost is given by the Theorem 4.28 in \cite{MR3409718}.

\begin{theorem}\label{thm:equiv_wardrop}
    Let $\Gamma \ss \mathcal P(N\times N)$ be a convex set, and let $H\in C^1(\R^E_{\geq0})$ be a potential for the cost $g=\nabla H\geq0$. Then,
    \[
    \argmin_{\gamma[q]\in\Gamma} H(i[q]) \ss \operatorname{WE}(G,\Gamma,g).
    \]
    Moreover, the sets are equal if $H$ is convex.
\end{theorem}

\begin{proof}
    The proof consists on computing the variational inequality for the minimization of $H\circ i$ under the restriction $q\in \mathcal Q(\Gamma)$. For $q^* \in \mathcal P(\operatorname{Path}(G))$, we have that $q^*\in\argmin_{\gamma[q]\in\Gamma} H(i[q])$ implies that 
    \begin{align}
        \label{eq:20}
    \sum_{\w\in\operatorname{Path}(G)} L_{q^*}(\w)q^*(\w) = \min_{\gamma \in \Gamma} \sum_{x,y\in N} d_{q^*}(x,y)\gamma(x,y).
    \end{align}
    Moreover, if $H$ is convex we also have that \eqref{eq:20} implies $q^*\in\argmin_{\gamma[q]\in\Gamma} H(i[q])$.

    Once this claim is proved, we just have to notice that \eqref{eq:20} is exactly the second identity found in Corollary \ref{cor:1}. This then would settle the proof.

    Notice that \eqref{eq:20} is equivalent to the inequality
    \begin{align}
        \label{eq:2}
    \sum_{\w\in\operatorname{Path}(G)} L_{q^*}(\w)q^*(\w) \leq \min_{\gamma \in \Gamma} \sum_{x,y\in N} d_{q^*}(x,y)\gamma(x,y).
    \end{align}
    The opposite inequality is always true by definition.
    
    Assume first that $q^* \in \argmin_{\gamma[q]\in\Gamma} H(i[q])$ and $\gamma \in \Gamma$. Our goal is to show that
    \[
    \sum_{\w\in\operatorname{Path}(G)} L_{q^*}(\w)q^*(\w) \leq \sum_{x,y\in N} d_{q^*}(x,y)\gamma(x,y).
    \]
    If $\gamma \notin \gamma[\mathcal P(\operatorname{Path}(G))]$, then the right hand side is infinite due to Lemma \ref{lem:connect}. Assume then that $\gamma \in \gamma[\mathcal P(\operatorname{Path}(G))]$. By Lemma \ref{lem:exis_geo} with $\xi^* := g(i[q^*])$, there exists $q \in \mathcal Q(\{\gamma\})$ supported on geodesics of $\xi^*$.
    
    For $t\in[0,1]$ we consider the convex combination $q_t = (1-t)q^*+tq \in \mathcal Q(\Gamma)$ such that $i[q_t] = (1-t)i[q^*]+ti[q]$. Given that $H(i[q_t]) \geq H(i[q^*])$ and using the Taylor expansion of $H$ at $i[q^*]$ we obtain that
    \[
    0 \leq t\nabla H(i[q^*])\cdot(i[q]-i[q^*]) + o(t).
    \]
    Dividing by $t\in (0,1)$ and taking the limit as $t\to0^+$,
    \begin{align*}
        0 &\leq \sum_{e\in E} \p_e H(i[q^*])(i[q](e)-i[q^*](e)),\\
        &= \sum_{e\in E} g_e(i[q^*])\sum_{\w\ni e}(q(\w)-q^*(\w)),\\
        &= \sum_{\w\in \operatorname{Path}(G)} (q(\w)-q^*(\w)) L_{q^*}(\w).
    \end{align*}
    Using that $L_{q^*}(\w) = d_{q^*}(\w^-,\w^+)$ over the support of $q$, and $\gamma[q]=\gamma$
    \begin{align*}
        \sum_{\w\in \operatorname{Path}(G)} L_{q^*}(\w)q^*(\w) &\leq \sum_{\w\in \operatorname{Path}(G)}  d_{q^*}(\w^-,\w^+)q(\w) = \sum_{x,y\in N} d_{q^*}(x,y)\gamma(x,y).
    \end{align*}
    Then we obtain \eqref{eq:2} by minimizing over $\gamma\in \Gamma$.
    
    Assume now that $H$ is convex and $q^*\in \mathcal Q(\Gamma)$ satisfies the inequality \eqref{eq:2}. For any competitor $q \in \mathcal Q(\Gamma)$  we have that
    \[
    \sum_{\w\in\operatorname{Path}(G)} L_{q^*}(\w)q^*(\w) \leq \sum_{x,y\in N} d_{q^*}(x,y)\gamma[q](x,y) \leq \sum_{\w\in \operatorname{Path}(G)} L_{q^*}(\w)q(\w).
    \]
    By the convexity of $H$
    \begin{align*}
        H(i[q]) - H(i[q^*]) &\geq \nabla H(i[q^*])\cdot(i[q]-i[q^*]) = \sum_{\w\in\operatorname{Path}(G)} L_{q^*}(\w)(q(\w)-q^*(\w)) \geq 0.
    \end{align*}
    This shows that $q^*$ minimizes $H\circ i$ and concludes the proof.
\end{proof}


The advantage of having a potential is that we now reduce the problem of finding efficient equilibria (game theory), to an optimization problem.

\begin{corollary}[Existence of efficient Wardrop equilibium]\label{cor:exis_wardrop}
    Let $\Gamma \ss \mathcal P(N\times N)$ be a closed convex set such that $\mathcal Q(\Gamma) \neq \emptyset$, and let $H\in C^1(\R^E_{\geq0})$ be a potential for the cost $g=\nabla H\geq0$. Then,
    \[
    \operatorname{WE}(G,\Gamma,g)\neq \emptyset.
    \]
\end{corollary}

\begin{proof}
    By Theorem \ref{thm:equiv_wardrop} it suffices to show that $\argmin_{\gamma[q]\in \G}H(i[q])\neq \emptyset$. Let
    \[
    \mathcal S := \{q\in \mathcal Q(\Gamma) \ | \ \{q>0\}\ss\operatorname{SPath}(G)\}.
    \]
    We will show that
    \[
    \emptyset \neq \argmin_{q\in \mathcal S}H(i[q]) \ss \argmin_{q\in \mathcal Q(\G)}H(i[q]).
    \]
    
    Let $q \in \mathcal Q(\Gamma)$. By Corollary \ref{cor:3}, there exists $q'\in \mathcal Q(\{\gamma[q]\})$ such that $\{q'>0\}\ss \operatorname{SPath}(G)$. This means that $\mathcal S \neq \emptyset$. We now use that $\mathcal S \ss \mathcal P(\operatorname{SPath}(G))$, with $\operatorname{SPath}(G)$ a finite set, to get that $\mathcal S$ is a compact set. By continuity, $\argmin_{q\in \mathcal S}H(i[q])\neq \emptyset$.

    Assume now by contradiction that there exists some $q^* \in \mathcal Q(\Gamma)$ such that
    \begin{align}
        \label{eq:3}
    H(i[q^*]) < \min_{q\in \mathcal S} H(i[q]).
    \end{align}
    
    For each path $\w\in \{q^*>0\}$, we construct the simple path $S(\w) \in \operatorname{SPath}(G)$ by erasing all the loops in $\w$. Then we let $S_\#q^* \in \mathcal P(\operatorname{SPath}(G))$ be the push-forward of $q^*$ such that for any $\w \in \operatorname{SPath}(G)$
    \[
    S_\#q^*(\w) := \mathbb P_{q^*}(\{S=\w\}) = \sum_{\w' \in S^{-1}(\w)} q^*(\w').
    \]
    The profile $S_\#q^*$ can also be extended to $\mathcal P(\operatorname{Path}(G))$ by setting $q=0$ outside $\operatorname{SPath}(G)$. This results in $S_\#q^*$ being a member of $\mathcal S$ such that $i[S_\#q^*]\leq i[q^*]$. By the monotonicity of $H$, we obtain that $H(i[S_\#q^*])\leq H(i[q^*])$. Therefore we get a contradiction with \eqref{eq:3} and complete the proof.
\end{proof}

Even for this optimization problem, we still face the challenge that the configuration space, given by $\mathcal Q(\G) \cap \mathcal P(\operatorname{SPath}(G))$, has a dimension too large for any practical implementations. However, for the long-term problem with $\Gamma = \Pi(\mu,\nu)$ such that $\{\mu>0\}\cap \{\nu>0\}=\emptyset$, we can reformulate the problem as an equivalent optimization in $\R^E$, a space with drastically smaller dimension.

\section{The Beckmann problem}\label{sec:BP}

The Beckmann problem is a minimization problem posed over flows $i\in \R^E_{\geq0}$ with a restriction on the divergence. The divergence $\operatorname{div} \colon \R^E \to \R^N$ is the linear transformation given by
\[
\operatorname{div}i(x) := \sum_{e^- = x} i(e) - \sum_{e^+ = x} i(e).
\]

\begin{definition}[Beckmann problem]
    Given $\mathcal M \ss \R^N$ and $H\colon \R^E_{\geq0}\to \R$, the Beckmann problem consists of minimizing $H(i)$ under the constraint that $\operatorname{div}i \in \cM$. We denote its set of solutions by
    \begin{align*}
    \operatorname{BP}(G,\cM,H) &:= \argmin_{\operatorname{div}i\in \mathcal M} H(i).
    \end{align*}
\end{definition}

We will see shortly that the the Beckmann problem is closely related with the Wardrop problem for a particular type of feasible set of transport plans.

\begin{definition}\label{def:pi_gm}
    Given $\mathcal M \ss \R^N$, let
    \[
        \Pi(G,\mathcal M) := \{\gamma \in \mathcal P(N\times N) \ | \ \pi^-[\gamma]-\pi^+[\gamma] \in \cM\}.
    \]
\end{definition}

The construction of $\Pi(G,\mathcal M)$ generalizes the one for $\Pi(\m,\nu)$. In particular, $\Pi(\m,\nu) \ss \Pi(G,\{\mu-\nu\})$, and the sets are equal if $\{\mu>0\}\cap \{\nu>0\}=\emptyset$. This claims follows from the following general lemma.

\begin{lemma}\label{lem:pi}
    Given $\mathcal M \ss \R^N$, the following inclusion holds\footnote{Given $a\in \R$, we denote its positive and negative parts respectively as $a^+ = \max\{a,0\}$ and $a^-=\max\{-a,0\}$, such that $a=a^+-a^-$.}
    \[
        \bigcup_{f\in \cM} \Pi(f^+,f^-) \ss \Pi(G,\mathcal M). 
    \]
    Furthermore, equality holds if and only if $\cM \ss \{f\in \R^N \ | \ f^\pm \in \mathcal P(N)\}$.
\end{lemma}

\begin{proof}
    Let $f\in \cM$. From the definitions we have that $\gamma \in \Pi(f^+,f^-)$ if and only if $\pi^\pm[\gamma] = f^{\mp}$. This implies that
    \[
    \pi^-[\gamma]-\pi^+[\gamma] = f^+-f^- = f \in \cM,
    \]
    thus $\gamma \in \Pi(G,\mathcal M)$.

    Let us now assume $\cM \ss \{f\in \R^N \ | \ f^\pm \in \mathcal P(N)\}$. Given $\gamma \in \Pi(G,\cM)$ and $f := \pi^-[\gamma]-\pi^+[\gamma] \in \cM$ we get that $f^+\leq \pi^-[\gamma]$ and $f^-\leq \pi^+[\gamma]$. Since both $f^\pm$ and $\pi^\pm[\gamma]$ belong to $\mathcal P(N)$, these inequalities must actually be equalities. Therefore, $\gamma \in \Pi(f^+,f^-)$.
\end{proof}

The main ingredient to connect the Beckmann and Wardrop problems is given by the following identity.

\begin{lemma}\label{lem:pitodiv}
    For $q\in \mathcal P(\operatorname{Path}(G))$
    \[
    \operatorname{div}i[q] = \pi^-[\gamma[q]]-\pi^+[\gamma[q]].
    \]
\end{lemma}
    
\begin{proof}
    We start with the definitions of the divergence and the edge flow. For any $x\in N$
    \begin{align*}
        \operatorname{div}i[q](x) &= \sum_{e^-=x}i[q](e)-\sum_{e^+=x}i[q](e)\\
        &= \sum_{e^-=x}\sum_{\w\ni e}q(\w)-\sum_{e^+=x}\sum_{\w\ni e}q(\w)\\
        &= \sum_{e\in E} \sum_{\w\in \operatorname{Path}(G)} q(\w)(\mathbbm 1_{\{e^-=x\}}-\mathbbm 1_{\{e^+=x\}})\mathbbm 1_{\{e\in \w\}}.
    \end{align*}
    By Fubini and a cancellation identity
    \[
        \operatorname{div}i[q](x) = \sum_{\w\in \operatorname{Path}(G)} q(\w)(\mathbbm 1_{\{\w^-=x\}}-\mathbbm 1_{\{\w^+=x\}}).
    \]
    
    From the definition of the transport plan and the marginal $\pi^-$
    \[
        \sum_{\w\in \operatorname{Path}(G)} q(\w)\mathbbm 1_{\{\w^-=x\}} = \sum_{y\in N}\sum_{\w\in \operatorname{Path}_{xy}(G)} q(\w) = \sum_{y\in N} \gamma[q](x,y) = \pi^-[\gamma[q]](x).
    \]
    In a similar way for the the marginal $\pi^+$
    \[
        \sum_{\w\in \operatorname{Path}(G)} q(\w)\mathbbm 1_{\{\w^+=x\}} = \sum_{y\in N}\sum_{\w\in \operatorname{Path}_{yx}(G)} q(\w) = \sum_{y\in N} \gamma[q](y,x) = \pi^+[\gamma[q]](x).
    \]
    This concludes the desired identity.
\end{proof}


As a consequence of the previous lemma, we obtain that
\[
\gamma[q] \in \Pi(G,\cM) \qquad\Leftrightarrow\qquad \operatorname{div} i[q] \in \cM.
\]
Hence we can immediately compare the two optimization problems
\begin{align}\label{eq:min}
\inf_{\operatorname{div} i\in \cM} H(i) \leq \inf_{\gamma[q]\in \Pi(G,\mathcal M)} H(i[q]).
\end{align}

The following theorem states that both problems are indeed equivalent. A discussion of the Beckmann formulation in the local case can also be found on pages 159-161 of \cite{MR3409718}. Our statement allows a broader class of cost and gives more precise description on the relation between the sets $\operatorname{BP}(G,\mathcal M,H)$ and $\argmin_{\gamma[q]\in \Pi(G,\mathcal M)} H(i[q])$, however all these ideas can also be found in \cite{MR3409718}.

\begin{theorem}\label{thm:equiv_opt}
Let $\mathcal M \ss \{f \in \R^N \ | \ f^\pm \in \mathcal P(N)\}$ be a closed set such that $\{i \in \R^E_{\geq0} \ | \ \operatorname{div}i\in \cM\} \neq \emptyset$, and let $H\in C^1(\R^E_{\geq0})$ such that $\nabla H\geq0$. Then,
\[
\operatorname{BP}(G,\mathcal M,H) \neq \emptyset \qquad\text{ and } \qquad \argmin_{\gamma[q]\in \Pi(G,\mathcal M)} H(i[q]) \neq \emptyset,
\]
and
\begin{align}
    \label{eq:main4}
\min_{\operatorname{div} i\in \cM} H(i) = \min_{\gamma[q]\in \Pi(G,\mathcal M)} H(i[q]). 
\end{align}
Moreover,
\[
\argmin_{\gamma[q]\in \Pi(G,\mathcal M)} H(i[q]) = \bigcup_{i^*\in \operatorname{BP}(G,\mathcal M,H)} \{q^* \in \mathcal Q(\Pi(G,\mathcal M)) \ | \ i[q^*]\leq i^*\}. 
\]
\end{theorem}

The following proof relies on a technical lemma, which we will state now and prove in the next section

\begin{lemma}[Smirnov Decomposition]\label{lem:smirnov}
Let $i\in \R^E_{\geq0}$ such that
\[
\text{$\m := (\operatorname{div} i)^+$ and $\nu:=(\operatorname{div} i)^-$ belong to $\mathcal P(N)$.}
\]
Then there exists $q\in \mathcal Q(\Pi(\m,\nu))$ such that $i[q]\leq i$.
\end{lemma}

\begin{proof}[Proof of Theorem \ref{eq:main4}]
    The hypotheses guarantee that $\Gamma = \Pi(G,\cM) \ss \mathcal P(N\times N)$ is a compact set. Let us now show that $\mathcal Q(\Gamma)\neq \emptyset$. Let $i \in \R^E_{\geq0}$ such that $\operatorname{div}i \in \cM$. By the assumption on $\cM$, we have that $\mu = (\operatorname{div} i)^+, \nu = (\operatorname{div}i)^- \in \mathcal P(N)$. Then by the Smirnov decomposition there exists $q \in \mathcal P(\operatorname{Path}(G))$ such that $\gamma[q] \in \Pi(\mu,\nu) \ss \Pi(G,\cM)$.
    
    From the proof of the Corollary \ref{cor:exis_wardrop} applied to $\Gamma = \Pi(G,\cM)$, we get that
    \[
    \argmin_{\gamma[q] \in \Pi(G,\cM)} H(i[q]) \neq \emptyset.
    \]
    (Notice that this statement does not require $\Gamma$ to be convex, but only compact with $\mathcal Q(\Gamma)\neq \emptyset$).
    
    For every $q^* \in \argmin_{\gamma[q] \in \Pi(G,\cM)} H(i[q])$ one must also have that
    \[
    i[q^*] \in \operatorname{BP}(G,\mathcal M,H).
    \]
    Otherwise, there exists $i \in \R^E_{\geq0}$ with $\operatorname{div}i \in \cM$ such that $H(i)< H(i[q^*])$. Once again by the Smirnov decomposition, there exists $q \in \mathcal P(\operatorname{Path}(G))$ such that $\gamma[q] \in \Pi(G,\cM)$ and $i[q]\leq i$. Given that $\nabla H\geq 0$, we get that $H(i[q])\leq H(i)$ and the following contradiction
    \[
    H(i[q]) \leq H(i) < H(i[q^*]) = \min_{\gamma[q] \in \Pi(G,\cM)} H(i[q]),
    \]

    This previous argument not only shows that $\operatorname{BP}(G,\mathcal M,H) \neq \emptyset$, but also says that the two minima must be equal. Indeed, for $i^* \in \operatorname{BP}(G,\mathcal M,H)$ and $q\in \mathcal Q(\Pi(G,\mathcal M))$, such that $i[q]\leq i^*$, we get that
    \[
    \min_{\operatorname{div} i\in \cM} H(i) = H(i^*) \geq H(i[q]) \geq \min_{\gamma[q]\in \Pi(G,\mathcal M)} H(i[q]).
    \]
    Given that the opposite inequality also holds, as was already pointed out in \eqref{eq:min}, we obtain the desired equality. 

    Moreover, and once again by the previous argument using the Smirnov decomposition
    \begin{align*}
    \argmin_{\gamma[q]\in \Pi(G,\mathcal M)} H(i[q])
    &\ss \{q^* \in \mathcal Q(\Pi(G,\mathcal M)) \ | \ i[q^*]\in \operatorname{BP}(G,\mathcal M,H)\}\\
    &\ss \bigcup_{i^*\in \operatorname{BP}(G,\mathcal M,H)} \{q^* \in \mathcal Q(\Pi(G,\mathcal M)) \ | \ i[q^*]\leq i^*\}. 
    \end{align*}
    On the other hand, if $q \in \mathcal Q(\Pi(G,\mathcal M))$ is such that $i[q]\leq i^*$ for some $i^* \in \operatorname{BP}(G,\mathcal M,H)$, then
    \[
    H(i[q]) \leq H(i^*) = \min_{\operatorname{div} i\in \cM} H(i) = \min_{\gamma[q]\in \Pi(G,\mathcal M)} H(i[q]).
    \]
    In conclusion, $q \in \argmin_{\gamma[q]\in \Pi(G,\mathcal M)} H(i[q])$, which settles the last claim in the theorem.
\end{proof}

A natural question arises: Given $\Gamma \ss \mathcal{P}(N\times N)$, does there exists $\cM \ss \R^N$ such that $\Gamma = \Pi(G,\mathcal M)$?

\begin{lemma}\label{lem:pim}
    Let $\Gamma\ss\mathcal P(N\times N)$. There exists $\cM\in \R^N$ such that $\Gamma = \Pi(G,\mathcal M)$ if and only if
    \begin{align}\label{eq:11}
    (\gamma + \operatorname{ker}(\pi^--\pi^+))\cap \mathcal P(N\times N) \ss \Gamma \text{ for every $\gamma\in \Gamma$}.
    \end{align}
\end{lemma}

\begin{proof}
    Let $\gamma \in \Pi(G,\mathcal M)$ and
    \[
    \bar\gamma \in (\gamma + \operatorname{ker}(\pi^--\pi^+))\cap \mathcal P(N\times N).
    \]
    Then $\pi^-[\bar\gamma]-\pi^+[\bar\gamma] = \pi^-[\gamma]-\pi^+[\gamma] \in \cM$, so that $\bar\gamma \in \Pi(G,\mathcal M)$.

    Assume now that \eqref{eq:11} holds for every $\gamma \in \Gamma$ and let us show that $\Gamma=\Pi(G,\mathcal M)$ for $\cM = (\pi^--\pi^+)(\Gamma)$. Clearly $\Gamma \ss \Pi(G,\mathcal M)$ by construction. For $\bar\gamma \in \Pi(G,\mathcal M)$ we get that $\pi^-[\bar\gamma]-\pi^+[\bar\gamma] = \pi^-[\gamma]-\pi^+[\gamma]$ for some $\gamma \in \Gamma$. Then $k = \bar\gamma-\gamma \in \operatorname{ker}(\pi^--\pi^+)$ shows that
    \[
    \bar\gamma = \gamma+k \in (\gamma + \operatorname{ker}(\pi^--\pi^+))\cap \mathcal P(N\times N) \ss\Gamma,
    \]
    which concludes the proof.
\end{proof}

\subsection{The Smirnov decomposition}\label{sec:smirnov}

The proof of Lemma \ref{lem:smirnov} is constructive and can be used to find the efficient equilibria from the optimal flow. The idea is closely related with the inductive argument in the proof of the Lemma \ref{lem:connect} and the proof of the Lemma \ref{lem:exis_geo}. First we show an preliminary lemma.

\begin{lemma}\label{lem:aux}
Given $i\colon E\to [0,\8)$ and $x\in\{\operatorname{div} i > 0\}$, there exists $\w \in \bigcup_{y\in \{\operatorname{div} i < 0\}}\operatorname{Path}_{xy}(G)$ such that $i(e)>0$ for every $e\in \w$.
\end{lemma}

\begin{proof}
We proceed by induction on the size of $\{i>0\}\ss E$, the case where $\#\{i>0\}=1$ being trivial. Assume then that $k:= \#\{i>0\} > 1$ and the result holds for any edge flow $j\colon E\to [0,\8)$ with $\#\{j>0\} < k$.

For $x\in\{\operatorname{div} i>0\}$ there is at least one edge $e_0$ with $e_0^-=x\neq e_0^+$ and $i_0:=i(e_0)>0$. If $y:= e_0^+ \in \{\operatorname{div} i<0\}$ we just take the path $\w=(e_0)$ to conclude. Otherwise, we assume that $\div i(y) \geq 0$ and consider $j = i-i_0\mathbbm 1_{e_0}\in \R^E$ with $\#\{j>0\}=k-1$.

We check that
\begin{align*}
&\{\operatorname{div}j >0\} = \begin{cases}
\{\operatorname{div} i>0\}\cup\{y\} \text{ if } \operatorname{div}i(x)>i_0,\\
(\{\operatorname{div} i>0\}\cup\{y\}) \sm \{x\} \text{ if } \operatorname{div}i(x)\leq i_0,
\end{cases}\\
&\{\operatorname{div} j<0\}=\begin{cases}
\{\operatorname{div} i<0\} \text{ if } \operatorname{div}i(x)\geq i_0,\\
\{\operatorname{div} i<0\}\cup\{x\} \text{ if } \operatorname{div}i(x)<i_0.
\end{cases}    
\end{align*}

Then we use the inductive hypothesis to find a path $\w_0$ from $y\in \{\operatorname{div}j >0\}$ to $\{\operatorname{div} j<0\}$ such that $j(e) =i(e)>0$ for any $e \in \w_0$. If this path ends at $\{\operatorname{div} i<0\}$, the concatenation $\w$ that starts with $e$ and then follows $\w_0$ gives the desired path from $x$ to $\{\operatorname{div} i<0\}$.

The remaining possibility is that $\w_0^+=x$, in this case we still consider for the concatenation $\w$ as in the previous paragraph, which in this case becomes a loop. Let $m := \min_{e\in \w} i(e)$ and $j'\in \R^E$ such that
\[
j' := i- mi[\mathbbm 1_\w] = i - m\sum_{e\in\w}\mathbbm 1_{e},
\]
such that $\operatorname{div} j' = \operatorname{div} i$. Given that $\#\{j'>0\}<k$ we obtain by the inductive hypothesis a path from $x\in \{\operatorname{div} j'>0\}$ to $\{\operatorname{div} j'<0\} = \{\operatorname{div} i<0\}$ and conclude the proof.
\end{proof}

\begin{remark}
    The same argument shows that for every $y \in \{\operatorname{div} i<0\}$ there exists $\w \in \bigcup_{x\in \{\operatorname{div} i > 0\}}\operatorname{Path}_{xy}(G)$ such that $i(e)>0$ for every $e\in \w$.
\end{remark}

\begin{proof}[Proof of Lemma \ref{lem:smirnov}]
    We proceed by induction and begin by introducing a few constructions for an arbitrary function $i\colon E\to \R$. For the positive and negative parts of the divergence we denote $\m[i] := (\div i)^+$ and $\nu[i] = (\div i)^-$. For the paths in the support of the edge flow, connecting the support of the positive and negative of the divergence
    \[
    \operatorname{Path}(i) := \{\w\in \operatorname{Path}(G) \ | \ \w\in\{i>0\}, \w^- \in\{\mu[i]>0\}, \w^+ \in\{\nu[i]>0\}\}.
    \]
    Finally, we consider the edges that are used by some path in the previous set
    \[
    \operatorname{Edge}(i) := \{e\in E\ | \ e \in \w \text{ for some } \w \in \operatorname{Path}(i)\}.
    \]
    
    The induction proceeds on the size of the set
    \[
    \mathcal S(i):= \operatorname{Edge}(i) \cup \{\m[i]>0\} \cup \{\nu[i]>0\}.
    \]    
    The base case is the one where $\{\mu[i]>0\}=\{x_0\}$, $\{\nu[i]>0\}=\{y_0\}$, and $\operatorname{Edge}(i)=\{(x_0,y_0)\}$. Whenever this happens, we just let $\w_0 = (x_0,y_0)$ and $q = \mathbbm 1_{\w_0}$.
    
    Assume now that $k:=\#S(i)>3$ and the conclusion holds for any $j\colon E\to [0,\8)$ with $\#\mathcal S(j)<k$ and under the hypothesis of the lemma. Let
    \[
    m := \min\3 \min_{e\in \operatorname{Edge}(i)} i(e), \min_{x\in\{\mu[i]>0\}} \mu[i](x), \min_{y\in\{\nu[i]>0\}} \nu[i](y)\4 \in (0,1].
    \]
    
    If $m=1$ we have that necessarily $\{\mu[i]>0\}=\{x_0\}$, $\{\nu[i]>0\}=\{y_0\}$, and $i=1$ over $\operatorname{Edge}(i)$. By Lemma \ref{lem:aux} there exists $\w_0 \in \operatorname{Path}_{x_0y_0}(G) \cap \operatorname{Path}(i)$. Then once again, $q=\mathbbm 1_{\w_0}$ gives the desired profile.

    Let us assume that $m\in(0,1)$. By construction there must exists $\w_0\in \operatorname{Path}(i)$ such that \[
    m = \min\3\min_{e\in\w_0} i(e),\mu[i](\w_0^-),\nu[i](\w_0^+)\4.
    \]
    Let
    \[
    j := \frac{1}{1-m}\1i - mi[\mathbbm 1_{\w_0}]\2,
    \]
    such that
    \[
    \mu[j] = \frac{1}{1-m}\1\mu[i] - m\mathbbm 1_{\w_0^-}\2 \in \mathcal P(N), \qquad \nu[j] = \frac{1}{1-m}\1\nu[i] - m\mathbbm 1_{\w_0^+}\2 \in \mathcal P(N).
    \]
    By its construction $\# \mathcal S(j) < k$.
    
    By the inductive hypothesis there exists $q_0\in\mathcal Q(\Pi(\mu[j],\nu[j]))$ such that $i[q_0] \leq j$. Let finally $q = (1-m)q_0 + m\mathbbm 1_{\w_0}$. By the linearity of the map $i$ we get that
    \[
    i[q] = (1-m)i[q_0] + mi[\mathbbm 1_{\w_0}] \leq i.
    \]
    Also by linearity
    \[
    \gamma[q] = (1-m)\gamma[q_0] + m\gamma[\mathbbm 1_{\w_0}] = (1-m)\gamma[q_0] + m\mathbbm 1_{\w_0^-\w_0^+}.
    \]
    For the marginals we obtain
    \[
    \pi^-[\gamma[q]] = (1-m)\mu[j] + m\mathbbm 1_{\w_0^-} = \mu[i], \qquad \pi^+[\gamma[q]] = (1-m)\nu[j] + m\mathbbm 1_{\w_0^+} = \nu[i].
    \]
    Hence $q\in \mathcal Q(\Pi(\mu[i],\nu[i]))$, which concludes the proof.
\end{proof}

\subsection{The constitutive relations}\label{sec:cons_rel}

The next lemma presents the variational inequalities for the Beckmann problem in terms of the Lagrange multiplier for the divergence constrain. Such multiplier will be usually denoted by $u\in \R^N$. The discrete gradient $D:\R^N\to \R^E$, transpose to $-\operatorname{div}$, is defined by
\[
Du(e) := u(e^+)-u(e^-).
\]

\begin{lemma}\label{lem:4}
    Let for $k\in\{1,2\}$, $m_k\in \N$, $A_k\colon \R^N\to \R^{m_k}$ linear, $b_k\in\R^{m_k}$, and
    \begin{align*}
    \mathcal M &= \{f \in \R^N \ | \ A_1[f] = b_1, A_2[f] \geq b_2\}.
    \end{align*}
    Given $H\in C^1(\R^E_{\geq0})$, we have that for every $i^*\in \operatorname{BP}(G,\cM,H)$ there exist $u_k^*\in \R^{m_k}$ such that for $u^* = A_1^T [u^*_1] - A_2^T[u_2^*] \in \R^N$
    \begin{align}\label{eq:main2}
    \begin{cases}
        \min\{i^*,\nabla H(i^*) - Du^*\}=0 \text{ in } E,\\
        \min\{A_2[\operatorname{div}i^*]-b_2,u_2^*\}=0,\\
        A_1[\operatorname{div}i^*] = b_1.
    \end{cases}
    \end{align}
    Moreover, if $H$ is convex, then for any $(i,u_1,u_2) \in \R^E\times\R^{m_1}\times\R^{m_2}$ and $u=A_1^T[u_1] - A_2^T[u_2]$ that satisfies the previous equations it must hold that $i \in \operatorname{BP}(G,\cM,H)$.
\end{lemma}

We call the first equation in \eqref{eq:main2} the \textit{constitutive relation}. Notice that it can also be written in terms of $g=\nabla H$ as
\[
\min\{i,g(i)-Du\} = 0 \text{ in } E.
\]

\begin{proof}
    The proof consists on computing the Karush–Kuhn–Tucker conditions on a non necessarily convex problem with affine constrains. The transpose of the divergence is computed using the integration by parts formula in Lemma \ref{lem:5} in the appendix section: For every $u \in \R^N$ and $i\in \R^E$
    \[
    u\cdot \operatorname{div}i = \sum_{x\in N} u(x)\operatorname{div} i(x) = -\sum_{e\in E} Du(e)i(e) = - Du\cdot i.
    \]

    In this case we must consider the Lagrangian $L\colon \R^E\times\R^{m_1}\times\R^{m_2}_{\geq0}\times \R^E_{\geq0}\to \R$ given by
    \begin{align*}
    L(i;u_1,u_2,\l)
    &= H(i) + u_1\cdot A_1[\operatorname{div} i] - u_2\cdot A_2[\operatorname{div}i] - \l\cdot i\\
    &= H(i) - (D(A^T[u_1] - A_2^T[u_2]) + \l)\cdot i.
    \end{align*}
    
    By \cite[Proposition 3.3.7]{MR3587371}, we get that for $i^* \in \operatorname{BP}(G,\cM,H)$ there exists $(u_1^*,u_2^*,\l^*)\in\R^{m_1}\times\R^{m_2}_{\geq0}\times \R^E_{\geq0}$ such that
    \[
    \begin{cases}
    \nabla H(i^*) - D(A_1^T[u_1^*] - DA_2^T[u_2^*]) - \l^* = 0 \qquad \text{(stationarity)},\\
    \min\{i^*,\l^*\}=0 \qquad \text{(feasibility and complementary slackness)},\\
    \min\{A_2[\operatorname{div}i^*]-b_2,u_2^*\}=0 \qquad \text{(feasibility and complementary slackness)},\\
    A_1[\operatorname{div} i^*] = b_1 \qquad \text{(feasibility)}.
    \end{cases}
    \]
    Then we replace $u^*=A_1^T[u_1^*] - A_2^T[u_2^*]$ and $\l^* = \nabla H(i^*) - Du^*$ in the second equation to get the constitutive relation.

    Assume now that $H$ is convex and $(i,u_1,u_2)$ satisfies \eqref{eq:main2}. Let $j \in \R^E_{\geq0}$ with $\div j \in \cM$ be a competitor. By the convexity of $H$
    \[
    H(j)-H(i) \geq \nabla H(i)\cdot(j-i).
    \]
    From the constitutive relation we get that $\nabla H(i) = Du+\l$ for some $\l\in \R^E$ such that $\min\{i,\l\}=0$. By integrating by parts
    \[
    H(j)-H(i) \geq (Du+\l)\cdot(j-i) = \l\cdot(j-i) - u\cdot \div(j-i).
    \]
    If we further assume that $u=A_1^T[u_1]-A_2^T[u_2]$ with $A_1[\operatorname{div}(j-i)]=0$, we get that
    \[H(j)-H(i)\geq \l\cdot(j-i) + u_2\cdot A_2[\div(j-i)].\]
    
    Let $e\in E$ and let us recall that $\min\{i,\l\}=0$. If $i_e>j_e$ then $i_e>0$ and $\lambda_e=0$. If instead $j_e\geq i_e$ we get that both factors in $\lambda_e (j_e - i_e)$ are non-negative. In this way we get that $\l\cdot(j-i)\geq 0$.
    
    Similarly, one can also show that if $\min\{A_2[\operatorname{div}i]-b_2,u_2\}=0$ and $A_2[\operatorname{div}j]\geq b_2$ hold, then $u_2\cdot A_2[\div(j-i)]\geq 0$. In conclusion, $i$ must minimize $H$. 
\end{proof}

\subsubsection{The long-term problem}

If $\cM = \{f\}$ for some fixed $f\in \R^N$, we get that $A_1$ can be taken as the identity map in $\R^N$, $b_1=f$, and we can ignore the inequalities constrains (or just set $A_2=0$, $b_2=0$, and get some redundant equations). Notice as well that $\div i = f$ requires that $\sum_{x\in N} f(x)=0$.

The resulting equations in this case turn out to be
\begin{align}
    \label{eq:long_term}
\begin{cases}
    \min\{i,\nabla H(i)-Du\}=0 \text{ in } E,\\
    \operatorname{div}i = f \text{ in } N.
\end{cases}
\end{align}

Recall that this Beckmann problem is the one related with the long-term Wardrop problem for $\Gamma = \Pi(\mu,\nu)$ when $\{\mu>0\}\cap\{\nu>0\}=\emptyset$ and $f=\mu-\nu$.

\subsubsection{The stationary minimal-time mean field game}

Consider the scenario presented in Remark \ref{rmk:MR3986796}, with $\m \in \mathcal{P}(N)$ and $S \subseteq N \setminus \{\m > 0\}$ the given data. Let us recall that the goal is to transport $\m$ inside $S$. Then we have that the feasible set is described by affine constrains on the divergence
\begin{align}\label{eq:M}
\mathcal M
&= \{f\in \R^N \ | \ f=\mu\text{ in } N\sm S, f \leq 0 \text{ in } S\}\\
&= \{f\in \R^N \ | \ \mathbbm 1_{N\sm S}f=\mu, -\mathbbm 1_{S}f \geq 0\},\notag
\end{align}
which leads to the following equations for $u = \mathbbm 1_{N\sm S} u_1 + \mathbbm 1_{S} u_2$
\begin{align}
    \label{eq:eikonal}
    \begin{cases}
    \min\{i,\nabla H(i) - Du\}=0 \text{ in } E,\\
    \min\{-\operatorname{div}i,u_2\} = 0 \text{ in } S,\\
    \operatorname{div} i = \m \text{ in } N\sm S.
\end{cases}
\end{align}
Given that $u$ consists of the contribution of two functions with disjoint supports we can simplify the multiplier to considering just $u$ and reduce the second equation to 
\[
\min\{-\operatorname{div}i,u\} = 0 \text{ in } S.
\]

\subsubsection{The general capacity problem}

As a final example let us now consider $S^\pm \ss N$ with $S^-\cap S^+= \emptyset$. In the Remark \ref{rmk:capacity} we considered
\[
\Gamma = \{\gamma \in \mathcal P(N\times N) \ | \ \{\pi^-[\gamma] > 0\} \ss S^-, \{\pi^+[\gamma]>0\}\ss S^+\} = \bigcup_{\substack{\{\mu>0\}\ss S^-\\\{\nu>0\}\ss S^+}}\Pi(\mu,\nu).
\]
The Beckmann problem is then a generalization of the capacity of $S^-$ with respect to $N\sm S^+$, see for instance \cite[Chapter 2]{MR1411441} for the continuous analogue. In this case we also have that $\Gamma = \Pi(G,\cM)$ for
\[
\cM = \left\{f\in \R^N \ \middle|  \ f\geq 0\text{ in } S^-, f = 0 \text{ in } N\sm(S^-\cup S^+), f \leq 0 \text{ in } S^+, \sum_{x\in S^-}f(x)=1\right\}.
\]
After some simplifications, we get the equations
\begin{align}\label{eq:capacity}
\begin{cases}
    \min\{i,\nabla H(i) - Du\}=0 \text{ in } E,\\
    \min\{\operatorname{div}i,-u\} = 0 \text{ in } S^-,\\
    \operatorname{div}i = 0 \text{ in } N\sm(S^-\cup S^+),\\
    \min\{-\operatorname{div}i,u\} = 0 \text{ in } S^+,\\
    \sum_{x\in S^-}\operatorname{div} i(x) = 1.
\end{cases}
\end{align}

\begin{remark}
    It is possible to deduce the sets of equations in \eqref{eq:eikonal}, and also in \eqref{eq:capacity}, from \eqref{eq:long_term}. In the case of \eqref{eq:eikonal} we consider a new (sink) node $\W \notin N$, and the graph
    \[
    G^\W := (N^\W,E^\W), \qquad N^\W := N\cup\{\W\}, \qquad E^\W := E\cup (S\times\{\W\}).
    \]
    The Beckmann problem for $\Pi(G,\cM)$, with $\cM$ as in \eqref{eq:M}, becomes equivalent to a long-term problem in $G^\W$ for $\Pi(\mu,\mathbbm 1_{\W})$, where $\mu$ is extended by zero at $\W$. The cost function $H \colon \R^{E^\W}\to \R$ is also an extension of $H\colon \R^E\to \R$, such that for $i\in \R^{E^\W}$, we have that $H(i) = H(i')$, where $i'\in \R^E$ denote the $E$-coordinates of $i$:
    \[i'_e=i_e \text{ for every $e\in E$.}\]
    
    Hence, we get the equations
    \[
    \begin{cases}
    \min\{i,\nabla H(i) - Dv\} = 0 \text{ in } E^\W,\\
    \operatorname{div}i = \mu - \mathbbm 1_\W \text{ in } N^\W.
    \end{cases}
    \]
    Using that $\p_e H =0$ for any $e\in S\times\{\W\}$, and letting $u(x) = v(x)-v(\W)$ we deduce \eqref{eq:eikonal}.

    For the case of \eqref{eq:capacity} the idea is very similar, we have to consider not only a sink node connected to $S^+$, but also a source node connected to $S^-$. This strategy will play a central role in the Section \ref{sec:aux_cons}.
\end{remark}

\subsection{From the constitutive relation to Wardrop equilibria}\label{sec:beckeq_to_wareq}

The following lemma brings us back to the geodesic problem, offering sufficient conditions to ensure that any path in $\{i>0\}$ is necessarily a geodesic.

\begin{lemma}\label{lem:pontryagin}
    Let $(i,\xi,u)\in \R^E_{\geq0}\times\R^E_{\geq0}\times\R^N$ such that $\min\{i,\xi-Du\}=0$. Then, for each $\w\ss\{i>0\}$ we have that $\w\in\operatorname{Geod}(G,\xi)$ and
    \[
    d_\xi(\w^-,\w^+) = L_\xi(\w) = u(\w^+)-u(\w^-).
    \]
\end{lemma}

\begin{proof}
    Consider $\w \ss \{i>0\}$ and let $x:=\w^-$ and $y:=\w^+$. If $x=y$
    \[
    0 \leq \sum_{e \in \w} \xi(e) = \sum_{e \in \w} Du(e) = 0. 
    \]
    Hence $L_\xi(\w) = 0$ and $\w$ must be a geodesic and $d_\xi(\w^-,\w^+) = L_\xi(\w) = u(\w^+)-u(\w^-)=0$.
    
    Assume otherwise that $x\neq y$, and let $\bar\w \ss \operatorname{Path}_{xy}(G)$ be an arbitrary competitor for the distance. By the given hypothesis
    \[
    \begin{cases}
        \text{$\xi(e) = Du(e)$ for any $e \in \w$},\\
        \text{$\xi(e) \geq Du(e)$ for any $e \in \bar\w$}.
    \end{cases}
    \]
    Then, due to the telescopic identity
    \[
    L_\xi(\w) = \sum_{e\in\w} \xi(e) = \sum_{e \in \w} Du(e) = u(y)-u(x) = \sum_{e \in \bar \w} Du(e) \leq \sum_{e\in\bar\w} \xi(e) = L_\xi(\bar \w).
    \]
    Given that $\bar\w \in \operatorname{Path}_{xy}(G)$ was arbitrary, this means that $\w$ must be a geodesic between $x$ and $y$. Moreover, $d_\xi(x,y) = L_\xi(\w) = u(y)-u(x)$.
\end{proof}

The following corollary can be used to recover Wardrop equilibria from solutions of the constitutive relation and the construction in the Smirnov decomposition (Lemma \ref{lem:smirnov}).

\begin{corollary}\label{cor:beckeq_to_wareq}
    Let $H \in C^1(\R^E_{\geq0})$ be a potential for the cost $g = \nabla H\geq 0$. Consider $(i,u)\in \R^E_{\geq0}\times\R^N$ be such that $\min\{i,\nabla H(i)-Du\}=0$, and let $q \in \mathcal P(\operatorname{Path}(G))$ be such that $i[q]\leq i$. Then $q$ is a Wardrop equilibrium with respect to $g$.
\end{corollary}

\subsection{A bound on the support of the edge flow}\label{sec:bound_spt_EF}

Lemma \ref{lem:pontryagin} can be used to obtain an upper bound on the inner diameter of $\{i>0\}$ whenever $\xi$ is bounded away from zero. We define the inner diameter of a set $A\ss E$ as
\[
\operatorname{in-diam}_\xi(A) := \max \{ L_\xi(\w) \ | \ \w \ss A \text{ and } \w \in \operatorname{Geod}(G,\xi)\}.
\]
We also define the diameter from $A\ss N$ to $B\ss N$ as
\[
\operatorname{diam}_\xi(A,B) := \max_{\substack{x\in A\\y\in B}}d_\xi(x,y).
\]

\begin{corollary}\label{cor:4}
    Let $(i,\xi,u)\in \R^E_{\geq0}\times\R^E_{\geq0}\times\R^N$ such that $\min_{e\in E}\xi(e)>0$, and $\min\{i,\xi-Du\}=0$. Then,
    \[
    \operatorname{in-diam}_\xi(\{i>0\}) \leq \operatorname{diam}_\xi(\{\div i>0\},\{\div i<0\}).
    \]
\end{corollary}

\begin{proof}
    Let $\w \ss \{i>0\}$, by Lemma \ref{lem:pontryagin} we already know that $\w$ must be necessarily a geodesic, our goal is to show that
    \[
    L_\xi(\w) \leq \operatorname{diam}_\xi(\{\div i>0\},\{\div i<0\}).
    \]
    
    Due to the constitutive relation and the lower bound on $\xi$, $\w$ can not be a loop. Otherwise we get the contradiction
    \[
    0 = u(\w^+) - u(\w^-) = \sum_{e\in \w} Du(e) = \sum_{e\in \w} \xi(e) > 0.
    \]
    
    Now let us extend $\w$ to some path $\bar\w \supseteq \w$ such that $\bar \w\ss \{i>0\}$, $\bar\w^-\in \{\div i >0\}$ and $\bar\w^+\in \{\div i<0\}$. Then
    \[
    L_\xi(\w) \leq L_\xi(\bar\w) = d_\xi(\bar\w^-,\bar\w^+) \leq \operatorname{diam}_\xi(\{\div i>0\},\{\div i<0\}),
    \]
    which concludes the proof.
\end{proof}

\begin{remark}\label{rmk:1}
    If we let $m:=\min_{e\in E}\xi(e)>0$ and $M := \max_{e\in E}\xi(e)$, then we get that under the same assumptions as in the previous result
    \[
    \operatorname{in-diam}_{\mathbbm 1} (\{i>0\}) \leq Mm^{-1}\operatorname{diam}_{\mathbbm 1}(\{\div i>0\},\{\div i<0\}).
    \]
    Indeed
    \begin{align*}
    m\operatorname{in-diam}_{\mathbbm 1} (\{i>0\}) &\leq \operatorname{in-diam}_\xi (\{i>0\})\\
    &\leq \operatorname{diam}_\xi(\{\div i>0\},\{\div i<0\})\\
    &\leq M\operatorname{diam}_{\mathbbm 1}(\{\div i>0\},\{\div i<0\}).
    \end{align*}
\end{remark}

When we apply these results to the constitutive relation, we recover a bound on the support of the edge flow.

\begin{theorem}\label{thm:fin_prop}
    Let $H\in C^1(\R^E_{\geq0})$ be such that $\nabla H\geq m\mathbbm 1>0$. Then there exists a constant $C>0$ such that for any $\mu,\nu\in \mathcal P(N)$ with $\{\mu>0\}\cap\{\nu>0\} = \emptyset$, and any solution $(i,u)\in \R^E_{\geq0}\times\R^N$ of
    \begin{align}\label{eq:8}
    \begin{cases}
        \min\{i,\nabla H(i)-Du\}=0,\\
        \div i = \mu-\nu
    \end{cases}
    \end{align}
    it holds that
    \[
    \operatorname{in-diam}_{\mathbbm 1} (\{i>0\}) \leq C\operatorname{diam}_{\mathbbm 1}(\{\mu>0\},\{\nu>0\}).
    \]
\end{theorem}

\begin{proof}
    Let
    \[
    \mathcal S := \{i\in \R^E_{\geq0} \ | \ \text{For some $u\in\R^N$, the pair $(i,u)$ satisfies \eqref{eq:8}}\}.
    \]
    Our goal is to show that $\mathcal S$ is contained in a compact set, independent of $\mu$ and $\nu$. Given that $\mathcal S$ is clearly closed, we would get that
    \[
    M:= \sup_{\substack{e\in E\\i\in \mathcal S}} \p_e H(i) <\8.
    \]
    This will allow us to apply the Remark \ref{rmk:1} and set $C=Mm^{-1}$ to conclude.
    
    Let $i\in \mathcal S$. By the Smirnov decomposition (Lemma \ref{lem:smirnov}), there is some $q\in \Pi(\mu,\nu)$ such that $i[q]\leq i$. Let us show that $i[q]=i$ and $\{q>0\}\ss\operatorname{SPath}(G)$.

    Let $j = i-i[q]$ and notice that $\div j=0$. If $j$ is non trivial there must be a loop $\w \in \{j>0\}\ss \{i>0\}$. However, this contradicts Lemma \ref{lem:pontryagin}, and we necessarily have that $i[q]=i$. The same argument by contradiction with Lemma \ref{lem:pontryagin} shows that $\{q>0\}\ss\operatorname{SPath}(G)$. Hence, $\mathcal S$ is contained in the image by $i$ of the compact set of $[0,1]^{\operatorname{SPath}(G)}$.
\end{proof}

\begin{corollary}\label{cor:fin_prop}
    Let $H\in C^1(\R^E_{\geq0})$ be such that $\nabla H\geq m\mathbbm 1>0$. There exists a constant $C>0$, such that the following holds: Given $S^{\pm}\ss N$ non-trivial and such that $S^-\cap S^+ = \emptyset$, $\cM \ss \{f\in \R^N \ | \ \{f^+>0\}\ss S^-, \{f^->0\}\ss S^+\}$, then for every $i \in \operatorname{BP}(G,\cM,H)$ it holds that
    \[
    \operatorname{in-diam}_{\mathbbm 1} (\{i>0\}) \leq C\operatorname{diam}_{\mathbbm 1}(S^-,S^+).
    \]
\end{corollary}

\subsection{Existence of the multiplier}\label{sec:sym}

Let $(i,\xi) \in \R^E_{\geq0}\times\R^E_{\geq0}$. In this section we analyze whether there exists $u \in \R^N$ such that $\min\{i,\xi-Du\}=0$. The first thing we notice is that the actual values of $i$ are irrelevant, what it is important is instead the set $\{i>0\}$. The goal is then to characterize when there exist solutions of
\begin{align}\label{eq:9}
\begin{cases}
Du \leq \xi \text{ in } E,\\
Du = \xi \text{ in } \{i>0\}.
\end{cases}
\end{align}

Notice that if both edges $\pm e \in \{i>0\}$, then $\xi(e) = Du(e) = -Du(-e) = -\xi(-e)$, given that $\xi(\pm e)\geq0$ we necessarily have that $\xi(\pm e)=0$. So this gives the first necessary hypothesis:
\begin{align}
    \label{eq:hyp1}
    \pm e\in\{i>0\} \qquad \Rightarrow \qquad \xi(\pm e)=0.
\end{align}

\begin{definition}
    Given a directed graph $G=(N,E)$, a set $S \ss E$, and $\xi\colon E\to \R$, we say that $\xi$ is odd over $S$ if $\xi(-e)=-\xi(e)$ whenever both edges $\pm e\in S$.
\end{definition}

Note that $(i,\xi) \in \R^E_{\geq0}\times\R^E_{\geq0}$ satisfies \eqref{eq:hyp1} if and only if $\xi$ is odd over $\{i>0\}$.

We can also restate \eqref{eq:9} in terms of the first inequality in such display. To do so, we first introduce the symmetrizations of $G$ and $\xi$ over the set $\{i>0\}$.

\begin{definition}
    Given a directed graph $G=(N,E)$ and a set $S \ss E$, we define the symmetrization $\operatorname{Sym}_S(G) = (N,\operatorname{Sym}_S(E))$ such that
    \[
    \operatorname{Sym}_S(E) := E \cup \{-e \ | \ e\in S\}.
    \]
\end{definition}

\begin{definition}
    Given a directed graph $G=(N,E)$, a set $S \ss E$, and $\xi\colon E\to\R$ odd over $S$, we define the (odd) symmetrization $\operatorname{Sym}_S[\xi]\colon\operatorname{Sym}_S(E)\to\R$ such that
    \[
    \operatorname{Sym}_S[\xi](e) := \begin{cases}
        \xi(e) \text{ if } e \in E,\\
        -\xi(-e) \text{ if } e \in \operatorname{Sym}_S(E) \sm E.
    \end{cases}
    \]
\end{definition}

From the assumption that $\xi$ is odd over $S$ we get that $\operatorname{Sym}_S[\xi]$ is odd over $S\cup -S \ss \operatorname{Sym}_S(E)$.

Let $\xi \colon E\to [0,\8)$ odd over $\{i>0\}$. We observe that $u \in \R^N$ satisfies $\min\{i,\xi-Du\}=0$ if and only if it satisfies
\[
Du \leq \operatorname{Sym}_{\{i>0\}}[\xi] \text{ in } \operatorname{Sym}_{\{i>0\}}(E).
\]
Indeed, if $e\in \{i>0\}$ we have that $\operatorname{Sym}_{\{i>0\}}[\xi](-e) = -\xi(e)$ such that we get the desired equality from the inequalities
\[
\xi(e) = -\operatorname{Sym}_{\{i>0\}}[\xi](-e) \leq -Du(-e) = Du(e) \leq \operatorname{Sym}_{\{i>0\}}[\xi](e) = \xi(e).
\]

Let $G=(N,E)$ be a directed graph and $\xi\colon E\to\R$ unsigned. From now on we focus on the problem of finding $u\in\R^N$ such that $Du\leq \xi$. We define as before the length and distance for the unsigned metric $\xi \in \R^{E}$
\[
    L_{\xi}(\w) := \sum_{e\in\w}\xi(e), \qquad d_{\xi}(x,y) := \inf_{\substack{\w^-=x\\\w^+=y}} L_{\xi}(\w).
\]
By default $d_{\xi}(x,y)=+\8$ if $\operatorname{Path}_{xy}(G)=\emptyset$.

We can guarantee that the lengths between two given points is bounded from below if the lengths of the loops are non-negative.

\begin{definition}
    Given a directed graph $G=(N,E)$ we say that $\xi\colon E\to\R$ satisfies the non-negative loop condition if for every loop $\w$, $L_\xi(\w)\geq 0$.
\end{definition}

Under the non-negative loop condition we obtain that if there exists a geodesic between two given nodes, then there also must exists a geodesic which is a simple path. Given that the simple paths form a finite set, the infimum can then be replaced by a minimum. Also the triangle inequality can be easily checked from the definition of $d_\xi$, under the non-negative loop condition.

The following lemma is a discrete version of Poincaré's lemma.

\begin{lemma}\label{lem:7}
    The following are equivalent for $\xi \in \R^E$:
    \begin{enumerate}
        \item There exists $u\in\R^N$ such that $Du\leq \xi$.
        \item $\xi$ has the non-negative loop condition.
    \end{enumerate}
\end{lemma}

\begin{proof}
    The implication $(1)\Rightarrow(2)$ follows from the telescopic identity. Given a loop $\w \in \operatorname{Path}(G)$, we have that
    \[
    \sum_{e\in \w} \xi(e) \geq \sum_{e\in \w} Du(e) = u(\w^+)-u(\w^-) = 0.
    \]

    Assume now that $(2)$ holds. Let $u\in \R^N$ such that
    \[
    u(x) := \min_{x_0\in N} d_{\xi}(x_0,x).
    \]
    Given $e= (x,y) \in E$ and $x_0\in N$ such that $u(x)= d_{\xi}(x_0,x)$ we obtain that
    \[
    Du(e) = u(y)-u(x) \leq d_{\xi}(x_0,y) - d_{\xi}(x_0,x) \leq d_{\xi}(x,y) \leq \xi(e).
    \]
    This settles the desired identity for $Du$.
\end{proof}

\begin{lemma}[1-Lipschitz type estimate] \label{lem:222}
    Given $(\xi,u) \in \R^{E}\times\R^N$ such that $Du\leq \xi$, we obtain the following upper bound on the differences of $u$ for any $x,y\in N$
    \[
    -d_\xi(y,x) \leq u(y)-u(x) \leq d_{\xi}(x,y).
    \]
\end{lemma}

\begin{proof}
    If $\w\in\operatorname{Path}_{xy}(G)$, then
\[
u(y)-u(x) = \sum_{e\in \w} Du(e) \leq \sum_{e\in \w} \xi(e) = L_{\xi}(\w).
\]
The right inequality now follows by taking the minimum over $\w\in \operatorname{Path}_{xy}(G)$. The left inequality follows by the same argument changing the roles of $x$ and $y$.
\end{proof}

In general we can not deduce that $|u(y)-u(x)| \leq d_{\xi}(x,y)$, because it is not necessarily always true that $d_\xi(x,y)=d_\xi(y,x)$.


As a consequence of Lemma \ref{lem:7} we obtain the following characterization for the existence of $u\in \R^N$ satisfying $\min\{i,\xi-Du\}=0$.

\begin{corollary}\label{cor:ext}
    The following are equivalent for $(i,\xi) \in \R^E_{\geq0}\times\R^E_{\geq0}$:
    \begin{enumerate}
        \item There exists $u\in \R^N$ such that $\min\{i,\xi-Du\}=0$.
        \item The metric $\xi$ is odd over $\{i>0\}$ and $\operatorname{Sym}_{\{i>0\}}[\xi]$ has the non-negative loop condition on the symmetrized graph $\operatorname{Sym}_{\{i>0\}}(G)$.
    \end{enumerate}
\end{corollary}

To conclude, we state a final characterization that brings us back to the Beckmann problem. Notice that the second formulation is independent of the multiplier.

\begin{definition}
    Given a directed graph $G=(N,E)$ and $H \in C^1(E)$, we say that $i \in \R^{E}_{\geq0}$ is a solution of the constitutive relation for the potential $H$ if any of the following two equivalent conditions holds:
    \begin{enumerate}
        \item There exists $u\in \R^N$ such that $\min\{i,\nabla H(i)-Du\}=0$.
        \item The metric $\xi=\nabla H(i)$ is odd over $\{i>0\}$ and $\operatorname{Sym}_{\{i>0\}}[\xi]$ has the non-negative loop condition on the symmetrized graph $\operatorname{Sym}_{\{i>0\}}(G)$.
    \end{enumerate}
\end{definition}

\subsection{Discrete elliptic equations}\label{sec:ellip_eqns}

When $H$ is convex, the constitutive relation can also be understood in terms of the \textit{Legendre transform of $H$} denoted by $H^*\colon \R^E\to (-\8,\8]$ and given by:
\[
H^*(\xi) := \sup_{i\in \R^E_{\geq0}} (\xi\cdot i - H(i)).
\]
One may also define $H$ in $\R^E$ by setting $H=+\infty$ outside $\R^E_{\geq0}$ and use the standard definition of the Legendre transform (taking the supremum over $\R^E$). In general, given $H\colon \R^E\to (-\8,\8]$ we denote $\operatorname{Dom}(H) := \{H<\8\}$.

Young's inequality states that for every $\xi\in \R^E$ and $i\in \R^E_{\geq0}$
\[
H^*(\xi) + H(i) \geq \xi\cdot i.
\]
Equality holds if and only if $\xi \in \p H(i)$. The set $\p H(i)$ is the \textit{sub-differential of $H$ at $i$} and is defined as
\[
\p H(i) := \{\xi \in \R^E \ | \ H(j) \geq H(i) + \xi\cdot(j-i) \text{ for all } j\in \R^E_{\geq0}\}.
\]

The equality case in Young's inequality also implies that $i \in \p H^*(\xi)$, and is an equivalence when $H$ is convex.

\begin{lemma}\label{lem:cons_rel}
    Let $H\in C^1(\R^E_{\geq0})$ be convex and $i\in \R^E_{\geq0}$. Then 
    \begin{align}\label{eq:5}
    \min\{i,\nabla H(i)-\xi\} = 0
    \end{align}
    is equivalent to $i \in \p H^*(\xi)$. Moreover, if $H$ is strictly convex we have that $H^* \in C^1(\operatorname{Dom}H^*)$ and \eqref{eq:5} is equivalent to
    \[
    i = \nabla H^*(\xi).
    \]
\end{lemma}

\begin{proof}
    It suffices to show that \eqref{eq:5} is equivalent to $\xi \in \p H(i)$, which in turn is equivalent to $i \in \p H^*(\xi)$ since $H$ is convex. The equivalence between the $C^1$-regularity of $H^*$ and the strict convexity of $H$ is a well known property of the Legendre transform. 
    
    By convexity, $\nabla H(i) \in \p H(i)$. If $\xi$ satisfies \eqref{eq:5}, then $\xi = \nabla H(i)-\lambda$ for some $\lambda\geq 0$, where $\lambda(e)>0$ only if $i(e)=0$. This gives $\lambda\cdot (j-i)\geq0$ for all $j\in \R^E_{\geq0}$. Hence
    \[
    H(j) - H(i) - \xi\cdot (j-i) = H(j) - H(i) - \nabla H(i)\cdot (j-i) + \l\cdot(j-i) \geq 0,
    \]
    and thus $\xi \in \p H(i)$.

    Now assume that $\xi \in\p H(i)$. Given $e_0\in E$, let $j = i + \e\mathbbm 1_{e_0}$ for some $\e>0$. Then
    \[
    \e\xi(e_0)= (j(e_0)-i(e_0))\xi(e_0) = \xi\cdot(j-i) \leq H(j)-H(i) = \e\p_{e_0} H(i) + o(\e).
    \]
    Dividing by $\e$ and taking the limit $\e\to0^+$, we find $\xi(e_0) \leq \p_{e_0} H(i)$. Since $e_0$ was arbitrary, we conclude $\xi\leq \nabla H(i)$.

    If $i(e_0)>0$, we can still choose $j = i - \e\mathbbm 1_{e_0} \in \R^E_{\geq0}$ for $\e>0$ small enough. The same argument shows $\xi(e_0) \geq \p_{e_0} H(i)$, which forces the equality. Thus, $\xi$ satisfies \eqref{eq:5} for the given $i$.
\end{proof}

We already pointed out the existence of solutions of the Beckmann problem in the Theorem \ref{thm:equiv_opt}. The following corollary addresses the uniqueness of solutions and characterizes them in the case that $H$ is strictly convex.

\begin{corollary}
    Let $k\in\{1,2\}$, $m_k\in \N$, $A_k\colon \R^N\to \R^{m_k}$ be linear, and $b_k\in\R^{m_k}$. Assume that
    \begin{align*}
    \mathcal M := \{f \in \R^N \ | \ A_1[f] = b_1, A_2[f] \geq b_2\}\ss\{f \in \R^N \ | \ f^\pm \in \mathcal P(N)\}
    \end{align*}
    and $\{i\in \R^E_{\geq 0} \ | \ \operatorname{div}i \in \cM\}\neq \emptyset$. Let $H\in C^1(\R^E_{\geq0})$ be strictly convex. Then, the unique minimizer $i^* \in \operatorname{BP}(G,\cM,H)$ can be computed as $i^* = \nabla H^*(Du)$, where $u = A_1^T[u_1] - A_2^T[u_2] \in \R^N$ and $u_k \in \R^{m_k}$ satisfy the system
    \begin{align}
    \label{eq:ellip}
        \begin{cases}
        A_1[\operatorname{div}(\nabla H^*(Du))] = b_1,\\
        \min\{A_2[\operatorname{div}(\nabla H^*(Du))]-b_2,u_2\}=0.
        \end{cases}
    \end{align}
\end{corollary}

It should be noted that, in general, there is no uniqueness for $u$. We will present specific examples in the next section. Neither we can guarantee the uniqueness of minimizers for $H\circ i$ over the feasible set $\mathcal Q(\Pi(\mu,\nu))$, even if $H$ is strictly convex. The mapping $q\in \mathcal P(\operatorname{Path}(G))\mapsto i[q]\in \R^E_{\geq0}$ may not be injective.

We say that the leading operator in \eqref{eq:ellip}, namely $\operatorname{div}(\nabla H^*(Du))$, is said to be elliptic because $H^*$ is convex. If for each $e\in E$, we get that $\partial_{e} H^*(Du)$ only depends on $Du(e)$, then we say that the operator is of local type. In general, however, we may have that $\operatorname{div}(\nabla H^*(Du))$ is non-local.

If $H^*$ satisfies that for some $0<\l\leq \L$, the functions $H^*(\xi)-\frac{\l}{2}|\xi|^2$ and $\frac{\L}{2}|\xi|^2-H^*(\xi)$ are convex, then we say that $\operatorname{div}(\nabla H^*(Du))$ is uniformly elliptic, otherwise we say that the operator is degenerate elliptic.

Notice that $H^*(\xi)-\frac{\l}{2}|\xi|^2$ is convex if and only if $H(i) - \frac{1}{2\l}|i|^2$ is concave. This is not the case because $H=+\8$ outside the positive quadrant. For instance, it is not possible to fit any paraboloid on top of $H$ over the boundary of $\R_{\geq0}^E$. Nevertheless, in the next section, we will see examples where the symmetry of the graph allows to reduce the problem to a uniformly elliptic equation.





\subsubsection{Local and strictly convex problems}

For each $e\in E$, consider the local cost $g_e\colon[0,\8)\to[0,\8)$, which is continuous, strictly increasing, and has inverse $h_e \colon[g_e(0),g_e(\8))\to [0,\8)$, where $g_e(\8):= \lim_{x\to\8}g_e(x)$. These costs arise from the strictly convex function
\[
H(i) = \sum_{e\in E}G_e(i(e)) \text{ where } G_e(x) = \int_0^xg_e(y)dy.
\]

Let us show that in this case $\partial_e H^*(\xi) = h_e^+(\xi_e)$ where $h_e^+\colon (-\8,g_e(\8)) \to [0,\8)$ is defined by
\[
h_e^+(x) =
\begin{cases}
    h_e(x) \text{ if } x \in [g_e(0),g_e(\8)),\\
    0 \text{ if } x< g_e(0).
\end{cases} 
\]
One can verify that the general case for this computation follows from the one-dimensional case. So, let us fix $e \in E$ and show that $\partial_e G_e^*(x) = h_e^+(x)$.

If $x\in [g_e(0),g_e(\8))$, we have $g_e(h_e^+(x)) = x$. From the convexity of $G_e$, it follows that for any $y\geq0$
\[
G_e(y) \geq G_e(h_e^+(x)) + g_e(h_e^+(x))(y-h_e^+(x)) = G_e(h_e^+(x)) + x(y-h_e^+(x)).
\]
Thus, in this case, we obtain $x \in \p G_e(h_e^+(x))$, which means that $\nabla G_e^*(x)=h_e^+(x)$.

If instead $x < g_e(0)$, we get that $h_e^+(x) = 0$. Once again using the convexity of $G_e$, we get that for any $y\geq 0$
\[
G_e(y) \geq G_e(0) + g_e(0)(y-0) \geq G_e(h_e^+(x)) + x(y-h_e^+(x)).
\]
From this, we conclude again that $\nabla G_e^*(x)=h_e^+(x)$.

\subsubsection{The p-Laplacian}

If $g_e(x) = x^{q-1}$ for $q>1$, we get a discrete analogue of the $p$-Laplacian, where $p=q/(q-1)$ is the conjugate of $q$,
\[
\operatorname{div}(\nabla H^*(Du))(x) = \Delta_pu(x) := \operatorname{div}(\max\{Du,0\}^{p-1}).
\]
In this case, if we further assume that $G$ is undirected ($E=-E$), we obtain
\[
\Delta_pu(x) = \sum_{y\sim x} |u(y)-u(x)|^{p-2}(u(y)-u(x)),
\]
where $y\sim x$ means that $(x,y)\in E$. When $p=q=2$ one obtains the classical discrete Laplacian
\[
\Delta u(x) := \Delta_2 u(x) = \sum_{y\sim x} (u(y)-u(x)).
\]

\subsubsection{The infinity-Laplacian and the 1-Laplacian}

An extreme case for the $p$-Laplacian occurs when $q=1$, so that $p=\8$. This case is obtained when $g = \mathbbm 1$ and $H(i) = \sum_{e\in E}i_e$. However, this potential is not strictly convex, so we are not able to give a straightforward constitutive relation of the form $i=\nabla H^*(Du)$ as before. In a more general setting we can also consider the constitutive relations that arise when $g = \xi$, a cost independent of the edge flow and perhaps with a different value over each edge.

Here we get the equation $\min\{i,\xi-Du\} = 0$, which was thoroughly discussed in the Sections \ref{sec:beckeq_to_wareq}, \ref{sec:bound_spt_EF}, and \ref{sec:sym}, and it is closely related with the stochastic geodesic problem for the metric $\xi$. If $u\in\R^N$ is such that $Du\leq \mathbbm 1$, the mapping
\[
u\mapsto \{\operatorname{div}i \in \R^N \ | \ \min\{i,\mathbbm 1 -Du\} = 0\},
\]
could be interpreted as a discrete version of the $\8$-Laplacian of $u$.

On the other extreme, we get that as $q\to\8$, then $p\to1^+$ and we limiting operator could be defined as the $1$-Laplacian
\[
\Delta_1 u := \operatorname{div}\1\mathbbm 1_{\{Du>0\}}\2.
\]

\subsubsection{Very degenerate operators}

From the modelling point of view, any local cost $g_e$ such that $g_e(0)=0$ is unrealistic, as there should be a positive cost to cross even an empty road. Following the $p$-Laplacian model discussed before, an alternative could be to consider for $\xi\in \R^E_{>0}$ and $q>1$ the cost $g_e(x) := \xi_e + x^{q-1}$. The corresponding operator is then
\[
\operatorname{div}(\nabla H^*(Du))(x) = \operatorname{div}(\max\{Du-\xi,0\}^{p-1}).
\]

If $G$ is undirected and $\xi$ is even ($\xi(-e)=\xi(e)$), we recover a very degenerate equation of the form
\[
\operatorname{div}(\nabla H^*(Du))(x) = \operatorname{div}(\mathbbm 1_{\{|Du|>\xi\}}|Du-\xi|^{p-2 }(Du-\xi)).
\]
The degeneracy of the equation is manifested in the fact that any function $u$ such that $|Du| \leq \xi$ is a solution of the homogeneous problem. This provides a further example on the lack of uniqueness.

\subsubsection{Non-local operators}

Assume that $G$ is an undirected graph. We now aim to model the congestion that occurs when traffic flows in both directions between two adjacent nodes.

Let $g_e = g_e(i_e,i_{-e}) \colon \R^2_{\geq0}\to [0,\8)$ the cost over the edge $e$, when the flow over such edge is $i_e$, and the flow on the opposite direction in $i_{-e}$. We will assume that $g_e(\R^2_{>0})\ss(0,\8)$, so there is a positive cost when both flows are positive.

The constitutive relations in this case are coupled in the following form. In the next identity we fix $e\in E$ and let $g = g_e$, $h=g_{-e}$, $i=i_e$ and $j=i_{-e}$. Also notice that $Du(-e)=-Du(e)$
\[
\begin{cases}
    \min\{i,g(i,j)-Du(e)\} = 0,\\
    \min\{j,h(j,i)+Du(e)\} = 0.
\end{cases}
\]
This implies that $i$ and $j$ can not be simultaneously positive. Otherwise we get the contradiction
\[
0 < g(i,j) = Du(e) = -h(j,i) < 0. 
\]

If $i>0$ (and hence $j=0$), we must have that $g(i,0) = Du(e)$. If $j>0$ (and $i=0$) we have instead that $Du(e) = -h(j,0)$. Finally, if $i=j=0$ we obtain $g(0,0) \geq Du(e) \geq-h(0,0)$.

\subsubsection{A quadratic and non-local potential}

Let us now examine a specific case from the previous example, still with $G$ undirected. For $\a,\gamma,\xi\in \R^E_{\geq0}$, with $\gamma$ even, the cost
\[
g_e(i_e,i_{-e}) := \a_\e i_e+\gamma_e i_{-e} + \xi_e,
\]
arises from the potential
\[
H(i) = \sum_{e\in E}\1\frac{\a_e}{2}i_e^2 + \frac{\gamma_e}{2}i_ei_{-e}+\xi_e i_e\2.
\]
Assuming that $\a>0$, we deduce from the previous general analysis, the constitutive relation
\[
i = \frac{1}{\a}\max\{Du-\xi,0\}.
\]
Surprisingly, this identity is local and independent of the parameter $\gamma$. For $\a=\mathbbm 1$ and $\xi=0$, the resulting operator is the discrete Laplacian.

We will find a truly non-local model in our discussion of dynamic examples in Section \ref{sec:dyn_ex}.

\section{Dynamic problems}\label{sec:dm}

We consider in this section the dynamic analogue of efficient equilibria. The idea is that the agents move on the graph at discrete units of time. Their configuration at a given time provides an edge flow that determines the metric over the graph at that time.

To be precise, we consider a finite time horizon given by a positive integer $T$. From the original graph $G=(N,E)$ we construct the time extended graph $G^T = (N^T,E^T)$ such that $N^T := N\times \{0,1,\ldots,T\}$. The edges are defined from $E$ using
\[
E^T := \{(e,t) := ((x,t-1),(y,t))\ | \ e=(x,y) \in E, t\in\{1,\ldots, T\}\}.
\]
Given $e=(x,y) \in E$ and $t\in\{1,\ldots, T\}$, we may also denote the edge $(e,t)$ by $(x,y,t)$. See Figure \ref{fig:ext_graph}.

\begin{figure}
    \centering
    \includegraphics[width=14cm]{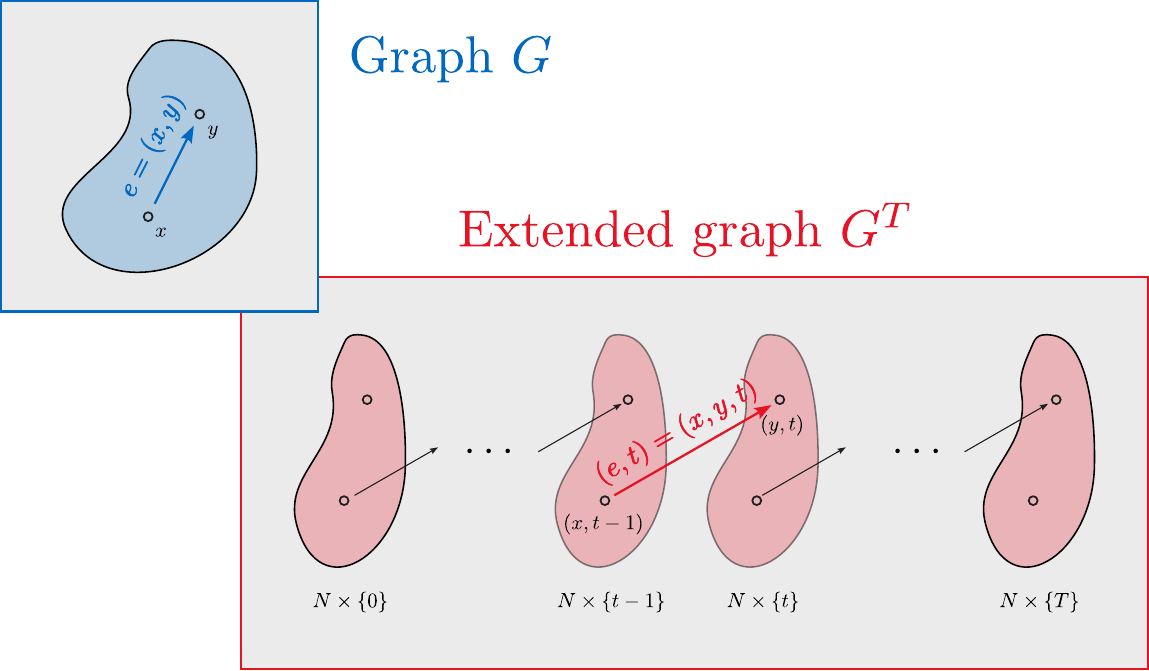}
    \caption{For the construction of the extended graph $G^T$ from the graph $G$ we consider $T+1$ copies of the nodes and join two nodes $(x,t-1)$ and $(y,t)$ if and only if $e=(x,y)$ is an edge of $G$.}
    \label{fig:ext_graph}
\end{figure}


Given a feasible set for the transport plan $\Gamma \ss \mathcal P(N\times N)$, we consider the extended plan
\begin{align*}
\Gamma^T := \left\{\gamma \in \mathcal{P}(N^T \times N^T) \ \middle| \right.
& \{\pi^-[\gamma]>0\} \subseteq N \times \{0\}, \\
& \{\pi^+[\gamma]>0\} \subseteq N \times \{1,\ldots,T\}, \\
&\left. \pi[\gamma] \in \Gamma \right\},
\end{align*}
where the marginals $\pi^\pm$ are defined as in \eqref{eq:marg} (but for the graph $G^T$), and $\pi\colon \mathcal P(N^T\times N^T) \to \mathcal P(N\times N)$ is given by
\[
\pi[\gamma](x,y) := \sum_{t,s=0}^T\gamma((x,t),(y,s)).
\]
The restriction $\{\pi^-[\gamma]>0\}\ss N\times\{0\}$ serves as an initial condition, indicating that all the mass required for transport is present at time $t=0$. On the other hand, $\{\pi^+[\gamma]>0\} \subseteq N \times \{1,\ldots,T\}$ means that the transport can be finished at any time between $1$ and $T$. Finally, the condition $\pi[\gamma] \in \Gamma$ fixes the transport that has to be completed in the given interval of time.

To formulate the dynamic Wardrop equilibrium problem as a Beckmann problem (Theorem \ref{thm:equiv_opt}), we require that $\Gamma^T = \Pi(G^T,\mathcal M)$ for some $\cM \ss \R^{N^T}$ (Definition \ref{def:pi_gm}). This is possible for instance, for the long-term problem or the the minimal-time mean field game (Remark \ref{rmk:MR3986796}).

In general, for $S \ss \mathcal P(N)\times\mathcal P(N)$ and $\Gamma = \bigcup_{(\mu,\nu)\in S} \Pi(\mu,\nu)$ we have from the definition of the extension that
\[
\Gamma^T = \bigcup_{f \in \cM^T(S)} \Pi(f^+,f^-),
\]
where
\begin{align}\label{eq:ext_gamma}
    \cM^T(S) := \left\{f \in \R^{N^T} \ \middle| \right.
    & f\geq 0 \text{ in } N \times \{0\},\\
    & f\leq 0 \text{ in } N \times \{1,\ldots,T\}, \notag\\
    &\left. (\pi[f^+],\pi[f^-]) \in S \right\},\notag
\end{align}
and $\pi\colon \R^{N^T}\to \R^N$ is given by
\[
\pi[f](x) := \sum_{t=0}^Tf(x,t).
\]
As a corollary from Lemma \ref{lem:pi} we obtain the following result:

\begin{corollary}\label{cor:ext_gamma}
    Given $S \ss \mathcal P(N)\times\mathcal P(N)$ and $\Gamma = \bigcup_{(\mu,\nu)\in S} \Pi(\mu,\nu)$, we have that
    \[
    \Gamma^T = \Pi(G^T,\mathcal M^T(S)).
    \]
\end{corollary}

Additionally, if $\cM \ss \{f\in \R^N \ | \ f^\pm\in\mathcal P(N)\}$, we obtain the follwing corollary.
    
\begin{corollary}
    For $\cM \ss \{f\in \R^N \ | \ f^\pm\in\mathcal P(N)\}$ and $\Gamma = \Pi(G,\cM)$, it holds that $\Gamma^T = \Pi(G^T,\cM^T)$ where
    \begin{align*}
    \cM^T := \left\{f \in \R^{N^T} \ \middle| \right.
    & f\geq 0 \text{ in } N \times \{0\}, \\
    & f\leq 0 \text{ in } N \times \{1,\ldots,T\}, \\
    &\left. \pi[f] \in \cM \right\}.
    \end{align*}
\end{corollary}

Assume that $S \ss \mathcal P(N)\times\mathcal P(N)$ can be presented in terms of affine equality and inequality constrains. Namely, for $k\in \{1,2\}$, $m_k \in \N$, $A_{k,\pm}\colon \R^N\to \R^{m_k}$ linear, and $b_k \in \R^{m_k}$
\[
S = \{(\mu,\nu) \in \R^N\times \R^N \ | \ A_{1,-}\mu+A_{1,+}\nu = b_1, A_{2,-}\mu+A_{2,+}\nu \geq b_2\}.
\]
Then we obtain that
\begin{align*}
\cM^T(S) := \left\{f \in \R^{N^T} \ \middle| \right.
& f\geq 0 \text{ in } N \times \{0\},\\
& f\leq 0 \text{ in } N \times \{1,\ldots,T\}, \notag\\
& A_{1,-}[\pi[\mathbbm 1_{N\times\{0\}}f]]-A_{1,+}[\pi[\mathbbm 1_{N\times\{1,\ldots,T\}}f]] = b_1, \notag\\
&\left. A_{2,-}[\pi[\mathbbm 1_{N\times\{0\}}f]]-A_{2,+}[\pi[\mathbbm 1_{N\times\{1,\ldots,T\}}f]] \geq b_2 \right\},\notag
\end{align*}
This leads to the following equations for $u\in \R^{N^T}$, $u_k \in \R^{m_k}$, and $u_\pm := A_{1,\pm}^T[u_1]-A_{2,\pm}^T[u_2] \in \R^N$
\begin{align}\label{eq:main_long}
\begin{cases}
    \min\{i,\nabla H(i) - Du\}=0 \text{ in } E^T,\\
    \min\{\operatorname{div} i,\pi^T[u_-]-u\}=0 \text{ in } N\times\{0\},\\
    \min\{-\operatorname{div} i,u+\pi^T[u_+]\}=0 \text{ in } N\times\{1,\ldots,T\},\\
    \min\{A_{2,-}[\pi[\mathbbm 1_{N\times\{0\}}\operatorname{div}i]]-A_{2,+}[\pi[\mathbbm 1_{N\times\{1,\ldots,T\}}\operatorname{div}i]]-b_2, u_2\} =0,\\
    A_{1,-}[\pi[\mathbbm 1_{N\times\{0\}}\operatorname{div}i]]-A_{1,+}[\pi[\mathbbm 1_{N\times\{1,\ldots,T\}}\operatorname{div}i]] = b_1.
\end{cases}
\end{align}
The linear operator $\pi^T\colon \R^N\to \R^{N^T}$ is the transpose of $\pi\colon \R^{N^T}\to \R^N$ which is given by $\pi^T[v](x,t) = v(x)$ for every $(x,t) \in N^T$.

The derivation of these equations follows as in the proof of the Lemma \ref{lem:4} from the computation of the Karush-Kuhn-Tucker conditions. We omit the details in this derivation, as an alternative approach in the Sections \ref{sec:aux_cons} and \ref{sec:cons_rel_dyn}, with the computations outlined in Theorem \ref{thm:main}. 


\subsection{Some dynamic equations}\label{sec:some_dyn_eq}

Let us present some examples to illustrate the previous computation. Additionally, we hope that the alternative approach introduced in the next section will further elucidate their meaning.

\subsubsection{The dynamic long-term problem}

Let $\mu,\nu\in \mathcal P(N)$. If $\G = \Pi(\mu,\nu)$, we get that $\G^T = \Pi(G^T,\mathcal M)$ for
\begin{align*}
\cM = \left\{f \in \R^{N^T} \ \middle| \right.
& f \geq 0 \text{ in } N \times \{0\},\\
& f\leq 0 \text{ in } N \times \{1,\ldots,T\}, \\
& \pi[\mathbbm 1_{N\times\{0\}}f] = \mu,\\
&\left. \pi[\mathbbm 1_{N\times\{1,\ldots,T\}}f] = -\nu \right\}.
\end{align*}
Hence, we obtain the following variational equations for $u\in \R^{N^T}$, $u_\pm\in \R^N$
\begin{align}\label{eq:long_term1}
\begin{cases}
    \min\{i,\nabla H(i) - Du\}=0 \text{ in } E^T,\\
    \min\{\operatorname{div} i,\pi^T[u_-]-u\}=0 \text{ in } N\times\{0\},\\
    \min\{-\operatorname{div} i,u+\pi^T[u_+]\}=0 \text{ in } N\times\{1,\ldots,T\},\\
    \operatorname{div} i = \mu \text{ in } N\times\{0\},\\
    \pi[\mathbbm 1_{N\times\{1,\ldots,T\}}\operatorname{div} i] = -\nu.
\end{cases}
\end{align}

\subsubsection{The minimal-time mean field game}

For the dynamic version of the problem presented in Remark \ref{rmk:MR3986796}, which is a discrete analogue of \cite{MR3986796}, we consider $\m\in \mathcal P(N)$, $S\ss N$, and $\Gamma = \bigcup_{\{\nu>0\}\ss S}\Pi(\mu,\nu)$. Then $\Gamma^T = \Pi(G^T,\mathcal M)$ for
\begin{align*}
\cM = \left\{f \in \R^{N^T} \ \middle| \right.
& f\geq 0 \text{ in } N \times \{0\},\\
& f \leq 0 \text{ in } N \times \{1,\ldots,T\}, \\
& \pi[\mathbbm 1_{N\times\{0\}}f] = \mu, \\
&\left. f = 0 \text{ in } (N\sm S) \times \{1,\ldots,T\} \right\}.
\end{align*}
We obtain the following equations for $u,u_+\in \R^{N^T}$, $u_-\in \R^N$
\begin{align}\label{eq:min_time_mfg}
\begin{cases}
    \min\{i,\nabla H(i) - Du\}=0 \text{ in } E^T,\\
    \min\{\operatorname{div} i,\pi^T[u_-]-u\}=0 \text{ in } S\times\{0\},\\
    \min\{-\operatorname{div} i,u+u_+\}=0 \text{ in } S\times\{1,\ldots,T\},\\
    \operatorname{div} i = 0 \text{ in } (N\sm S) \times \{1,\ldots,T\},\\
    \operatorname{div} i = \mu \text{ in } N\times\{0\}.
\end{cases}
\end{align}

\subsection{An auxiliary construction}\label{sec:aux_cons}

Both systems of equations \eqref{eq:long_term1} and \eqref{eq:min_time_mfg} in the previous examples can be simplified, though it may not be immediately obvious from the algebraic expressions. The following geometric construction offers the advantage of deducing and equivalent set of equations simpler than those given above.

Let $G^{T,\W}:= (N^{T,\W},E^{T,\W})$ such that $N^{T,\W} := N^T\cup N^\W$ and $E^{T,\W} = E^T\cup E^\W$. Here $\W$ is just a symbol, $N^\W = N\times\{\W\}$, and new edges are given by
\[
E^\W := \{(x,t,\W) := ((x,t),(x,\W)) \in N^T\times N^\W\ | \ x\in N, \, t\in\{1,\ldots,T\}\}.
\]
The idea is to use $N^\W$ as a final deposit at the end of the transport.

Let $\iota \colon \R^N\times \R^N \to \R^{N^{T,\W}}\times\R^{N^{T,\W}}$ be the inclusion given by
\[
\iota[\gamma]((x,t),(y,s)) := \begin{cases}
    \gamma(x,y) \text{ if $(t,s)=(0,\W)$},\\
    0 \text{ otherwise}.
\end{cases}
\]
Given $\Gamma \ss \mathcal P(N\times N)$, we define the extension
\[
\Gamma^{T,\W} := \{\iota[\gamma] \in \mathcal P(N^{T,\W}\times N^{T,\W}) \ | \ \gamma\in \Gamma\}.
\]

As an example let us compute $\Gamma^{T,\W}$ for $\Gamma = \bigcup_{(\mu,\nu)\in S}\Pi(\mu,\nu)$. In the following computations we define the inclusions $\iota_0,\iota_\W\colon \R^N\to\R^{N^{T,\W}}$ by
\[
    \iota_0[\mu](x,t) := \begin{cases}
    \m(x) \text{ if } t=0,\\
    0 \text{ otherwise},
\end{cases}
\qquad
    \iota_\W[\nu](y,s) := \begin{cases}
    \nu(y) \text{ if } s=\W,\\
    0 \text{ otherwise}.
\end{cases}
\]

\begin{lemma}
    Given $\gamma \in \mathcal P(N\times N)$,
    \[
    \pi^-[\iota[\gamma]] = \iota_0[\pi^-[\gamma]] \qquad\text{ and }\qquad \pi^+[\iota[\gamma]] = \iota_\W[\pi^+[\gamma]].
    \]
\end{lemma}

\begin{proof}
    Let $\mu=\pi^-[\gamma]$. For $t\neq 0$ we have that
    \[
    \pi^-[\iota[\gamma]](x,t)=\sum_{(y,s)\in N^{T,\W}} \iota[\gamma]((x,t),(y,s)) = 0 =\iota_0[\mu](x,t).
    \]
    Meanwhile, for $t=0$,
    \[
    \pi^-[\iota[\gamma]](x,0) = \pi^-[\gamma](x) = \mu(x) = \iota_0[\mu](x,0).
    \]
    A similar computation shows that $\pi^+[\gamma] = \iota_\W[\mu]$, which concludes the proof.
\end{proof}

\begin{corollary}\label{cor:ext_gamma2}
    For $S \ss \mathcal P(N)\times \mathcal P(N)$ and $\Gamma = \bigcup_{(\mu,\nu)\in S} \Pi(\mu,\nu)$
    \[
    \Gamma^{T,\W} = \bigcup_{(\mu,\nu)\in S}\Pi(\iota_0[\mu],\iota_\W[\nu]) = \Pi(G^{T,\W},\cM^\W(S)),
    \]
    where
    \begin{align}\label{eq:ext_gamma_omega}
    \cM^\W(S) := \{f \in \R^{N^{T,\W}} \ | \ f = \iota_0[\mu]-\iota_\W[\nu] \text{ for some } (\mu,\nu)\in S\}.
    \end{align}
\end{corollary}

The costs that we will consider over $G^{T,\W}$ will always be extensions of costs over $G^T$. Given $i \in \R^{E^{T,\W}}_{\geq 0}$, let us denote by $i' \in \R^{E^T}_{\geq 0}$ the $E^{T}$-coordinates of $i$. In other words
\[
i'_e = i_e \text{ for any } e\in E^T\ss E^{T,\W}.
\]
A cost $g\colon \R^{E^{T}}_{\geq0}\to \R^{E^{T}}_{\geq0}$ can then be extended to a cost $g \colon \R^{E^{T,\W}}_{\geq0}\to \R^{E^{T,\W}}_{\geq0}$ (denoted the same) such that
\[
g_e(i) = \begin{cases}
    g_e(i') \text{ if } e\in E^T,\\
    0 \text{ otherwise}.
\end{cases}
\]
In other words, it has no cost, and effect on the cost of other edges, to impose any flow towards the deposits. In a similar way we may extend a potential $H\colon \R^{E^{T}}_{\geq0}\to \R$ to a potential $H \colon \R^{E^{T,\W}}_{\geq0}\to \R$, namely $H(i)=H(i')$.

Next we show the equivalence between the sets of efficient equilibria $\operatorname{WE}(G^T,\Gamma^T,g)$ and $\operatorname{WE}(G^{T,\W},\Gamma^{T,\W},g)$. This is the content of Theorem \ref{thm:ext2} ahead. Before proceeding, we introduce some necessary constructions.


Any $\w' \in \operatorname{Path}(G^T)$ can be extended in a unique way to a $\W(\w') \in \operatorname{Path}(G^{T,\W})$, with just one additional edge ending the path in $N^\W$. Let us say that $(x,t)=(\w')^+$ is the final node of $\w'$, then $(x,t,\W) \in E^\W$ is the only edge we can add to $\w'$ in order to end $\W(\w')$ at $N^\W$. On the other hand, if $\w\in \operatorname{Path}(G^{T,\W})$ is a path (of length at least one) that ends at $N^\W$, then there exists a unique path $\w'\in \operatorname{Path}(G^T)$ such that $\W(\w')=\w$. In this case we say that $\w'$ is the restriction of $\w$ to $G^T$.

Given a path profile $q' \in \mathcal P(\operatorname{Path}(G^T))$ we construct the extension $\W[q']\in \mathcal P(\operatorname{Path}(G^{T,\W}))$ such that
\[
\W[q'](\w) := \begin{cases}
    q'(\w') \text{ if } \w=\W(\w'),\\
    0 \text{ otherwise}.
\end{cases}
\]
In other words, $\W[q']$ is supported on paths that end at $N^\W$, and for each of these paths assigns the same weight that $q'$ gives to their corresponding restriction. Also, if $q\in \mathcal P(\operatorname{Path}(G^{T,\W}))$ is supported on paths of length at least one that end at $N^\W$, there exists a unique $q' \in \mathcal P(\operatorname{Path}(G^T))$ such that $\W[q'] = q$. We say in this case that $q'$ is the restriction of $q$.

\begin{lemma}
    Let $q' \in \mathcal P(\operatorname{Path}(G^T))$ be supported on paths that start on $N\times\{0\}$. Then
    \[
    \iota[\pi[\gamma[q']]] = \gamma[\W[q']].
    \]
\end{lemma}

\begin{proof}
    We have to show that for any $(x,t),(y,s) \in N^{T,\W}$ we have that
    \[
    \gamma[\W[q']]((x,t),(y,s)) = \begin{cases}
        \pi[\gamma[q']](x,y) \text{ if } (t,s)=(0,\W),\\
        0 \text{ otherwise}.
    \end{cases}
    \]
    
    From the definition of the transport plan
    \[
    \gamma[\W[q']]((x,t),(y,s)) = \sum_{\w\in\operatorname{Path}_{(x,t)(y,s)}(G^{T,\W})} \W[q'](\w).
    \]
    The extension $\W[q']$ is supported on paths that start at $N\times\{0\}$ (by hypothesis) and end at $N^\W$ (by definition of the map $\W$). Hence, we automatically have that $\gamma[\W[q']]((x,t),(y,s))=0$ if $(t,s)\neq(0,\W)$.

    In the case $(t,s)=(0,\W)$, we have that any $\w\in \operatorname{Path}_{(x,0)(y,\W)}(G^{T,\W})$ has a restriction $\w' \in \operatorname{Path}_{(x,0)(y,t)}(G^T)$ for some $t\in\{1,\ldots,T\}$. Then
    \begin{align*}
    \gamma[\W[q']]((x,0),(y,\W)) &= \sum_{\w \in \operatorname{Path}_{(x,0)(y,\W)}(G^{T,\W})} \W[q'](\w)\\
    &= \sum_{t=1}^T \sum_{\w'\in\operatorname{Path}_{(x,0)(y,t)}(G^{T})} q'(\w')\\
    &= \sum_{t=1}^T\gamma[q']((x,0),(y,t))\\
    &= \pi[\gamma[q']](x,y),
    \end{align*}
    which concludes the proof.
\end{proof}

\begin{corollary}
    Let $\Gamma \ss \mathcal P(N\times N)$ and $q' \in \mathcal P(\operatorname{Path}(G^T))$. Then, $q' \in \mathcal Q(\Gamma^T)$ if and only if $\W[q']\in \mathcal Q(\Gamma^{T,\W})$.
\end{corollary}

\begin{proof}
    For $q'\in\mathcal P(\operatorname{Path}(G^T))$, either of the conditions, $q' \in \mathcal Q(\Gamma^T)$ or $\W[q']\in \mathcal Q(\Gamma^{T,\W})$, requires that $q'$ is supported on paths that start at $N\times\{0\}$.

    Let us assume first $\gamma[q'] \in \Gamma^T$. This means that $\gamma := \pi[\gamma[q']]\in \Gamma$ and then by the previous lemma $\gamma[\W[q']] = \iota[\gamma] \in \Gamma^{T,\W}$.

    Assume now that $\gamma[\W[q']] \in \Gamma^{T,\W}$. Observing that the map $\gamma \in \mathcal P(N\times N) \mapsto \iota[\gamma] \in \mathcal P(N^{T,\W}\times N^{T\times\W})$ is injective, we get once again by the previous lemma that $\pi[\gamma[q']] \in \Gamma$. The other conditions that define $\Gamma^T$ hold automatically for $\gamma[q']$, given that $q'$ is supported on paths that start at $N\times\{0\}$.
\end{proof}

The following theorem follows from the previous corollary.

\begin{theorem}\label{thm:ext2}
    Let $\Gamma \ss \mathcal P(N\times N)$, $g\colon \R^{E^T}_{\geq0}\to\R^{E^T}_{\geq0}$, and $q' \in \mathcal P(\operatorname{Path}(G^T))$. Then $q' \in \operatorname{WE}(G^T,\Gamma^T,g)$ if and only if $\W[q'] \in \operatorname{WE}(G^{T,\W},\Gamma^{T,\W},g)$.
\end{theorem}

\begin{proof}
    The previous corollary states that $q' \in \mathcal Q(\Gamma^T)$ if and only if $\W[q']\in \mathcal Q(\Gamma^{T,\W})$. Equilibria means in either case that the expected length equals the expected distance. In both graphs these computations coincide given that the metric vanishes on $E^\W$.
\end{proof}

\subsection{The constitutive relations for the dynamic Beckmann problem}\label{sec:cons_rel_dyn}

Let us fix in this section $S \subset \mathcal P(N) \times \mathcal P(N)$ and $\Gamma = \bigcup_{(\mu,\nu) \in S} \Pi(\mu,\nu)$.

As a consequence of Corollary \ref{cor:ext_gamma}, we demonstrated that the dynamic Wardrop problem for $\Gamma^T$ as above can be formulated as a Beckmann problem over $G^T$. The key observation is that $\Gamma^T = \Pi(G^T,\cM^T(S))$, where $\cM^T(S)$ is given by \eqref{eq:ext_gamma}. This allows us to apply the general equivalence results from Theorem \ref{thm:equiv_wardrop} and Theorem \ref{thm:equiv_opt} in this context.

Furthermore, by Theorem \ref{thm:ext2}, the dynamic Wardrop problem also has an equivalent formulation over $G^{T,\W}$. From Corollary \ref{cor:ext_gamma2}, we know that $\Gamma^{T,\W} = \Pi(G^{T,\W},\cM^\W(S))$, where $\cM^\W(S)$ is given by \eqref{eq:ext_gamma_omega}. Following a similar approach as in Section \ref{sec:some_dyn_eq}, we now present the equations that arise from the Beckmann problem for $\Gamma^{T,\W} = \Pi(G^{T,\W},\cM^\W(S))$.

\begin{theorem}\label{thm:main}
    Let for $k\in \{1,2\}$, $m_k \in \N$, $A_{k,\pm}\colon \R^N\to \R^{m_k}$ linear, $b_k \in \R^{m_k}$, and
    \[
    S = \{(\mu,\nu) \in \R^N\times \R^N \ | \ A_{1,-}[\mu]+A_{1,+}[\nu] = b_1, A_{2,-}[\mu]+A_{2,+}[\nu] \geq b_2\} \ss \mathcal P(N)\times \mathcal P(N).
    \]
    Given $H\in C^1(\R^{E^T})$, we have that for any $i \in \operatorname{BP}(G^{T,\W},\cM^\W(S),H)$, there exist $u \in \R^{N^{T,\W}}$ and $u_k\in \R^{m_k}$ such that $\iota_0^T[u] = A_{1,-}^T[u_1]-A_{2,-}^T[u_2]$, $\iota_\W^T[u] = A_{1,+}^T[u_1]-A_{2,+}^T[u_2]$, and
    \begin{align}
        \label{eq:main}
    \begin{cases}
        \min\{i,\nabla H(i)-Du\} = 0 \text{ in } E^{T,\W},\\
        \operatorname{div}i = 0 \text{ in } N\times\{1,\ldots,T\},\\
        \min\{A_{2,-}[\iota_0^T[\operatorname{div}i]] - A_{2,+}[\iota_\W^T[\operatorname{div}i]] - b_2, u_2\}=0,\\
        A_{1,-}[\iota_0^T[\operatorname{div}i]] - A_{1,+}[\iota_\W^T[\operatorname{div}i]] = b_1.\\
    \end{cases}
    \end{align}
    Moreover, if $H$ is convex, then the previous equations imply that $i \in \operatorname{BP}(G^{T,\W},\cM^\W(S),H)$.
\end{theorem}

We will refer to the first two equations in \eqref{eq:main} as the dynamic Beckmann equations over $G^{T,\W}$. These are the ones that will be common to this general family of problems. In the next section we will see how they can be used to obtain bounds on the support of the edge flow, and also to determine whether the flow can be extended by zero or not.

Comparing \eqref{eq:main_long} with \eqref{eq:main}, we observe that \eqref{eq:main} is a simplification of the system \eqref{eq:main_long}, which has more multipliers and equations. Nevertheless, since both systems describe the critical points of equivalent Beckmann problems, they are ultimately equivalent.

For reference, let us mention that the adjoint transformations $\iota^T_0\colon \R^{N^{T,\W}}\to \R^N$ and $\iota^T_\W\colon \R^{N^{T,\W}}\to \R^N$ are given by
\[
\iota_0^T[f](x) = f(x,0), \qquad \iota_\W^T[f](x) = f(x,\W).
\]

\begin{proof}
    Under the given assumptions, we can express $\cM^\W(S)$ in terms of affine constraints as follows:
    \begin{align*}
    \cM^\W(S) := \left\{f \in \R^{N^{T,\W}} \ \middle| \right.
    & f = 0 \text{ in } N \times \{1,\ldots,T\}\\
    & A_{1,-}[\iota_0^T[f]] - A_{1,+}[\iota_\W^T[f]] = b_1\\
    &\left. A_{2,-}[\iota_0^T[f]] - A_{2,+}[\iota_\W^T[f]] \geq b_2 \right\}.
    \end{align*}
    By computing the critical equations resulting from the Karush-Kuhn-Tucker conditions, as we already did in the proof of Lemma \ref{lem:4}, we derive the system equations in the theorem.
\end{proof}

We can separate the constitutive relations $\min\{i,\nabla H(i)-Du\} = 0$ in the two sets of edges $E^T$ and $E^\W$. Using that $\nabla H=0$ over $E^\W$, we get that
\[
\min\{i,-Du\} = 0 \text{ in } E^\W.
\]
If $(x,t_1),(x,t_2)\in N^T$ are such that $i(x,t_1,\W)$ and $i(x,t_2,\W)$ are both positive, then we necessarily have that
\[
u(x,t_1) = -u(x,\W)-\cancel{Du(x,t_1,\W)} = -u(x,\W)-\cancel{Du(x,t_2,\W)} = u(x,t_2).
\]
In other words, $u(x,\cdot)$ has to be constant on the times when $x$ sends mass towards $N^\W$.

    

In the case of the dynamic long-term problem given by $S = (\mu,\nu)$, we get that $\cM^\W(S) = \{\iota_0^T[\mu]-\iota_\W^T[\nu]\}$. The corresponding equations become
\[
\begin{cases}
    \min\{i,\nabla H(i)-Du\} = 0 \text{ in } E^{T,\W},\\
    \operatorname{div}i = \iota_0[\mu] - \iota_\W[\nu] \text{ in } N^{T,\W}.
\end{cases}
\]

For the minimal-time mean field game problem we have that $\cM^\W(S) = \{\iota_0^T[\mu] - \iota_\W^T[\nu] \ | \ \nu = 0 \text{ in } N\sm S\}$. The equations in this case are
\[
\begin{cases}
    \min\{i,\nabla H(i)-Du\} = 0 \text{ in } E^{T},\\
    \operatorname{div}i = \iota_0[\mu] \text{ in } N^{T,\W} \sm (S\times \{\W\}).
\end{cases}
\]

In both cases, these equations are a simplification of those displayed at \eqref{eq:long_term1} and \eqref{eq:min_time_mfg} respectively.

\subsection{Finite propagation for the edge flow}

For the following result we consider a time dependent potential $H\colon\R^{E}_{\geq0}\times\Z_{\geq 1} \to \R$. For each $T\geq 1$ we construct the extended potential $H^T\colon \R^{E^T}_{\geq0}\to\R$ by
\[
H^T(i) = \sum_{t=1}^T H(i(t),t).
\]
Notice that $g_{(e,t)} = \p_{(e,t)} H^T = \p_e H(\cdot,t)$ gives a cost that is local in time.

The following result states sufficient hypotheses on a potential $H^T$ in order to obtain a bound on the support of any solution of the dynamic Beckmann equation in $G^{T}$.

\begin{theorem}\label{thm:fin_prop2}
    Let $H\colon\R^{E}_{\geq0}\times\Z_{\geq 1} \to \R$ be a time dependent potential such that
    \begin{align*}
    m &:=  \inf\{\p_e H(i,t) \ | \ t\geq 1, i\in [0,1]^E, e\in E\}>0,\\
    M &:= \sup\{\p_e H(i,t) \ | \ t\geq 1, i\in [0,1]^E, e\in E\}<\8.
    \end{align*}
    For $T\geq \operatorname{diam}_{\mathbbm 1}(G)$, $H^T(i) = \sum_{t=1}^T H(i(t),t)$, $\mu,\nu\in\mathcal P(N)$, and any solution $i \in \R^{E^{T,\W}}_{\geq 0}$ of the equations
    \[\begin{cases}
        \min\{i,\nabla H^T(i)-Du\} = 0 \text{ in } E^{T,\W},\\
        \operatorname{div}i=\iota_0[\mu]-\iota_\W[\nu] \text{ in } N^{T,\W},
    \end{cases}\]
    it happens that
    \[
    T_0 := \max\{t\in \{1,\ldots, T\}\ | \ i(t) \neq 0\} \leq Mm^{-1}\operatorname{diam}_{\mathbbm 1}(\{\mu>0\},\{\nu>0\}),
    \]
\end{theorem}

\begin{proof}
    Given $i$ a solution of the dynamic Beckmann equations, let $\xi = \nabla H^T(i)$. By the divergence conditions we get that $t\mapsto \sum_{e\in E} i(e,t)$ is a decreasing function with $\sum_{e\in E} i(e,1) = \sum_{x\in N} \m(x) = 1$. Therefore, we must have that $i(t)\in [0,1]^E$, and from the hypotheses on $H$
    \[
    m \mathbbm 1_{E^T} \leq \xi \leq M\mathbbm 1_{E^T}.
    \]

    From $m \mathbbm 1_{E^T} \leq \xi$ and the construction of $T_0$
    \[
    mT_0 = m\operatorname{in-diam}_{\mathbbm 1_{E^T}}(\{i>0\}) \leq \operatorname{in-diam}_\xi(\{i>0\}).
    \]
    By Corollary \ref{cor:4}
    \[
    \operatorname{in-diam}_\xi(\{i>0\}) \leq \operatorname{diam}_\xi(\{\iota_0[\mu]>0\},\{\iota_\W[\nu]>0\}).
    \]
    From $\xi \leq M\mathbbm 1_{E^T}$
    \[
    \operatorname{diam}_\xi(\{\iota_0[\mu]>0\},\{\iota_\W[\nu]>0\}) \leq M\operatorname{diam}_{\mathbbm 1_{E^T}}(\{\iota_0[\mu]>0\},\{\iota_\W[\nu]>0\}).
    \]
    
    To finish the proof we will show that if $T \geq \operatorname{diam}_{\mathbbm 1}(G)$, then
    \[
    \operatorname{diam}_{\mathbbm 1_{E^T}}(\{\iota_0[\mu]>0\},\{\iota_\W[\nu]>0\}) = \operatorname{diam}_{\mathbbm 1}(\{\mu>0\},\{\nu>0\}).
    \]
    Let $f\colon \{\w\in \operatorname{Path}(G) \ | \ L_{\mathbbm 1}(\w)\leq T\}\to \{\w \in \operatorname{Path}(G^{T,\W}) \ | \ \w^- \in N\times\{0\}, \w^+\in N^\W\}$ such that for $\w = (x_0,\ldots,x_\ell)$
    \[
    f(\w) = ((x_0,0),\ldots,(x_\ell,\ell),(x_\ell,\W)).
    \]
    This map is actually a bijection and satisfies $L_{\mathbbm 1_{E^T}}(f(\w)) = L_{\mathbbm 1}(\w)$. Then it also establishes a bijection between $\operatorname{Geod}(G,\mathbbm 1)$ and $\operatorname{Geod}(G^{T,\W},\mathbbm 1_{E^T})$, from where we finish the last step of the proof.
\end{proof}

Now we would like to analyze when it is possible to extend a given solution of the dynamic constitutive relation by zero.

\begin{theorem}\label{thm:dyn_ext}
    Let $H\colon\R^{E}_{\geq0}\times\Z_{\geq 1} \to \R$ be a time dependent potential and let $H^T(i) = \sum_{t=1}^T H(i(t),t)$. Let $i \in \R^{E^{T,\W}}_{\geq 0}$ be such that for some $u\in \R^{N^{T,\W}}$
    \[\begin{cases}
        \min\{i,\nabla H^T(i)-Du\} = 0 \text{ in } E^{T,\W},\\
        \operatorname{div}i=0 \text{ in } N\times\{1,\ldots,T\}.
    \end{cases}\]
    Let $j=i\mathbbm 1_{E^{T,\W}} \in \R^{E^{T+1,\W}}_{\geq 0}$ be the extension of $i$ by zero, and $\xi = \nabla H^{T+1}(j)$. Then, there exists $v\in \R^{N^{T+1,\W}}$ such that
    \[\min\{j,\nabla H^{T+1}(j)-Dv\} = 0 \text{ in } E^{T+1,\W}\]
    if and only if for every $e_1=(x_1,y_1),\ldots,e_k=(x_k,y_k),e_{k+1}=e_1\in E$ one has that
    \[
    \sum_{j=1}^k \p_{e_j}H(0,T+1) + d_{\operatorname{Sym}_{\{i>0\}}[\xi]}((y_j,\W),(x_{j+1},T))\geq 0.
    \]
\end{theorem}

\begin{proof}
    By Corollary \ref{cor:ext}, we have to check that $\operatorname{Sym}_{\{j>0\}}[\xi]$ has the non-negative loop condition. We notice first that $\{j>0\} = \{i>0\} \ss E^{T,\W}$, hence for any loop $\w$ contained in $\operatorname{Sym}_{\{j>0\}}(G^{T,\W})$ we obtain that $L_{\operatorname{Sym}_{\{j>0\}}[\xi]}(\w) \geq 0$, because $i$ is by hypothesis a solution to the dynamic Beckmann equations. So, we must focus on loops that are not contained in $\operatorname{Sym}_{\{j>0\}}(G^{T,\W})$.

    For every loop $\w \in \operatorname{Path}(\operatorname{Sym}_{\{j>0\}}(G^{T+1,\W}))$ that is not contained in $\operatorname{Sym}_{\{j>0\}}(G^{T,\W})$, let $e_1=(x_1,y_1),\ldots,e_k=(x_k,y_k) \in E$ such that each one of the edges $(e_1,T+1),\ldots,(e_k,T+1)$ appears in $\w$ in this given order.
    
    Let us assume without loss of generality that $\w^- = (x_1,T)$ and $\w$ is the concatenation of $\w_1,\bar\w_1,\ldots,\w_k,\bar\w_k$ where
    \begin{align*}
        \w_j = ((x_j,T),(y_j,T+1),(y_j,\W)),\qquad \bar \w_j \in \operatorname{Path}_{(y_j,\W),(x_{j+1},T)}(\operatorname{Sym}_{\{j>0\}}G^{T,\W}), \qquad x_{k+1}=x_1.
    \end{align*}
    
    The non-negative loop condition then means that
    \[
    L_{\operatorname{Sym}_{\{j>0\}}}(\w) = \sum_{j=1}^k \1L_{\operatorname{Sym}_{\{j>0\}}}(\w_j) + L_{\operatorname{Sym}_{\{j>0\}}}(\bar \w_j) \2 \geq 0.
    \]
    The first term in the sum is just $L_{\operatorname{Sym}_{\{j>0\}}}(\w_j)  = \p_{e_j}H(0,T+1)$, meanwhile that for the second the lowest possible value is given by the distance
    \[
    L_{\operatorname{Sym}_{\{j>0\}}}(\bar \w_j) \geq d_{\operatorname{Sym}_{\{i>0\}}[\xi]}((y_j,\W),(x_{j+1},T)).
    \]
    This shows that the criteria stated in the theorem is true.
\end{proof}

The previous theorem may be unpractical as we have to compute the symmetrizations and distances. The following lemma gives an easier positive criteria for the extension in terms of a multiplier for the solution.

\begin{corollary}\label{cor:5}
    Let $H\colon\R^{E}_{\geq0}\times\Z_{\geq 1} \to \R$ be a time dependent potential and let $H^T(i) = \sum_{t=1}^T H(i(t),t)$. Let $(i,u) \in \R^{E^{T}}_{\geq 0}\times \R^{N^{T,\W}}$ be such that 
    \[\begin{cases}
        \min\{i,\nabla H^T(i)-Du\} = 0 \text{ in } E^{T,\W},\\
        \operatorname{div}i=0 \text{ in } N\times\{1,\ldots,T\}.
    \end{cases}\]
    Let $j=i\mathbbm 1_{E^{T,\W}} \in \R^{E^{T+1,\W}}_{\geq 0}$ be the extension of $i$ by zero. Then there exists $v\in \R^{N^{T+1,\W}}$ such that
    \[\min\{j,\nabla H^{T+1}(j)-Dv\} = 0 \text{ in } E^{T+1,\W}\]
    if for every $e=(x,y) \in E$
    \[
    \p_{e} H(0,T+1) + u(x,T) - u(y,\W)\geq 0.
    \] 
\end{corollary}

\begin{proof}
    Given $\xi = \nabla H^{T+1}(j)$, we apply the 1-Lipschitz type estimate in Lemma \ref{lem:222}
    \[
    d_{\operatorname{Sym}_{\{i>0\}}[\xi]}((y,\W),(x,T))\geq u(x,T) - u(y,\W).
    \]
    Therefore, the hypothesis implies the assumptions in Theorem \ref{thm:dyn_ext}.
\end{proof}




\section{Analysis of some dynamic problems}\label{sec:dyn_ex}

In this final section, we address the dynamic Beckmann problem for one of the simplest graphs with non-local interactions, specifically the intersection of two roads, as outlined in the introduction and illustrated in Figure \ref{fig:2roads}. As a preliminary exercise, we will analyze the case of a single independent road in the first part of this section.

\subsection{One road}\label{sec:dyn_ex0}

Consider the very simple graph $G = (N,E)$ with $N=\{a,b\}$ and $E=\{(a,a),(a,b)\}$, illustrated by the Figure \ref{fig:intersection1}. The transport plan consists on sending one unit of mass from the node $a$ to the node $b$, so we fix $\Gamma = \Pi(\mu,\nu)$ where
\begin{align}\label{eq:10}
\m(a) = \nu(b)=1, \qquad \m(b)=\nu(a)=0.
\end{align}

\begin{figure}
    \centering
\begin{tikzpicture}[every loop/.style={}]
  \node[circle, draw] (2) at (2,0) {$b$};
  \node[circle, draw] (4) at (-2,0) {$a$};

  \draw[-{Latex[length=3mm]}] (4) -- (2);

  \path[-{Latex[length=3mm]}] (4) edge [loop left] node {} ();
\end{tikzpicture}
    \caption{A graph modelling a simple road.}
    \label{fig:intersection1}
\end{figure}
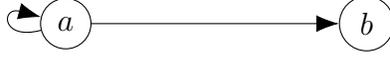



    

We may also consider the graph with a loop at the node $b$ or an edge from $b$ to $a$. However, it is not difficult to realize that in many cases of interest we can restrict the problem to equilibria that do not increase its cost by avoiding the loop $(b,b)$ or the edge $(b,a)$.

Let $i\colon E^{T,\W}\to [0,\8)$ such that $\div i = \iota_0[\mu]-\iota_\W[\nu]$. Our first observation is that the values of $i$ over the edges of the form $(a,a,t)$ determine the rest of the values of $i$. Let
\[
j(t) := i(a,a,t).
\]

If we declare $j(0)=1$ and $j(T)=0$, then we get that the divergence equations are equivalent to setting for $t\geq 1$
\[
i(a,b,t) = i(b,t,\W)= j(t-1)-j(t) =: -D^-j(t), \qquad i(a,t,\W)=0.
\]
Notice that $j$ must be non-increasing.


Let $H=H(i_{aa},i_{ab},t)\colon\R^{E}_{\geq0}\times\Z_{\geq 1} \to \R$ be a time dependent potential. We then construct the potential $H^T(i) = \sum_{t = 1}^T H(i(t),t)$. Due to the previous considerations, we obtain that the Beckmann optimization problem is equivalent to
\[
    \min \3 \sum_{t=1}^T H(j(t),-D^-j(t),t) \ | \ j \geq 0, D^-j \leq 0, j(0)=1, j(T)=0\4.
\]

\subsubsection{The obstacle problem}

In this section we show that the dynamic Beckmann problem is equivalent to an obstacle problem for $j$ under mild assumptions on the potential.

\begin{definition}
    Let $L=L(j,\d j,t)\colon\R_{\geq0}\times\R\times \Z_{\geq 1}\to \R$. Given $T\geq1$, and $j_0,j_T\geq 0$, the obstacle problem consists on computing
    \[
    \operatorname{OP}^T(j_0,j_T,L) := \argmin\3\sum_{t=1}^T L(j(t),D^-j(t),t) \ | \ j \geq 0, j(0)=j_0, j(T)=j_T\4.
    \]
\end{definition}

An immediate observation is that if we let $L(j,\d j,t) = H(j,|\d j|,t)$ then
\begin{align*}
&\min \3 \sum_{t=1}^T L(j(t),D^-j(t),t) \ | \ j \geq 0, j(0)=1, j(T)=0\4\\
\leq &\min \3 \sum_{t=1}^T H(j(t),-D^-j(t),t) \ | \ j \geq 0, D^-j \leq 0, j(0)=1, j(T)=0\4.
\end{align*}
The following results gives sufficient conditions for the equivalence of these two problems.

\begin{lemma}\label{lem:obs}
    Let $G=(N,E)$ with $N=\{a,b\}$ and $E=\{(a,a),(a,b)\}$. Let $\m,\nu\colon N\to[0,1]$ as defined in \eqref{eq:10}. Consider $H=H(i_{aa},i_{ab},t)\colon\R^{E}_{\geq0}\times\Z_{\geq 1} \to \R$ be a time dependent potential such that for each $t\geq 1$ and $i,h\in \R^E_{\geq0}$ with $h_{ab}>0$, it holds that $H(i+h)>H(i)$. Then, for $H^T(i) = \sum_{t=1}^T H(i(t),t)$ and $L(j,\d j,t) = H(j,|\d j|,t)$, we have
    \begin{align*}
    \{j = i(a,a,\cdot) \ | \ i \in \operatorname{BP}^T(G,\mu,\nu,H^T) \} = \operatorname{OP}^T(1,0,L).
    \end{align*}
\end{lemma}

\begin{proof}
    It suffices to show that if $j^*\in\operatorname{OP}^T(1,0,L)$, then $D^-j \leq 0$. Assume then by contradiction that $t_0 \in \{1,\ldots,T\}$ is such that $D^-j^*(t_0) > 0$.

    We can not have $t_0=T$, otherwise $j^*(T-1) < j^*(T)=0$. Let us now assume without loss of generality that $D^-j^*(t_0+1)\leq 0$. In other words, $t_0$ is a strict local maximum of $j^*$.
    
    For $\e = j^*(t_0)/2>0$, we get that $j := j^* -\e \mathbbm 1_{t_0}$ is an admissible competitor for the obstacle problem, so that
    \begin{align*}
    &L(j^*(t_0),D^-j^*(t_0),t_0) + L(j^*(t_0+1),D^-j^*(t_0+1),t_0+1) \leq\\
    &L(j(t_0),D^-j(t_0),t_0) + L(j(t_0+1),D^-j(t_0+1),t_0+1).
    \end{align*}
    
    By the strict monotonicity of $H$ we get the contradiction
    \begin{align*}
        0 &> L(j(t_0),D^-j(t_0),t_0)-L(j^*(t_0),D^-j^*(t_0),t_0)\\
        &\geq L(j^*(t_0+1),D^-j^*(t_0+1),t_0+1)-L(j(t_0+1),D^-j(t_0+1),t_0+1)\\
        &> 0,
    \end{align*}
    so we conclude the proof.
\end{proof}

If each $L(\cdot,\cdot,t)$ is also differentiable, minimizers of the obstacle problem satisfy the following variational equation\footnote{$D^+ f(t) := f(t+1)-f(t)$}
\[
\min\{j, -D^+[\p_{\d j} L]+\p_jL\} = 0.
\]
If each $L(\cdot, \cdot, t)$ is in addition convex, we also obtain that the minimizers are exactly the solutions of the variational equation.

For example, if $H = \tfrac{1}{2}i_{ab}^2+i_{aa}$, then $L(j,\d j,t)=\tfrac{1}{2}|\d j|^2+j$ and one obtains the classical obstacle problem\footnote{$\D j(t) := D^+D^- j(t) = j(t+1)-2j(t)+j(t-1)$.}
\[
\min\{j,-\D j+1\}=0.
\]

In our case, $L(j,\d j,t) = H(j,|\d j|,t)$ may not be differentiable at $\{\d j=0\}$. However, under the hypotheses of Lemma \ref{lem:obs}, over $\{\d j=0\}$ we only need to consider variations of $L$ towards $\{\d j<0\}$ where
\[
\p_{\d j}^- L(j,0,t) := \lim_{\e\to 0^+}\frac{L(j,-\e,t)-L(j,0,t)}{-\e} = -\p_{ab} H(j,0,t).
\]
Whenever we write $\p_{\d j}L$ we should keep in mind that for $\d j=0$ this derivative is actually the Dini derivative from the left.

In Lemma \ref{lem:obs3} we show the equivalence between the variational equations for the obstacle problem and the constitutive relations for the dynamic Beckmann problem. First we have the following technical result, closely related with Lemma \ref{lem:obs}.

\begin{lemma}\label{lem:obs2}
    Let $L=L(j,\d j,t)\colon\R_{\geq0}\times\R\times\Z_{\geq 1} \to \R$ be such that $L(j,-\d j,t) = L(j,\d j,t)$, the restriction of $L(\cdot,\cdot,t)$ to the (closed) positive quadrant is differentiable, satisfies $\nabla L(\cdot,\cdot,t)\geq 0$, and moreover
    \[
        \p_j L + \p_{\d j}L >0 \text{ in } \{\d j>0\}.
    \]
    Then, for any $j\colon \{0,\ldots,T\}\to \R$ such that
    \[
        \min\{j,-D^+[\p_{\d j}L]+\p_{j}L\}=0, \qquad j(T)=0.
    \]
    we have that $D^-j\leq 0$ in $\{1,\ldots,T\}$.
\end{lemma}

\begin{proof}
    Assume by contradiction that $t_0 \in \{1,\ldots,T\}$ is such that $D^-j(t_0) > 0$.

    We can not have $t_0=T$, otherwise $j(T-1) < j(T)=0$. Let us now assume without loss of generality that $D^-j(t_0+1)\leq 0$. By evaluating the variational equation for the obstacle problem at $t=t_0$ we get that (note that $j(t_0)>0$)
    \[
    \p_{\d j} L(j(t_0+1),D^-j(t_0+1)) - \p_{\d j} L(j(t_0),D^-j(t_0)) = \p_j L(j(t_0),D^-j(t_0)).
    \]

    By the odd symmetry and the monotonicity hypotheses on $L$ we get that
    \[
    \p_{\d j} L(j(t_0+1),D^-j(t_0+1),t_0+1) = -\p_{\d j} L(j(t_0+1),-D^-j(t_0+1),t_0+1) \leq 0.
    \]
    This turns out to be a contradiction, because
    \begin{align*}
    0 &\geq \p_{\d j} L(j(t_0+1),D^-j(t_0+1),t_0+1)\\
    &= \p_j L(j(t_0),D^-j(t_0),t_0) + \p_{\d j} L(j(t_0),D^-j(t_0),t_0)\\
    &>0.
    \end{align*}
    So we conclude the proof.
\end{proof}

\begin{lemma}\label{lem:obs3}
    Let $G=(N,E)$ with $N=\{a,b\}$ and $E=\{(a,a),(a,b)\}$. Let
    \[
    H=H(i_{aa},i_{ab},t)\colon\R^{E}_{\geq0}\times\Z_{\geq 1} \to \R
    \]
    be a time dependent potential such that
    \[
    \p_{aa} H + \p_{ab} H > 0 \text{ in } \{i_{ab}>0\},
    \]
    and let $L(j,\d j,t) = H(j,|\d j|,t)$. Given $j\colon \{0,\ldots,T\}\to \R$ and $i \in \R^{E^{T,\W}}$ such that
    \[
    i(a,a,t)=j(t), \quad i(a,a,T)=j(T)=0, \quad i(a,b,t) = i(b,t,\W) = -D^-j(t), \quad i(a,t,\W) = 0.
    \]
    We have that $j$ satisfies
    \[
        \min\{j,-D^+[\p_{\d j}L]+\p_{j}L\}=0
    \]
    if and only if $i$ satisfies the constitutive relations for the dynamic Beckmann problem.
\end{lemma}

\begin{proof}
    By using the given relations between $H$ and $L$, as well as the relation between $i$ and $j$, we verify that the equations for the obstacle problem follow from the constitutive relations over the edges of the form $(a,a,t)$. Let us then assume that $j$ satisfies the equations for the obstacle problem. Our goal is to see that $i$ satisfies the constitutive relations by constructing an appropriated multiplier.
    
    To ease the notation let us denote $\xi = \nabla H^T(i)$ and
    \[
    \l(t) = \xi(a,b,t+1)-\xi(a,b,t)+\xi(a,a,t)\geq0,
    \]
    such that $\min\{j,\l\}=0$.
    
    Starting from $u(a,0)=0$, define $u\colon N^{T,\W}\to\R$ such that
    \[
    \begin{cases}
        u(a,t+1) = u(a,t) + \xi(a,a,t+1) - \l(t+1),\\
        u(b,t+1) = u(a,t) + \xi(a,b,t+1),\\
        u(b,0) = u(b,1),\\
        u(b,\W) = u(b,t),\\
        u(a,\W) = \min\{u(a,t) \ | \ t\in\{0,\ldots,T\}\}.
    \end{cases}
    \]
    To see that this function is well defined at $(b,\W)$ we need to check that $u(b,t)$ is independent of $t$. This follows because for $t\geq 1$
    \begin{align*}
        u(b,t+1) &= u(a,t) + \xi(a,b,t+1)\\
        &= u(a,t) +\xi(a,b,t) - \xi(a,a,t)+\l(t)\\
        &= u(a,t-1)+\xi(a,b,t)\\
        &= u(b,t).
    \end{align*}
    
    By Lemma \ref{lem:obs2} and the relation between $i$ and $j$, we get that $i\geq 0$ over every edge of $G^{T,\W}$. From the construction of $u$ we obtain that $\nabla H^T \geq Du$ also over every edge of $G^{T,\W}$. Moreover, we have the equality $\nabla H^T = Du$ over edges of the form $(a,b,t)$, $(b,t,\W)$, and $(a,t,\W)$. To conclude we notice that over edges of the form $(a,a,t)$ we have that if $i(a,a,t)=j(t)>0$, then from the obstacle equation $\l(t)=0$ and
    \[
    \nabla H^T(i(a,a,t),i(a,b,t),t) = \xi(a,a,t) = u(a,t)-u(a,t-1) = Du(a,a,t).
    \]
    So we also have over these edges that $\min\{i,\nabla H^T-Du\}=0$.
\end{proof}

Given a solution $j\colon\{0,\ldots,T\}\to\R$ such that
\[
    \min\{j,-D^+[\p_{\d j}L]+\p_{j}L\}=0, \qquad j(T)=0,
\]
we have that the extension by zero given by $J = j\mathbbm 1_{\{t\leq T\}}$, solves similar equations over $\{1,\ldots, T\}$ if and only if
\[
-\p_{\d j} L(0,0,T+1) \geq \p_{j}L(0,-D^-j(T),T) - \p_{\d j}L(0,-D^-j(T),T).
\]

Due to the equivalence given by Lemma \ref{lem:obs3}, we obtain the following result stating the possibility of extending by zero a solution of the dynamic Beckmann equations. Notice that, when we apply the construction of the multiplier $u$ in the proof of Lemma \ref{lem:obs3}, to the Corollary \ref{cor:5}, we get exactly the same criteria.

\begin{corollary}
    Under the assumptions of Lemma \ref{lem:obs3}, we have that $I = i\mathbbm 1_{E^T}$ is a solution to the dynamic Beckmann equations in $G^{T+1}$ if and only if
    \[
    \p_{ab} H(0,0,T+1) \geq \p_{ab}H(0,i(a,b,T),T) - \p_{aa}H(0,i(a,b,T),T).
    \]
\end{corollary}

\subsubsection{Quadratic potential} 

Consider for $\a>0$ and $\b,\e\geq0$ the stationary potential
\[
H_{\a,\b,\e}(i) := \frac{\a}{2} i_{ab}^2 + \frac{\b}{2} i_{aa}^2 + \e i_{aa}.
\]
For $T\geq 1$, we define then the extended potential $H^T_{\a,\b,\e}\colon \R^{E^T}_{\geq0}\to\R$ such that
\[
H^T_{\a,\b,\e}(i) := \sum_{t=1}^T H_{\a,\b,\e}(i(t)).
\]
Notice that $H_{\a,\b,\e}$ is strictly increasing because $\a>0$. The case $\a=0$ is a trivial one, a solution would be to send all the mass through the edge $(a,b)$ at the first opportunity and without cost.

We could have considered as well a linear term in $i_{ab}$. However, for the given divergence constrain set by \eqref{eq:10}, this terms ends up adding to a constant term in the extended potential $H^T$.

We also notice that it suffices to analyze the problem for $\a=1$ to understand all the other possible cases. This due to the identity $H_{\a,\b,\e} = \a H_{1,\b/\a,\e/\a}$.

Assume from now on $\a=1$. By Lemma \ref{lem:obs}, we can compute $j(t)=i(a,a,t)$ from the obstacle problem
\[
\begin{cases}
    \min\{j,-\D j + \b j + \e\} = 0 \text{ in } \{1,\ldots,(T-1)\},\\
    j(0)=1,\\
    j(T)=0.
\end{cases}
\]

Let $\r := 1+\b/2$ and $r := \r + \sqrt{\r^2-1}\geq 1$ be the largest root of the characteristic polynomial $x^2-2\r x+1$. Note that the other root is $r^{-1}\leq 1$, and we have a double root if and only if $\b=0$. For $T\geq 1$, we get that the solution of the free problem
\[
\begin{cases}
\D J - \b J - \e = 0 \text{ in } \{1,\ldots,(T-1)\},\\
J(0)=1,\\
J(T)=0,
\end{cases}
\]
is given by
\begin{align}\label{eq:J}
J_T(t) &:= \begin{cases}
    \11+\frac{\e}{\b}\2\frac{r^{T-t}-r^{t-T}}{r^T-r^{-T}} + \frac{\e}{\b}\frac{r^t-r^{-t}}{r^T-r^{-T}} - \frac{\e}{\b} \text{ if } \b>0,\\
    1 - \11+\e\frac{T^2}{2}\2\frac{t}{T} + \e\frac{t^2}{2} \text{ if } \b=0.
\end{cases}
\end{align}

By the maximum principle, or by just checking the corresponding signs by hand, we can verify that $J_T$ also satisfies the obstacle problem if $\e=0$. 

Let us now focus on the case $\e>0$. Keep in mind that, due to the Theorem \ref{thm:fin_prop2}, we know that any solution of the dynamic Beckmann problem has a finite support, independent of $T$.

\begin{lemma}\label{lem:6}
    Let $\e>0$ and $J_T$ as in \eqref{eq:J}. Then there exists $T_0 = T_0(\b,\e)$ such that $J_T(t) \geq 0$ for all $t\in\{0,\ldots,T\}$ if and only if $T\leq T_0$.
\end{lemma}

\begin{proof}
    Let 
    \[
    T_0 := \min\{T\geq 1\ | \ J_{T+1}(T)<0\}
    \]
    We will show that $T_0$ is well defined and $J_T(t) \geq 0$ for all $t\in\{1,\ldots,T\}$ if and only if $T\leq T_0$.
    
    We can verify by a computation that
    \[
    \lim_{T\to\8} J_T(T-1) = \begin{cases}
    (\e/\b)(r^{-1}-1) \text{ if } \b>0,\\
    -\8 \text{ if }\b=0.
    \end{cases}
    \]
    Given that $r>1$ if $\b>0$, we get that for $T$ sufficiently large $J_T(T-1) < 0$, hence $T_0$ is a well defined finite number.

    Let us show by induction that $J_{T+1}(T)<0$ for any $T\geq T_0$. Assume that $T$ is such that $J_{T+1}(T)<0$, let us show then that also $J_{T+2}(T+1)<0$. Let $J:[0,T+2]\cap \Z\to \R$ such that $J(t) = J_{T+1}(t)$ for $t\leq T+1$, and $J(T+2)=0$. Then
    \[
    \D J - \b J - \e \leq 0 \text{ in } \{1,\ldots,T+1\}.
    \]
    Indeed, the equality holds for $t\in\{1,\ldots,T\}$ by construction. For $t=T+1$ we have that
    \[
    \D J(T+1) - \b J(T+1) - \e = J_{T+1}(T) - \e < 0.
    \]
    By the maximum principle (Lemma \ref{lem:mp}) we then have that $J \geq J_{T+2}$ in $\{1,\ldots,(T+1)\}$. Moreover, the inequality has to be strict at $t=T+1$, otherwise we get that $J_{T+1}=J=J_{T+2}$ in $\{1,\ldots,(T+1)\}$, due to the strong maximum principle. This contradiction implies that $0 = J(T+1) > J_{T+2}(T+1)$.

    Let us now show that for any $T\leq T_0$ one has that $J_T \geq 0$ in $\{1,\ldots,T\}$. We already know by construction that $J_T(T-1) \geq 0$. If we assume by contradiction that $J_T(t_0)<0$ for some $t_0 \in \{1,\ldots,T-2\}$, we would get a contradiction with the maximum principle over the interval $\{t_0,\ldots,T\}$.
\end{proof}


\begin{corollary}
    Let $\e>0$, $J_T$ as in \eqref{eq:J}, and $T_0$ as in Lemma \ref{lem:6}. Then the solution $i \in \operatorname{BP}^T(G,\mu,\nu,H^T_{1,\b,\e})$ is given by
    \begin{align*}
    i(a,a,t) &= J_{\min\{T,T_0\}}(t)\mathbbm 1_{\{t\leq \min\{T,T_0\}\}},\qquad    i(a,b,t) = i(a,a,t-1) - i(a,a,t).
    \end{align*}
\end{corollary}

Another property that can be deduced from the maximum principle is the monotonicity of $T_0$ with respect to $\e$ and $\b$. Let us denote $L_{\b,\e}(j,\d j) := H_{1,\b,\e}(j,-\d j)$. For any $0\leq \b_1\leq\b_2$, $0\leq \e_1\leq \e_2$, we have that if $j \in \operatorname{OP}^T(1,0,L_{\b_1,\e_1})$, then $j$ is also a super-solution of the obstacle problem with respect to the parameters $\b_2$ and $\e_2$
\[
\min\{j,-\D j + \b_1 j + \e_1\}=0 \qquad\Rightarrow\qquad \min\{j,-\D j + \b_2 j + \e_2\} \geq 0.
\]
The solution of the obstacle problem with respect to $\b_2$ and $\e_2$, and with the same boundary values as $j$, must therefore be smaller than or equal to $j$. In consequence,
\[
T_0(\b_2,\e_2) \leq T_0(\b_1,\e_1).
\]

In order to further analyze the value $T_0$ let us consider the function
\[
f_\b(x) := \frac{r^{x+1}-r^{-(x+1)}-r^x+r^{-x}}{r-r^{-1}}.
\]
It arises when we solve for $T$ in $J_{T+1}(T)=0$ such that
\[
T_0(\b,\e) = \lfloor f_\b^{-1}(1+\b/\e)\rfloor +1.
\]


For $x\gg 1$ we get that $f_\b(x) \sim Cr^x$. Then for $\e\ll \b$ we get that
\[
T_0(\b,\e) \sim \log_r(1+\b/\e).
\]

\subsection{The intersection of two roads}\label{sec:dyn_ex2}

Consider the graph $G = (N,E)$ with $N=N_1\cup N_2$, $E=E_1\cup E_2$, such that for $k\in \{1,2\}$
\[
N_k := \{a_k,b_k\}, \qquad E_k := \{(a_k,a_k),(a_k,b_k)\}.
\]
This graph was illustrated in the Figure \ref{fig:intersection}.

We would like to send $m_1 \in[0,1]$ units of mass from $a_1$ to $b_1$, and $m_2:=1-m_1$ units of mass from $a_2$ to $b_2$. So we fix $\Gamma = \Pi(\mu,\nu)$ with
\[
\m(a_k) = \nu(b_k) = m_k, \qquad \mu(b_k) = \nu(a_k) = 0.
\]
Even thought the graph has two disjoint components, the interesting feature of this model is that the cost should reflect some interaction between the edges $(a_1,b_1)$ and $(a_2,b_2)$.

As in the previous section, we obtain that under the divergence constrains, the function given by
\[
j_k(t) := i(a_k,a_k,t)
\]
determines the rest of the edge flow together with the boundary conditions $j_k(0)=m_k$ and $j_k(T)=0$. As before, we set for $t\geq 1$
\[
i(a_k,b_k,t) = i(b_k,t,\W) = -D^-j_k(t), \qquad i(a_k,t,\W)=0.
\]

Let $H = H(i_{a_1a_1},i_{a_2a_2},i_{a_1b_1},i_{a_2b_2})\colon \R^E_{\geq0}\to \R$ be a stationary potential and let $H^T(i) = \sum_{t=1}^T H(i(t))$. We get that the Beckmann problem can then be written in terms of $j_1$ and $j_2$ as
\begin{align}
    \label{eq:opt_2roads}
    \min\3\sum_{t=1}^T H(j_1(t),j_2(t),-D^-j_1(t),-D^-j_2(t)) \ | \ D^\pm j_k\leq 0, j_k(0)=m_k, j_k(T)=0\4.
\end{align}
Notice that the restriction $j_k\geq0$ is a consequence of $j_k(T)=0$ and $D^+ j_k\leq 0$.

In order to present the variational equation in this case, we let
\[
L(j_1,j_2,D^-j_1,D^-j_2) := H(j_1,j_2,-D^-j_1,-D^-j_2).
\]
We will also use multipliers defined over the half integers in the interval $[0,T]$. In other words, $\a \colon \{1/2,3/2,\ldots,T-1/2\}\to \R$. We denote gradient of such function at $t\in\{1,\ldots,(T-1)\}$ as
\[
D\a(t) := \a(t+1/2) - \a(t-1/2).
\]
Similarly, for $j\colon \{0,\ldots,T\}$ we denote
\[
Dj(t+1/2) := D^-j(t+1) = D^+j(t) = j(t+1)-j(t).
\]

The optimization problem in \eqref{eq:opt_2roads} then leads to the boundary value problem
\begin{align}\label{eq:eq_2roads}
\begin{cases}
    -D^+[\p_{\d j_k}L]+\p_{j_k}L = D\a_k \text{ for } t\in\{1,\ldots,(T-1)\},\\
    \min\{-Dj_k,\a_k\}=0 \text{ for } t\in\{1/2,\ldots,(T-1/2)\},\\
    j_k(0) = m_k\mathbbm 1_0 \text{ for } t\in\{0,T\}.
\end{cases}
\end{align}
Indeed, the Lagrangian is given by
\begin{align*}
&L^T(j_1,j_2;u_1,u_2)\\
&:= \sum_{t=1}^T L(j_1(t),j_2(t),D^-j_1(t),D^-j_2(t)) + \a_1 \cdot Dj_1 + \a_2 \cdot Dj_2\\
&= \sum_{t=1}^T L(j_1(t),j_2(t),D^-j_1(t),D^-j_2(t)) - D\a_1 \cdot j_1 - D\a_2 \cdot j_2 + C
\end{align*}
The term $C$ depends on $\a_1$, $\a_2$, and the boundary values, and is independent of $j_1(t)$ and $j_2(t)$ for $t\in\{1,\ldots,(T-1)\}$. Then the Karush-Kuhn-Tucker conditions give us the equations in \eqref{eq:eq_2roads}. Notice that the first set of equations in \eqref{eq:eq_2roads} give $2(T-1)$ scalar equations, meanwhile the second one give $2T$ scalar equations.

To illustrate a concrete example, we now assume that $H$ follows this structure
\[
H = H_1(i_{a_1a_1},i_{a_1b_1}) + H_2(i_{a_2a_2},i_{a_2b_2}) + I(i_{a_1b_1},i_{a_2b_2}).
\]
In this form, the interaction between the two components of $G$ is restricted to the edges $i_{a_1b_1}$ and $i_{a_2b_2}$. In the following section, we further assume that $H$ is a quadratic function, enabling us to perform numerical computations.

\subsubsection{Quadratic potential}

Given $\e,\gamma\geq 0$, consider the following expression for $H$
\[
H(i_{a_1a_1},i_{a_2a_2},i_{a_1b_1},i_{a_2b_2}) = \1\frac{1}{2}i_{a_1b_1}^2 + \e i_{a_1a_1}\2 + \1\frac{1}{2}i_{a_2b_2}^2 + \e i_{a_2a_2}\2 + \gamma i_{a_1b_1}i_{a_2b_2}.
\]
This formulation implies that the cost of keeping mass on $a_k$ is given by $\e$. The cost of sending $i$ units of mass across the edge $a_kb_k$ is $i+\gamma i'$, where $i'$ represents the amount of mass being transferred over the crossing edge.

In terms of $j_1$ and $j_2$
\[
L(j_1(t),j_2(t),D^-j_1(t),D^-j_1(t)) = \1\frac{1}{2}(D^-j_1)^2 + \e j_1\2 + \1\frac{1}{2}(D^-j_2)^2 + \e j_2\2 + \gamma D^-j_1D^-j_2.
\]
As usual
\[
L^T(j_1,j_2) = \sum_{t=1}^T L(j_1(t),j_2(t),D^-j_1(t),D^-j_1(t)).
\]

The corresponding equations take then the form
\begin{align}\label{eq:eq_2roads2}
\begin{cases}
    (-\D)j_1 + \gamma(-\D)j_2 + \e j_1 = D\a_1 \text{ for } t\in\{1,\ldots,(T-1)\},\\
    (-\D)j_2 + \gamma(-\D)j_1 + \e j_2 = D\a_2 \text{ for } t\in\{1,\ldots,(T-1)\},\\
    \min\{-Dj_k,\a_k\}=0 \text{ for } t\in\{1/2,\ldots,(T-1/2)\},\\
    j_k(0) = m_k\mathbbm 1_0 \text{ for } t\in\{0,T\}.
\end{cases}
\end{align}

By Corollary \ref{cor:beckeq_to_wareq}, any solution of \eqref{eq:eq_2roads2} allows for the construction of an efficient Wardrop equilibrium over $G^{T,\W}$ for the cost given by $g=\nabla H^T$ and the transport plan
\[
\Pi(\mu,\nu)^{T,\W} = \{m_1\mathbbm 1_{(a_1,0),(b_1,\W)}+m_2\mathbbm 1_{(a_2,0),(b_2,\W)}\}.
\]
If $(j_1^*,j_2^*)$ satisfies \eqref{eq:eq_2roads2}, then $j_1^*$ minimizes $j_1\mapsto L^T(j_1,j_2^*)$ under the given constrains for $j_1$. This holds because, after fixing $j_2=j_2^*$, the functional becomes strictly convex in the variable $j_1$. A similar statement applies to $j_2^*$ with respect to the functional $j_2\mapsto L^T(j_1^*,j_2)$.

The problem is strictly convex in both variables if and only if $\gamma \in (0,1)$. This can be checked for example by computing the Hessian of $L$ or just completing the square
\[
\frac{1}{2}(D^-j_1)^2 + \frac{1}{2}(D^-j_2)^2 + \gamma D^-j_1D^-j_2 = \frac{1}{2}(D^-j_1+\gamma D^-j_2)^2 + \frac{1}{2}(1-\gamma^2)(D^-j_2)^2.
\]
Recall that in this case, the solution of \eqref{eq:eq_2roads2} is the unique minimizer of the dynamic Beckmann problem.

We use an iterative method to compute a solution of \eqref{eq:eq_2roads2}. Given $j_k^\pm \in \R$ for $k\in\{1,2\}$, let
\begin{align*}
&E(j_1,j_2) := K_1(j_1) + K_2(j_2) + I(j_1,j_2),\\
&K_k(j) := \frac{1}{2}(j-j_k^-)^2 + \frac{1}{2}(j-j_k^+)^2 + \e j,\\
&I(j_1,j_2) := \gamma(j_1-j_1^-)(j_2-j_2^-) + \gamma(j_1-j_1^+)(j_2-j_2^+).
\end{align*}
Let us assume that at any given step we have that $D^\pm j_k \leq 0$ and consider some $t_0 \in\{1,\ldots,(T-1)\}$. We update the values of $j_1(t_0)$ and $j_2(t_0)$ such that for $j_k^\pm := j_k(t_0\pm 1)$
\[
(j_1(t_0),j_2(t_0)) \in \argmin\{ E(j_1,j_2) \ | \ (j_1,j_2) \in [j_1^+,j_1^-]\times [j_2^+,j_2^-]\}.
\]
The constrains in this problem ensure that at subsequent steps we still have $D^\pm j_k \leq 0$. In the following lemma, we demonstrate that, once $t_0$ is chosen, this step is uniquely defined under certain assumptions.

\begin{lemma}\label{lem:2roads}
    Let $j_k^+\leq j_k^-$ for $k\in\{1,2\}$ such that $\ell_1 := (j_1^--j_1^+) < \ell_2 := (j_2^--j_2^+)$. For $\gamma \in [0,1)\cup(1,\8)$ and $\e>0$, there exists a unique minimizer of $E$ under the constrain $(j_1,j_2) \in[j_1^+,j_1^-]\times [j_2^+,j_2^-]$.
\end{lemma}

\begin{proof}
    We have that the feasible set is compact, hence we only need to focus on the uniqueness of the minimizer. If $\gamma\in[0,1)$ we have a strictly convex function being minimized over a convex set, so the claim is immediate. Let us assume from now on that $\gamma>1$.
    
    The function $E$ has a unique critical point, which is of saddle type. As a result, the minimizers must lie on the boundary of the rectangle $[j_1^+,j_1^-]\times [j_2^+,j_2^-]$. We consider now two cases based on the location of the critical point of $E$. This critical point, denoted by $(j_1^*,j_2^*)$, is computed as follows
    \[
    j_k^* := j_k^0 - \frac{\e}{2(\gamma+1)}, \qquad j_k^0 := \frac{j_k^++j_k^-}{2}.
    \]
    See Figure \ref{fig:local_min}.

    \begin{figure}
        \centering
        \includegraphics[width=0.5\linewidth]{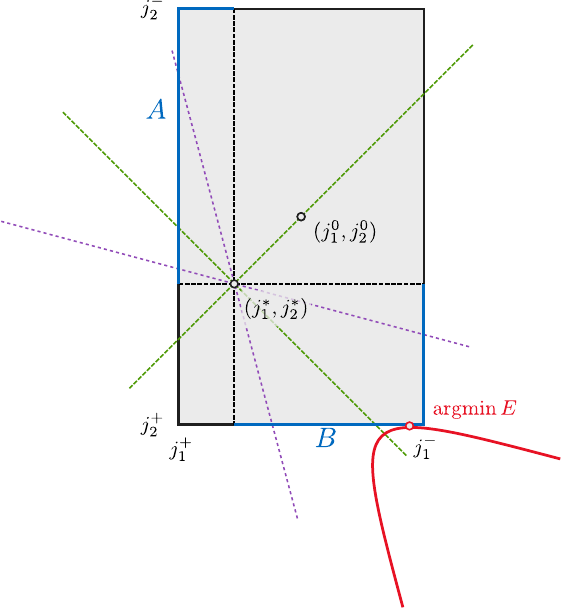}
        \caption{Illustration of the minimization of the quadratic function $E$ over the rectangle $[j_1^+,j_1^-]\times [j_2^+,j_2^-]$.}
        \label{fig:local_min}
    \end{figure}
    
    Notice that
    \[
    E(j_1^*+j_1,j_2^*+j_2) = \frac{1}{2}j_1^2+\frac{1}{2}j_2^2+\gamma j_1j_2 + E(j_1^*,j_2^*).
    \]
    The positive and negative eigenspaces of this quadratic form are parallel to the lines $j_1-j_2=0$ and $j_1+j_2=0$ respectively.

    \textbf{Case 1:} Assume that $(j_1^*,j_2^*) \in \R^2\sm ((j_1^+,j_1^-)\times (j_2^+,j_2^-))$. In this scenario, for any $j_1\in [j_1^+,j_1^-]$, the function $j_2\mapsto E(j_1,j_2)$ is increasing over the interval $[j_2^+,j_2^-]$, and the same also holds if we fix $j_2 \in [j_2^+,j_2^-]$ and vary $j_1 \in [j_1^+,j_1^-]$ instead. This implies that the bottom-left corner $(j_1^+,j_2^+)$ is the only minimizer in this case.

    \textbf{Case 2:} Assume that $(j_1^*,j_2^*) \in (j_1^+,j_1^-)\times (j_2^+,j_2^-)$. In this case we have that
    \begin{align*}
    &\emptyset \neq \{E<E(j_1^*,j_2^*)\} \cap \p([j_1^+,j_1^-]\times [j_2^+,j_2^-]) \ss A\cup B,\\
    &A:= \{j_1^+\}\times [j_2^*,j_2^-] \cup [j_1^+,j_1^*]\times \{j_2^+\},\\
    &B := [j_1^*,j_1^-]\times\{j_2^+\} \cup \{j_1^-\}\times [j_2^+,j_2^*]. 
    \end{align*}
    We will show first that there can not be minimizers on the upper-left set $A$.

    Notice that $E$ has an even symmetry over the line $j_1-j_2=j_1^*-j_2^*$
    \[
    E(j_2+j_1^*-j_2^*,j_1+j_2^*-j_1^*) = E(j_1,j_2).
    \]
    Using that $\ell_1<\ell_2$, we get that if $(j_1,j_2) \in \{j_1^+\}\times (j_2^0 - \ell_1/2,j_2^-)$, then $(j_2+j_1^*-j_2^*,j_1+j_2^*-j_1^*)$ is in the interior of the rectangle. As a consequence, minimizers can not intersect this segment which covers the left side of $A$.
    
    Let $Q(j_1,j_2) := E(j_1,j_2) - \e(j_1+j_2)$. Notice that $(j_1^0,j_2^0)$ is the critical point of $Q$ and
    \[
    Q(j_1^0+j_1,j_2^0+j_2) = \frac{1}{2}j_1^2+\frac{1}{2}j_2^2+\gamma j_1j_2 + Q(j_1^0,j_2^0).
    \]
    Then we have the following symmetry around $(j_1^0,j_2^0)$
    \begin{align*}
    E(j_1^0-j_1,j_2^0-j_2)
    &= Q(j_1^0-j_1,j_2^0-j_2) + \e((j_1^0-j_1)+(j_2^0-j_2))\\
    &= Q(j_1^0+j_1,j_2^0+j_2) + \e((j_1^0-j_1)+(j_2^0-j_2))\\
    &= E(j_1^0+j_1,j_2^0+j_2) - 2\e(j_1+j_2).
    \end{align*}
    The point $(j_1^0-j_1,j_2^0-j_2) \in \p([j_1^+,j_1^-]\times[j_2^+,j_2^-])$ if and only if $(j_1^0+j_1,j_2^0+j_2) \in \p([j_1^+,j_1^-]\times[j_2^+,j_2^-])$. So, by the previous identity, and using that $\e>0$, we get that the minimizers have to be contained over the intersection of $\p([j_1^+,j_1^-]\times[j_2^+,j_2^-])$ with the half space $\{j_1+j_2\leq j_1^0+j_2^0\}$. This implies that there are not minimizers on the top side of the rectangle.

    Let $m$ be the minimum of $E$ over $[j_1^+,j_1^-]\times[j_2^+,j_2^-]$. To conclude the proof we use that the set $\{(j_1,j_2)\in \R^2 \ | \ E(j_1,j_2) \leq m\}$ is bounded by two hyperbolas. Only one of them is inside the quadrant $\{(j_1,j_2) \in \R^2 \ | \ j_1\geq j_1^0, j_2\leq j_2^0\}$ and can touch $B$ at a unique point (by a convexity argument).
\end{proof}

\begin{remark}
    If $\ell_1=\ell_2$ and $\gamma>1$ we have two minimizers if the critical point of $E$ belongs to $(j_1^+,j_1^-)\times(j_2^+,j_2^-)$. This happens because of the symmetry around the line $j_1-j_2=j_1^*-j_2^*$.

    If $\e=0$ we have that the critical point of $E$ is the center of $[j_1^+,j_1^-]\times[j_2^+,j_2^-]$. Also in this case we find two symmetric minimizers around the line $j_1-j_2=j_1^*-j_2^*$.

    Finally, if $\gamma=1$ we get that the set of minimizers of $E$ over $\R^2$ is the line $j_1+j_2=j_1^*+j_2^*$. If this line intercepts $(j_1^+,j_1^-)\times (j_2^+,j_2^-)$ we obtain an infinite number of minimizers. Otherwise we only get one minimizer at the bottom-left corner.
\end{remark}

We apply Algorithm \ref{alg:2roads} to compute an approximate solution of \eqref{eq:eq_2roads2} for some given input parameters $T$, $m_1$, $m_2$, $\e$, $\gamma$, $j_1$, and $j_2$ with the boundary conditions $j_k = m_k\mathbbm 1_{\{0,T\}}$. The main step consists on minimizing $E$ for some given rectangle, referred to as the local minimization. According to Lemma \ref{lem:2roads} the result of this local minimization is unique, under certain assumptions.

When $\gamma<1$ the problem is strictly convex and easy to solve. For $\gamma>1$ the minimizer is restricted to a one dimensional set then it is can also be solved efficiently. To break the ties in the exceptional cases—specifically, when $\gamma = 1$, $\epsilon = 0$, or $\ell_1 = \ell_2$—we approximate the minimizer by considering solutions obtained under conditions where $\gamma > 1$, $\epsilon > 0$, and $\ell_1 < \ell_2$. In this way we guarantee a unique selection from the set:
\[
\argmin\{E(j_1,j_2) \ | \ (j_1,j_2) \in [j_1^+,j_1^-]\times[j_2^+,j_2^-]\}.
\]

\begin{algorithm}
\caption{Particular solutions of \eqref{eq:eq_2roads2}}\label{alg:2roads}
\begin{algorithmic}
\State \( H^T_{\text{prev}}, H^T_{\text{curr}} \gets \infty, 0 \)

\While{ \( H^T_{\text{prev}} \neq H^T_{\text{curr}} \)}
    \State \( H^T_{\text{prev}} \gets H^T_{\text{curr}} \)
    \For{ \( t = 1 \) to \( T-1 \) }
        \State \( j_1^+, j_1^-, j_2^+, j_2^- \gets j_1(t+1), j_1(t-1), j_2(t+1), j_2(t-1) \)
        \State \( j_1(t), j_2(t) \gets \argmin\{E(j_1,j_2) \ | \ (j_1,j_2) \in [j_1^+,j_1^-]\times[j_2^+,j_2^-]\}\)
    \EndFor
    \State \( H^T_{\text{curr}} \gets H^T(j_1,j_2) \)
\EndWhile
\end{algorithmic}
\end{algorithm}

A python implementation is given in \cite{2roads}

\begin{center}
    \href{https://github.com/hchanglara/2roads}{https://github.com/hchanglara/2roads}.
\end{center}

The Figure \ref{fig:2roads} shows 16 plots for $j_1$ and $j_2$ with $T=10$, $m_1=2$, $m_2=3$, $\e = 0.1$ and $\gamma = k\d \gamma$ for $k\in \{0,\ldots,15\}$ and $\d\gamma = 2/15 \sim 0.1333$. We compute first the case $\gamma=2$ starting with $j_k = m_k \mathbbm 1_{\{1,\ldots,(T-1)\}}$. In every subsequent case we decreased $\gamma$ by $\d\gamma$ and used the previous solution as initial value in the iteration.

We appreciate some interesting phenomena in these plots. The qualitative behavior of the solutions changes at $\gamma=1$, this is expected as the functional $H^T$ changes from being strictly convex for $\gamma \in(0,1)$, to being a quadratic function with negative eigenvalues for $\gamma>1$. 

We also notice that for $\gamma>1$, $\max\{Dj_1,Dj_2\}=0$ happens almost all the time. Indeed, in the proof of Lemma \ref{lem:2roads} we observed that for $\gamma>1$, the minimizer of $E$ must lie on the boundary of the rectangle $[j_1^+,j_1^-]\times[j_2^+,j_2^-]$. Specifically, it will be located on either the bottom or right side if $\ell_1<\ell_2$. Whenever we get the bottom side we have that $D^+j_2=0$, meanwhile on the right side we get $D^-j_1=0$. This behavior is also expected when $\gamma\to \8$, given that if $\max\{Dj_1,Dj_2\}<0$ then we get an interaction term $\gamma D^-j_1D^-j_2$ arbitrarily large. On the other hand, even for $T=2$ one may have that $\max\{Dj_1,Dj_2\}=0$ does not hold for every $t$, as it can be checked by hand as in the proof of the Lemma \ref{lem:2roads}.

\clearpage
\vspace*{\fill} 
\begin{figure}[h!]  
    \centering
    \includegraphics[width=\linewidth]{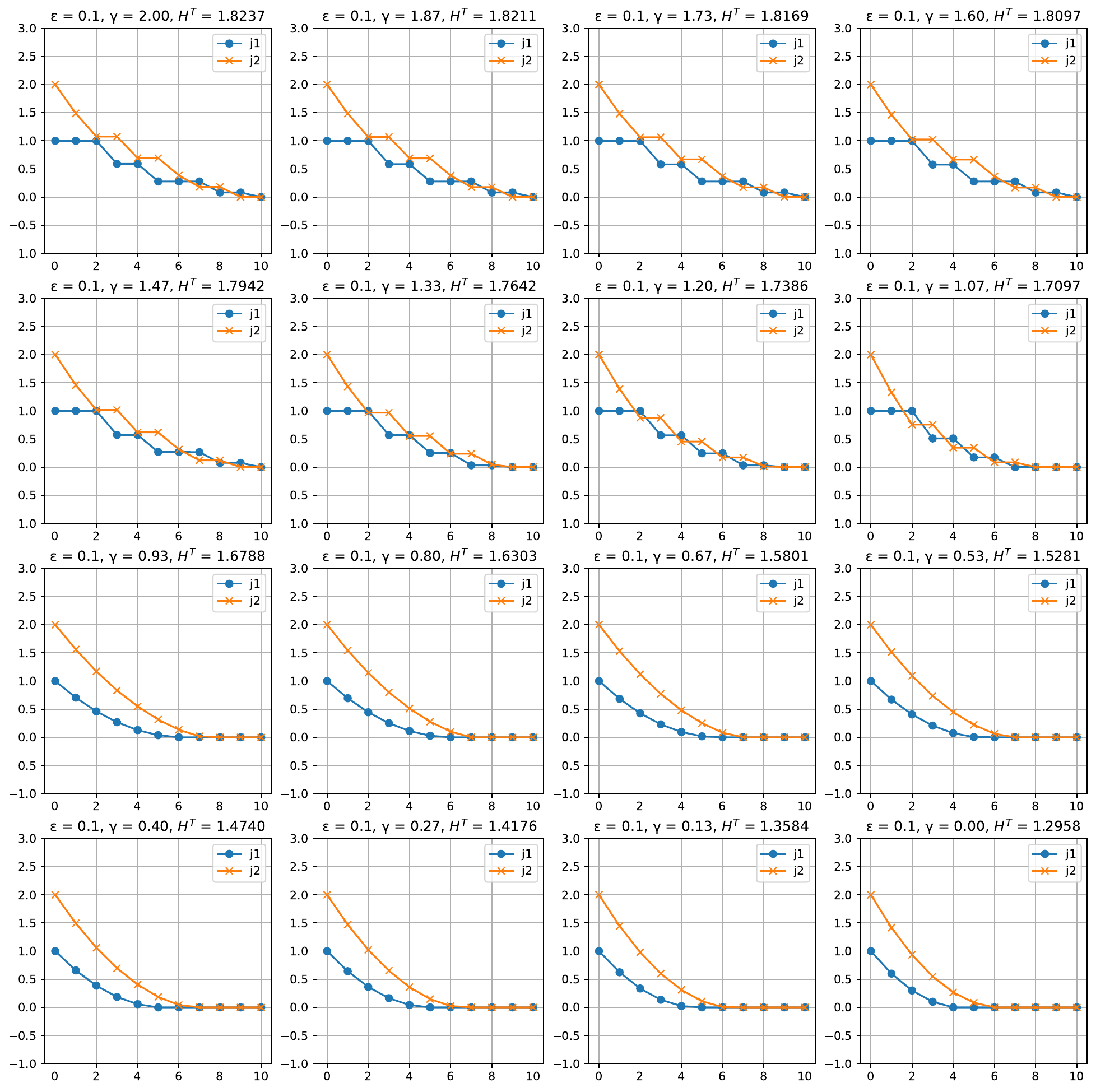}
    \caption{Solutions of the dynamic problem computed by an iterative method for the graph in Figure \ref{fig:intersection}, $T=10$, $m_1=2$, $m_2=3$, $\e=0.1$, and $\gamma$ decreasing uniformly from 2 to 0.}
    \label{fig:2roads}
\end{figure}
\vspace*{\fill} 
\clearpage

\section{Appendix}\label{sec:appendix}

\begin{lemma}[Integration by parts]\label{lem:5}
For every $u\in\R^N$ and $i\in\R^E$
\begin{align*}
\sum_{x\in N} u(x)\operatorname{div}i(x) = -\sum_{e\in E} i(e)Du(e).
\end{align*}
\end{lemma}

\begin{proof}
    Follows by a computation
    \begin{align*}
    \sum_{x\in N} u(x)\div i(x) &= \sum_{x\in N}\1\sum_{e^-=x} i(e)u(x) -\sum_{e^+=x} i(e)u(x)\2\\
    &= \sum_{x\in N}\sum_{e\in E} i(e)u(x)(\mathbbm 1_{e^-}(x)- \mathbbm 1_{e^+}(x))\\
    &= \sum_{e\in E}\sum_{x\in N} i(e)u(x)(\mathbbm 1_{e^-}(x)- \mathbbm 1_{e^+}(x))\\
    &= \sum_{e \in E} i(e)(u(e^-) -u(e^+))\\
    &= -\sum_{e\in E} i(e)Du(e).
    \end{align*}
\end{proof}

\begin{lemma}[Strong maximum principle]\label{lem:mp}
    Let $b\geq 1$ be a real number, and $t_0<t_1$ be two integers. If $u\colon\{t_0,\ldots,t_1\}\to \R$ is such that
    \[
    \begin{cases}
    u(t+1) - 2bu(t) + u(t-1) \geq 0 \text{ for }t\in \{(t_0+1),\ldots,(t_1-1)\},\\
    u(t) \leq 0 \text{ for } t\in\{t_0,t_1\}.
    \end{cases}
    \]
    then $u(t) \leq 0$ for $t\in \{t_0,\ldots,t_1\}$. Moreover, if $u(t)=0$ for some $t \in \{(t_0-1),\ldots,(t_1+1)\}$, then $u=0$ for every $t\in \{t_0,\ldots,t_1\}$.
\end{lemma}

\begin{proof}
    Assume by contradiction that $M:= \max_t u(t)>0$. Let $t^* := \min\{t>t_0 \ | \ u(t)=M\}$. This number is well defined unless $t_1=t_0+1$ in which case the theorem is void.

    Then we get the following contradiction when we evaluate the equation at $t^*$
    \begin{align*}
    0 &\leq u(t^*+1) - 2bu(t^*) + u(t^*-1)\\
    &= (u(t^*+1) - u(t^*)) + (u(t^*-1) - u(t^*)) - 2(b-1)u(t^*) < 0. 
    \end{align*}

    For the second part of the lemma let $t_2\in \{(t_0-1),\ldots,(t_1+1)\}$ such that $u(t_2) = 0$. Let us assume that $t_2>t_0+1$ and show that $u=0$ in $\{t_0,\ldots,t_2-1\}$. Assume by contradiction that $u<0$ somewhere in $\{t_0,\ldots,t_2-1\}$, then define $t^* := 1+\max\{t \in \{t_0+1,\ldots,t_2-1\}\ | \ u(t)<0\}$. By the same computation as in the first part we get a contradiction. The same reasoning over $\{t_2+1,\ldots,t_1\}$ allow us to finish the proof.
\end{proof}

\bibliographystyle{plainurl}
\bibliography{mybibliography}

\end{document}